\documentclass[11pt]{amsart}
\usepackage{amsmath,amsfonts,amsthm}
\usepackage{amssymb,latexsym}
\usepackage[colorlinks=true, linkcolor=blue,urlcolor=blue]{hyperref}
\usepackage{graphicx}
\usepackage[all]{xy}
\usepackage{enumitem}

\theoremstyle{plain}
\newtheorem{theorem}{Theorem}[section]
\newtheorem{lemma}[theorem]{Lemma}
\newtheorem{proposition}[theorem]{Proposition}
\newtheorem{corollary}[theorem]{Corollary}

\newtheorem{thmx}{Theorem}

\theoremstyle{definition}

\newtheorem{definition}[theorem]{Definition}
\newtheorem{example}[theorem]{Example}

\newtheorem{remark}[theorem]{Remark}

\newtheorem{notation}[theorem]{Notation}

\newtheorem{defx}[thmx]{Definition}

\theoremstyle{remark}
\numberwithin{equation}{section}

\newcounter{TmpEnumi}

\newcommand{\N}{\mathbb{N}}
\newcommand{\Z}{\mathbb{Z}}

\newcommand{\C}{\mathbb{C}}

\newcommand{\spec}{{\operatorname{sp}}}

\newcommand{\dist}{{\operatorname{dist}}}
\newcommand{\Ped}{{\operatorname{Ped}}}

\newcommand{\af}{\alpha}
\newcommand{\bt}{\beta}

\newcommand{\dt}{\delta}
\newcommand{\ep}{\varepsilon}
\newcommand{\zt}{\zeta}
\newcommand{\et}{\eta}

\newcommand{\ld}{\lambda}

\newcommand{\ph}{\varphi}
\newcommand{\ps}{\psi}
\newcommand{\rh}{\rho}
\newcommand{\om}{\omega}
\newcommand{\ta}{\tau}

\newcommand{\andeqn}{\qquad {\mbox{and}} \qquad}
\RequirePackage[normalem]{ulem} 
\RequirePackage{color}\definecolor{RED}{rgb}{1,0,0}\definecolor{BLUE}{rgb}{0,0,1} 

\begin{document}

\title{Simple tracially $\mathcal{Z}$-absorbing C*-algebras}



\author[M. Amini, N. Golestani,  S. Jamali, N. C. Phillips]{
  Massoud Amini, Nasser Golestani, Saeid Jamali,
  N. Christopher Phillips}


\address{{\textbf{M. Amini, N. Golestani, S. Jamali}} \newline
           Faculty of Mathematical Sciences, Tarbiat Modares University
\\
           Tehran 14115-134
\newline
              {\textbf{N. C. Phillips}} \newline
   Department of Mathematics, University of Oregon
\\
   Eugene OR 97403-1222}



\begin{abstract}
We define a notion of tracial $\mathcal{Z}$-absorption for simple
not necessarily unital C*-algebras, study it systematically, and prove 
its permanence properties.
This extends the notion defined
by Hirshberg and Orovitz for unital  C*-algebras.
The Razak-Jacelon algebra, simple nonelementary C*-algebras with tracial rank zero, 
and simple purely infinite C*-algebras are tracially $\mathcal{Z}$-absorbing.
We obtain the first purely infinite examples of 
   tracially $\mathcal{Z}$-absorbing C*-algebras
which are not $\mathcal{Z}$-absorbing.
We use techniques from 
reduced free products of von~Neumann algebras to
 construct these examples.
A stably finite example was given   by Z. Niu and Q. Wang in 2021.
We study the Cuntz semigroup of
a simple tracially $\mathcal{Z}$-absorbing C*-algebra
and prove that it
is almost unperforated and the algebra is weakly almost divisible.
\end{abstract}

\maketitle
\tableofcontents
\section{Introduction and Main Results}\label{sec_intro}
%

X.~Jiang and H.~Su \cite{JS} constructed
a unital separable simple infinite dimensional nuclear
C*-algebra $\mathcal{Z}$
with the same K-theoretic invariant as that of $\mathbb C$.
The Jiang-Su algebra, which is the stably finite analog
of the Cuntz algebra ${\mathcal{O_{\infty}}}$, has played a
central role in Elliott classification program for nuclear C*-algebras.
A unital C*-algebra $A$ is $\mathcal{Z}$-absorbing
if $A \cong A \otimes \mathcal{Z}$.
Certain $\mathcal{Z}$-absorbing C*-algebras
are classified by their ordered K-groups \cite{LN, W2}
and all classes of unital simple nuclear C*-algebras
for which the Elliott conjecture is confirmed
are already $\mathcal{Z}$-absorbing.
It was conjectured
 that the properties of strict
comparison, finite nuclear dimension, and $\mathcal{Z}$-absorption
are equivalent for unital separable
simple infinite-dimensional nuclear C*-algebras
(Toms-Winter conjecture) \cite{TW07}. It is now known 
that the last two conditions are equivalent \cite{CE}, 
\cite{CETWW}, \cite{Ti14}, \cite{W1} (building upon \cite{MS}).
We also know that $\mathcal{Z}$-absorption
implies strict comparison for unital simple exact
C*-algebras~\cite{Ro04}
and  that for
a unital separable simple infinite-dimensional nuclear C*-algebra
with finitely many extremal traces, strict
comparison is equivalent to $\mathcal{Z}$-absorption~\cite{MS}. 
Finally, we know that, when there is at least one trace on a unital 
separable simple infinite-dimensional nuclear C*-algebra, then 
$\mathcal{Z}$-absorption is equivalent to strict comparison plus 
the uniform $\Gamma$-property \cite{CETW}.

\medskip

Hirshberg and Orovitz introduced the notion of
tracial $\mathcal{Z}$-absorption for unital C*-algebras~\cite{HO13}.
(See Definition~\ref{deftzau} below.)
They show in \cite[Proposition 2.2 and Theorem 4.1]{HO13}
that for simple unital separable nuclear C*-algebras,
$\mathcal{Z}$-absorption
is equivalent to tracial $\mathcal{Z}$-absorption.
The main result of \cite{CLS2021} states this equivalence in the nonunital case.  
The main advantage of tracial $\mathcal{Z}$-absorption is that
it is easier to check in certain cases,
particularly for crossed products.
(See, e.g.,~\cite{Ke17}.)
This is due to the fact that tracial $\mathcal{Z}$-absorption
is defined
via a local property, whereas to verify $\mathcal{Z}$-absorption one needs to
verify  the existence of an isomorphism
which may not be easy.

\medskip

The aim of this paper is twofold.
First, there has recently been considerable interest
in the classification and structure of nonunital 
simple C*-algebras \cite{GngLin, GngLinII, GngLinIII}.
Nonunital tracial $\mathcal{Z}$-absorption
may be useful in applying the results of the classification
of nonunital C*-algebras
(or, rather, stably projectionless C*-algebras),
in particular, the classification of crossed products of simple
nonunital C*-algebras by actions with Rokhlin type properties.
This requires, however, further progress in the
recently started classification project
for simple nonunital C*-algebras.
(See e.g.\   \cite{EGLN17}.)
Second, tracial $\mathcal{Z}$-absorption
can substitute for $\mathcal{Z}$-absorption
when attempting to prove various regularity properties
for simple C*-algebras,
but holds more generally and is easier to verify.
There are therefore motivations to  study tracial 
$\mathcal{Z}$-absorption in its own right.

\medskip

In this paper we extend
the notion of tracial $\mathcal{Z}$-absorption
to the simple nonunital case and study it
systematically.

\begin{defx}\label{defx_tz}
We say that a simple C*-algebra $A$
is \emph{tracially $\mathcal{Z}$-absorbing} if
$A \ncong \mathbb{C}$
and for every $x, a \in A_{+}$ with $a \neq 0$,
every finite set $F \subseteq A$, every $\varepsilon > 0$,
and every $n \in \mathbb{N}$,
there is a c.p.c.~order zero map $\varphi \colon M_{n} \to A$
such that:
\begin{enumerate}
\item\label{defx_tz-it1}
$\bigl( x^{2} - x \varphi (1) x - \varepsilon \bigr)_{+} \precsim a$.
\item\label{defx_tz-it2}
$\| [\varphi (z), b] \| < \varepsilon$
for any $z \in M_{n}$ with $\| z \| \leq 1$ and any $b \in F$.
\setcounter{TmpEnumi}{\value{enumi}}
\end{enumerate}
\end{defx}
Part~\eqref{defx_tz-it1} contains our main idea of 
 ``nonunital tracial smallness."
If $A$ is unital, then we can take $x=1$, and it turns
out that this definition extends 
the definition of the unital case (\cite[Definition~2.1]{HO13});
see Remark~\ref{rmkdeftza}. 
 This definition motivated 
the second named author and Forough to introduce a 
suitable notion of weak tracial Rokhlin property for
finite group actions on simple nonunital C*-algebras.
See Definition~1.1 of \cite{FG17} and the discussion after that.
Also, see \cite[Definition~6.6]{FG17ar},
   \cite{SJ_thesis}, and \cite[Definition~4.4]{FLL21}
   in which Definition~\ref{defx_tz} is quoted from the 
   unpublished version of the current paper,
   showing that this idea of nonunital tracial smallness
   has been motivating.
This definition also appeared recently in \cite{CLS2021}
which has minor overlap in results with ours.
\medskip

We prove in Section~\ref{Sec_Finite} that if $A$ is finite then one can add the following
condition to Definition~\ref{defx_tz}:

\begin{enumerate}
\setcounter{enumi}{\value{TmpEnumi}}
\item\label{defx_tz-it3}
$\| \ph (1) a \ph (1) \| > 1 - \ep$.
\end{enumerate}
This is parallel to the definition
of the weak tracial Rokhlin property for finite group actions
\cite{GH20, FG17}, and is needed for the proof of a result in
\cite{AGJP17}
that the permutation action
on the minimal tensor product of finitely many copies
of a tracially $\mathcal{Z}$-absorbing C*-algebra
has the weak tracial Rokhlin property,
provided that the tensor product is finite.

Different characterizations of Definition~\ref{defx_tz} and its permanence
properties are given in Sections~\ref{sec_tz} and \ref{sec_per}.
In particular, it is preserved under inductive limits
and stable isomorphism, and passes to
hereditary subalgebras.

\medskip

The class of simple tracially $\mathcal{Z}$-absorbing C*-algebras
is large. 

\begin{thmx}\label{thmx_ex}
The following classes of C*-algebras are tracially $\mathcal{Z}$-absorbing:
\begin{enumerate}
\item\label{thmx_ex_it1}
simple nonelementary C*-algebras with tracial rank zero;

\item\label{thmx_ex_it2}
simple purely infinite C*-algebras;

\item\label{thmx_ex_it3}
simple $\mathcal{Z}$-absorbing C*-algebras.
\end{enumerate}
\end{thmx}

Part~\eqref{thmx_ex_it1} implies that many 
simple stably finite C*-algebras are
tracially $\mathcal{Z}$-absorbing. It is proved in
Lemma~\ref{P_8809_TAFtoTrZAbs} (unital case) and
Proposition~\ref{P_8Y03_NonUTAFtoTrZAbs} (nonunital
case). 
The special case of Part~\eqref{thmx_ex_it1} where $A$
is unital and separable follows from 
\cite[Proposition~6.9 and Theorem~4.11]{FLL21}
and \cite[Theorem~5.9]{FL21} in which the authors
use different methods from ours.
A similar result by Matui and Sato is that
every
  simple separable unital nuclear infinite-dimensional 
  C*-algebra with tracial rank zero  
   $\mathcal{Z}$-absorbing
  \cite[Theorem~5.4]{MS}.
 (The nonunital version of this result follows from
 Part~\eqref{thmx_ex_it1} and \cite[Theorem~A]{CLS2021}.)
An example of a simple separable unital C*-algebra with tracial
rank zero (hence with stable rank one and real rank zero
\cite{Ln1, LnBook})
which is not $\mathcal{Z}$-absorbing is constructed
in \cite{NiuWng}. It is nonnuclear and exact. As far as we know, this
stably finite example and purely infinite examples
of Theorem~\ref{thmx_pi} (below) are
the first known tracially $\mathcal{Z}$-absorbing C*-algebras
which are not $\mathcal{Z}$-absorbing.

Part~\eqref{thmx_ex_it2} follows directly  form Definition~\ref{defx_tz}
and characterizations of pure infiniteness
\cite[Proposition~5.4]{KR00} (see Example~\ref{exapi}).
We obtain purely infinite C*-algebras which are not
$\mathcal{Z}$-absorbing:

\begin{thmx}\label{thmx_pi}
There are simple separable unital purely
infinite (hence tracially $\mathcal{Z}$-absorbing) C*-algebras 
which are not $\mathcal{Z}$-absorbing.
\end{thmx}

The proof of this result is based on techniques from
reduced free products of von~Neumann algebras 
\cite{DkmRdm, Brnt}. See Proposition~\ref{P_8624_NotZStab}
and Example~\ref{E_6825_PINotZSt}.
As a concrete example,
the reduced C*-algebra free product
\[
(M_2 \otimes M_2) \star_{\mathrm{r}} C ([0, 1])
\]
 is
simple separable unital purely
infinite (hence tracially $\mathcal{Z}$-absorbing) but not $\mathcal{Z}$-absorbing
(Example~\ref{E_8702_PINotZStable}).
By \cite[Theorem 4.1]{HO13}, this example is nonnuclear.

\medskip
We remark that an interesting result of \cite{FLL21} says that
every simple separable  tracially approximately divisible C*-algebra
is  purely infinite or  has stable rank one \cite[Corollary~6.5]{FLL21}.
Since tracial approximate divisibility is equivalent to
tracial $\mathcal{Z}$-absorption for simple separable
C*-algebras \cite[Theorem~4.11]{FLL21}, we see that
every simple separable  tracially  $\mathcal{Z}$-absorbing
C*-algebra
is  purely infinite or  has stable rank one

\medskip

Part~\eqref{thmx_ex_it3} of Theorem~\ref{thmx_ex} follows from tracial $\mathcal{Z}$-stability
of $\mathcal{Z}$ \cite[proposition~2.2]{HO13} and that tracial 
$\mathcal{Z}$-absorption passes to
minimal tensor products (Theorem~\ref{thmztz}).
In particular, the Razak-Jacelon algebra $\mathcal{W}$
is a stably projectionless C*-algebra which is
tracially $\mathcal{Z}$-absorbing.

\medskip

It is shown in \cite[Theorem~3.3]{HO13} that if
 $A$ is a simple tracially $\mathcal{Z}$-absorbing C*-algebra
 then the Cuntz semigroup $W(A)$ is almost unperforated
 and hence $A$ has strict comparison. We extend this to
 the nonunital case. 
 We define a notion of strict  comparison
 (called weak strict comparison)
in a sense suitable for nonunital C*-algebras
(Definition \ref{D_8604_StrComp}).
Purely infinite simple C*-algebras
have weak strict comparison,
and for simple unital C*-algebras,
we show that our definition reduces to strict comparison.
The proof of different parts of the following result
is given in Section~\ref{sec_cu}.

\begin{thmx}\label{thmx_Cu}
Let A be a simple tracially $\mathcal{Z}$-absorbing C*-algebra.
Then
\begin{enumerate}
\item\label{thmx_Cu_it1}
 $W (A)$ is
almost unperforated;

\item\label{thmx_Cu_it2}
$A$ has weak strict comparison;

\item\label{thmx_Cu_it3}
$A$ is weakly almost divisible.
\end{enumerate}
\end{thmx}

Weak almost divisibility in Part~\eqref{thmx_Cu_it3}
is a weak version of divisibility introduced in \cite{OPW17},
and is useful to show that certain
 crossed products have strict comparison
in the absence of tracial $\mathcal{Z}$-stability \cite{OPW17}.

\medskip
	
The structure of this paper is as follows.
In Section~\ref{sec_pre}, we recall some  facts and
prove some results  on Cuntz subequivalence.
In Section~\ref{sec_tz},
we define the notion of tracial $\mathcal{Z}$-absorption for
simple not necessarily unital C*-algebras (Definition~\ref{deftza})
and investigate its primary properties.
In Section~\ref{sec_per}, we show that
tracial $\mathcal{Z}$-absorption passes to
hereditary subalgebras,
matrix algebras, and direct limits,
and that it is Morita invariant.
In Section~\ref{sec_tzz},
we compare tracial $\mathcal{Z}$-absorption with
$\mathcal{Z}$-absorption.
In Section~\ref{sec_cu},
we study the Cuntz semigroup of 
 tracially $\mathcal{Z}$-absorbing C*-algebras.
In Section \ref{Sec_Finite}, we verify  finite  
tracially $\mathcal{Z}$-absorbing C*-algebras.

\medskip

We use the following (standard) notation.
For a C*-algebra $A$, $A_{+}$ denotes the positive cone of $A$.
Also, $A^{+}$ denotes the unitization of $A$
(adding a new identity even if $A$ is unital),
while $A^{\sim} = A$ if $A$ is unital and
$A^{\sim} = A^{+}$ if $A$ is nonunital.
The notation $a \approx_{\varepsilon} b$
means $\|a - b\| < \varepsilon$.
We write $\mathcal{K} = K (\ell^{2})$ and
${M}_{n} = M_{n} (\mathbb{C})$.
We take ${\mathbb{Z}}_n = {\mathbb{Z}} / n {\mathbb{Z}}$.
(The $p$-adic integers will never appear.)
We take $\mathbb{N} = \{ 1, 2, \ldots \}$.
We abbreviate ``completely positive contractive'' to ``c.p.c.''..
For any C*-algebra~$A$,
we denote its tracial state space by $T (A)$.

\section{Preliminaries on Cuntz subequivalence}\label{sec_pre}

\indent
In this section we recall some results
and provide some lemmas on Cuntz subequivalence
needed in the subsequent sections.
Some of these are already in the literature.
See~\cite{APT11}
for an extensive discussion of Cuntz subequvalence.

Let $A$ be a C*-algebra.
For $a, b \in A_{+}$, we write $a \precsim b$
if $a$ is Cuntz subequivalent to
$b$, i.e., there is a sequence $(v_n)_{n \in \N}$ in $A$
such that $\|a - v_{n} b v_{n}^{*} \| \to 0$.
We write $a \sim b$ if both $a \precsim b$ and $b \precsim a$.
The following lemma will be used in many places.

\begin{lemma}[\cite{KR02}, Lemma~2.2]\label{lemkr}
Let $A$ be a C*-algebra,
let $a, b \in A$ be positive, and let $\varepsilon > 0$.
If $\|a - b\| < \varepsilon$
then there is a contraction $d \in A$ such that
$(a - \varepsilon)_{+} = d b d^{*}$.
In particular, $(a - \varepsilon)_{+} \precsim b$.
\end{lemma}



Observe that if $a \precsim b$ in $A$ then, by definition,
there is a sequence $(v_{n})_{n \in \N}$
in $A$ such that $\| a - v_{n} b v_{n}^{*} \| \to 0$.
But it is not the case that there always exists
a \emph{bounded} sequence with this property.
However, we have the following lemma.
It may be in the literature.
We give a proof for the sake of completeness.
(There is a similar result in \cite[Lemma~2.4(ii)]{KR02},
but there is a gap in the proof,
and the statement may well be false~\cite{RrPC}.
Its proof essentially implies
our lemma.)

\begin{lemma}\label{lembdd}
Let $A$ be a C*-algebra,
let $a, b \in A_{+}$, and let $\delta > 0$.
If $a \precsim (b - \delta)_{+}$ then there exists a sequence
$(v_{n})_{n \in \N}$
in $A$ such that $\|a - v_{n} b v_{n}^{*} \| \to 0$
and
$\| v_{n} \| \leq \| a \|^{1/2} \delta^{- 1/2}$
for every $n \in \mathbb{N}$.
\end{lemma}

\begin{proof}
Let $n \in \mathbb{N}$.
Since $a \precsim (b - \delta)_{+}$,
there exists $w_{n} \in A$ such that
$\bigl\| a - w_{n} (b - \delta)_{+}w_{n}^{*} \bigr\| < \frac{1}{n}$.
By Lemma~\ref{lemkr}, there exists
$d_{n} \in A$ such that
\[
\bigl( a - \tfrac{1}{n} \bigr)_{+}
 = d_{n} w_{n} (b - \delta)_{+} w_{n}^{*}d_{n}^{*}.
\]
By \cite[Lemma~2.4(i)]{KR02}, there exists $v_{n} \in A$ such that
\[
\bigl( a - \tfrac{1}{n} \bigr)_{+} = v_{n} b v_{n}^{*}
\andeqn
\| v_{n} \|
 \leq \| \bigl( a - \tfrac{1}{n} \bigr)_{+} \|^{1/2} \delta^{- 1/2}.
\]
Now
$v_{n} b v_{n}^{*} \to a$
and $\| v_{n} \| \leq \| a \|^{1/2} \delta^{- 1/2}$.
\end{proof}

\begin{lemma}\label{lemkey}
Let $A$ be a C*-algebra,
let $x \in A$ be nonzero, and let $b \in A_{+}$.
Then for any $\varepsilon > 0$,
\[
( x b x^{*} - \varepsilon)_{+}
 \precsim x \bigl( b - \varepsilon / \| x \|^{2} \bigr)_{+} x^{*}
 \precsim \bigl( b - \varepsilon / \| x \|^{2} \bigr)_{+}.
\]
If $\| x \| \leq 1$ then
$(x b x^{*} - \varepsilon)_{+}
 \precsim x (b - \varepsilon)_{+} x^{*}
 \precsim (b - \varepsilon)_{+}$.
\end{lemma}

\begin{proof}
For the first statement, we have
\begin{align*}
\bigl\| x b x^{*}
  - x \bigl( b - \varepsilon / \| x \|^{2} \bigr)_{+} x^{*} \bigr\|
& \leq \| x \|^{2}
   \bigl\| b - \left( b - \varepsilon / \| x \|^{2} \right)_{+} \bigr\|
\\
& \leq \| x \|^{2}
   \cdot \left( \frac{\varepsilon}{\| x \|^{2}} \right)
  = \varepsilon.
\end{align*}
Let $\rh > 0$.
Using \cite[Lemma 2.5(i)]{KR00}
at the first step,
and Lemma~\ref{lemkr} at the second step,
we have
\[
\bigl( (x b x^{*} - \varepsilon)_{+} - \rh \bigr)_{+}
 = (x b x^{*} - \varepsilon - \rh)_{+}
 \precsim x \bigl( b - \varepsilon / \| x \|^{2} \bigr)_{+} x^{*}.
\]
Since $\rh > 0$ is arbitrary,
it follows from \cite[Proposition 2.6]{KR00}
that
\[
(x b x^{*} - \varepsilon)_{+}
 \precsim x \bigl( b - \varepsilon / \| x \|^{2} \bigr)_{+} x^{*}.
\]
Clearly
$x \bigl( b - \varepsilon / \| x \|^{2} \bigr)_{+} x^{*}
 \precsim \bigl( b - \varepsilon / \| x \|^{2} \bigr)_{+}$.

If $\| x \| \leq 1$
then $\varepsilon / \| x \|^{2} \geq \varepsilon$ and so
\[
(x b x^{*} - \varepsilon)_{+}
 \precsim x \bigl( b - \varepsilon / \| x \|^{2} \bigr)_{+} x^{*}
 \leq x (b - \varepsilon)_{+} x^{*}
 \precsim (b - \varepsilon)_{+}.
\]
This completes the proof.
\end{proof}

In general, the reverse of the Cuntz subequivalence relation
in Lemma~\ref{lemkey}
does not hold.
However, a weaker version holds.
See Lemma~\ref{lemst} below.
We need the following proposition in its proof.
It should be in the literature.
We provide a proof for the sake of
completeness.

\begin{proposition}\label{propst}
Let $A$ be a C*-algebra,
let $K \subseteq \mathbb{C}$ be a compact set,
and let
$f \colon K \to \mathbb{C}$ be a continuous function.
Set
\[
S = \bigl\{ x \in M (A) \colon
 {\mbox{$x$ is normal and $\spec (x) \subseteq K$}} \bigr\}.
\]
Then the functional calculus map
$x \mapsto f (x)$,
from $S$ to $M (A)$,
is strictly continuous.
Furthermore, if $(x_{i})_{i \in I}$ is a net in $M (A)_{+}$
which converges to $x \in M (A)$ strictly,
then $x \in M (A)_{+}$
and $(x_{i} - \varepsilon)_{+} \to (x - \varepsilon)_{+}$
strictly.
\end{proposition}

\begin{proof}
First note that $S$ is norm bounded,
by $\sup ( \{ |\lambda| \colon \lambda \in K \} )$.

Let $(x_{i})_{i \in I}$ be a net in~$S$
which tends to $x \in S$ strictly.
We have to show that $f (x_{i}) \to f (x)$ strictly.
Since $(x_{i})_{i \in I}$ is bounded
and the product is jointly continuous in the strict topology
on bounded subsets of $M (A)$,
and since $a \mapsto a^*$ is strictly continuous,
the statement holds when $f (\zt)$ is a polynomial
in $\zt$ and~${\overline{\zt}}$.
Now let $y \in A$.
We need to show that $f (x_{i}) y \to f (x) y$
and $y f (x_{i}) \to y f (x)$ in norm in $A$.
We do only the first;
the second is similar.
Let $\varepsilon > 0$.
Choose $\delta > 0$ such that
$\delta \| y \| < \frac{\varepsilon}{3}$.
There is a polynomial $g (\zt)$
in $\zt$ and~${\overline{\zt}}$ such that
$\sup_{\zt \in K} | g (\zt) - f (\zt) | < \delta$.
Since $g (x_{i}) y \to g (x) y$, there is $i_{0} \in I$ such that
$\|g (x_{i}) y - g (x) y\| < \frac{\varepsilon}{3}$
for any $i \in I$ with $i \geq i_{0}$.
Thus for $i \geq i_{0}$ we have
\begin{align*}
& \| f (x_{i}) y - f (x) y \|
\\
& \hspace*{3em} {\mbox{}}
 \leq \| f (x_{i}) y - g (x_{i}) y \| + \| g (x_{i}) y - g (x) y \|
  + \| g (x) y - f (x) y \|
\\
& \hspace*{3em} {\mbox{}}
  < \delta \| y \| + \tfrac{\varepsilon}{3} + \delta \| y \|
  < \varepsilon.
\end{align*}
So $f (x_{i}) \to f (x)$ strictly.

Since $M (A)_{+}$ is strictly closed in $M (A)$,
the second part of the statement
follows from the first.
\end{proof}

We will use the following lemma
in the proof of Proposition~\ref{prop_stz}.

\begin{lemma}\label{lemst}
Let $A$ be a $\sigma$-unital C*-algebra,
let $x$ be a strictly positive  element, let $b \in M (A)$,
and let $\delta > 0$.
Then for every $\varepsilon > 0$ there is $n_{0} \in \mathbb{N}$
such that for any $n \geq n_{0}$ we have
\[
\bigl( x (b - \delta)_{+} x - \varepsilon \bigr)_{+}
  \precsim x \bigl( x^{1/n} b x^{1/n} - \delta \bigr)_{+} x
  \precsim  \bigl( x^{1/n} b x^{1/n} - \delta \bigr)_{+}.
\]
\end{lemma}

\begin{proof}
Consider $A \subseteq M (A)$
and note that $x^{1/n} \to 1_{M (A)}$ strictly.
So $x^{1/n} b x^{1/n} \to b$ strictly.
Thus, by Proposition~\ref{propst},
$(x^{1/n} b x^{1/n} - \delta)_{+} \to (b - \delta)_{+}$ strictly.
Hence
$x \bigl( x^{1/n} b x^{1/n} - \delta \bigr)_{+} x
 \to x (b - \delta)_{+} x$
in $A$,
and so there is $n_{0} \in \mathbb{N}$
such that for every $n \geq n_{0}$,
\[
\bigl\| x \bigl( x^{1/n} b x^{1/n} - \delta \bigr)_{+} x
  - x (b - \delta)_{+} x \bigr\|
 < \varepsilon.
\]
Therefore, using Lemma~\ref{lemkr} at the first step, we have
\[
\bigl( x (b - \delta)_{+} x - \varepsilon \bigr)_{+}
 \precsim x \bigl( x^{1/n} b x^{1/n} - \delta \bigr)_{+} x
 \precsim \bigl( x^{1/n} b x^{1/n} - \delta \bigr)_{+}.
\]
This completes the proof.
\end{proof}


\section{Tracial $\mathcal{Z}$-absorption}\label{sec_tz}

\indent
In this section we define tracial $\mathcal{Z}$-absorption for
simple not necessarily unital C*-algebras (Definition~\ref{deftza}).
We show that this definition
extends the unital version defined by Hirshberg and Orovitz
(see Remark~\ref{rmkdeftza}).
As an example, every simple purely infinite C*-algebra is
tracially $\mathcal{Z}$-absorbing.
In the separable case,
we give an equivalent definition
for tracial $\mathcal{Z}$-absorption in terms of
the central sequence algebra
(Proposition~\ref{prop_cstza}).
At the end of this section we give another version of
tracial $\mathcal{Z}$-absorption
(called strong tracial $\mathcal{Z}$-absorption;
see Definition~\ref{defstza})
and we compare it with Definition~\ref{deftza}.

We recall the definition
of a unital tracially $\mathcal{Z}$-absorbing C*-algebra.

\begin{definition}[\cite{HO13}, Definition~2.1]\label{deftzau}
A unital C*-algebra $A$ is called
\emph{tracially $\mathcal{Z}$-absorbing} if $A \ncong \mathbb{C}$
and for every finite set $F \subseteq A$,
every $\varepsilon > 0$, every $a \in A_{+} \setminus \{ 0 \}$,
and every $n \in \mathbb{N}$,
there is a c.p.c.~order zero map
$\varphi  \colon  {M}_{n} \to A$ such that the following hold:
\begin{enumerate}[label=$\mathrm{(\arabic*)}$]
\item\label{Item_deftzau_Small}
$1 - \varphi (1) \precsim a$.
\item\label{Itemdeftzau_Comm}
$\| [\varphi (z), b] \| < \varepsilon$
for any $z \in M_{n}$ with $\| z \| \leq 1$ and any $b \in F$.
\end{enumerate}
\end{definition}

\begin{remark}\label{rmkdeftzu}
Although Definition~\ref{deftzau} makes sense for unital C*-algebras,
it works well mainly for
unital \emph{simple} C*-algebras.
For example, in \cite{HO13} the assumption of simplicity
appears in
almost all of the results.
Because of this, in this paper we define
tracial $\mathcal{Z}$-absorption
only for (not necessarily unital) simple C*-algebras.
However, we will not use the assumption
of simplicity in some results
(e.g., Theorems~\ref{thmher} and \ref{thmloc}).
\end{remark}

We need the following equivalent version
of tracial $\mathcal{Z}$-absorption in order to define
this notion for nonunital C*-algebras.

\begin{lemma}\label{lemtzau}
Let $A$ be a unital C*-algebra with $A \not\cong \mathbb{C}$.
Then $A$ is tracially
$\mathcal{Z}$-absorbing if and only if
for every finite set $F \subseteq A$, every $\varepsilon > 0$,
every $a \in A_{+} \setminus \{ 0 \}$,
and every $n \in \mathbb{N}$
there is a c.p.c.~order zero map $\varphi \colon M_{n} \to A$
such that:
\begin{enumerate}[label=$\mathrm{(\arabic*)}$]
\item\label{it1_lemtzau}
$( 1 - \varphi (1) -\varepsilon )_{+} \precsim a$.
\item\label{it2_lemtzau}
$\| [\varphi (z), b]\| < \varepsilon$
for any $z \in M_{n}$ with $\| z \| \leq 1$ and any $b \in F$.
\setcounter{TmpEnumi}{\value{enumi}}
\end{enumerate}
\end{lemma}

\begin{proof}
The forward implication is obvious because
$( 1 - \varphi (1) - \varepsilon )_{+} \precsim 1 - \varphi (1)$.

Now suppose that we are given $F$, $\varepsilon$, $a$, and~$n$
as in the statement.
We have to find a c.p.c.~order zero map
$\psi \colon M_{n} \to A$ such that:
\begin{enumerate}
\setcounter{enumi}{\value{TmpEnumi}}
\item\label{it3_lemtzau}
$1 - \psi (1) \precsim a$.
\item\label{it4_lemtzau}
$\| [\psi (z), b] \| < \varepsilon$
for any $z \in M_{n}$ with $\| z \| \leq 1$ and any $b \in F$.
\end{enumerate}
Choose $\delta > 0$ such that
\[
\delta
 < \frac{\varepsilon}{
   1 + 2 \max \bigl( \{ \| x \| \colon x \in F \} \bigr)}.
\]
By assumption there is a c.p.c.~order zero map
$\varphi \colon M_{n} \to A$
such that $(1 - \varphi (1) - \delta)_{+} \precsim a$
and $\| [\varphi (z), b] \| < \delta$ for any
$z  \in M_{n}$ with $\| z \| = 1$.
Define a continuous function $f  \colon  [0, 1] \to [0, 1]$ by
\[
f (\lambda)
= \begin{cases}
\frac{\lambda}{1 - \delta}
& \hspace*{1em} 0 \leq \lambda \leq 1 - \delta
\\
1
& \hspace*{1em} 1 - \delta < \lambda \leq 1.
\end{cases}
\]
Using functional calculus for c.p.c.~order zero maps
(\cite[Corollary~4.2]{WZ09}),
set $\psi = f (\varphi)$.
Thus $\psi \colon M_{n} \to A$ is a c.p.c.~order zero map,
and,
similar to the proof of
\cite[Lemma~2.8]{ABP16},
we have
\begin{equation}\label{equ1.1}
\| \varphi (z) - \psi (z) \| \leq \delta \| z \|
\end{equation}
for any $z \in M_{n}$ with $\| z \| \leq 1$,
and
\begin{equation}\label{equ1.2}
1 - \psi (1)
 = \tfrac{1}{1 - \delta} (1 - \varphi (1) - \delta)_{+}
 \sim (1 - \varphi (1) - \delta)_{+}
 \precsim a.
\end{equation}
By (\ref{equ1.2}) we get \eqref{it3_lemtzau}.
To prove \eqref{it4_lemtzau},
let $z \in M_{n}$ with $\| z \| \leq 1$
and let $x \in F$.
By (\ref{equ1.1}) we get
\begin{align*}
& \|\psi (z) x - x \psi (z) \|
\\
& \hspace*{3em} {\mbox{}}
 \leq \| \psi (z) x - \varphi (z) x \|
        + \|\varphi (z) x - x \varphi (z) \|
        + \| x \varphi (z)- x \psi (z) \|
\\
& \hspace*{3em} {\mbox{}}
 \leq \delta \| x \| + \delta + \delta \| x \|
 = \delta (1 + 2 \| x \|)
 < \varepsilon. 
\end{align*}
This completes the proof.
\end{proof}

One may also work with $\varepsilon$-order zero maps
to check tracial $\mathcal{Z}$-absorption.
Recall that a c.p.c.\  map $\varphi \colon M_{n} \to A$
is {\emph{$\varepsilon$-order zero}}
if whenever $y, z \in M_{n}$ are orthogonal positive
elements of norm at most~$1$,
then $\| \varphi (y) \varphi (z) \| < \varepsilon$.


\begin{proposition}\label{epoz}
Let $A$ be a unital C*-algebra with $A \not\cong \mathbb{C}$.
Then $A$ is tracially
$\mathcal{Z}$-absorbing if and only if
for any finite set $F \subseteq A$, $\varepsilon > 0$,
$a \in A_{+} \setminus \{ 0 \}$,
and $n \in \mathbb{N}$,
there is a c.p.c.~$\varepsilon$-order zero map
$\varphi \colon M_{n} \to A$ such that:
\begin{enumerate}[label=$\mathrm{(\arabic*)}$]
\item
$(1- \varphi (1) -\varepsilon)_{+} \precsim a$.
\item
$\| [\varphi (z), b] \| < \varepsilon$
for any $z \in M_{n}$ with $\| z \| \leq 1$ and any $b \in F$.
\end{enumerate}
\end{proposition}

\begin{proof}
This follows easily from Lemma~\ref{lemtzau} and
\cite[Proposition 2.5]{KW04}.
\end{proof}

The following lemma will be used several times in the sequel.

\begin{lemma}\label{lemind}
Let $A$ be a C*-algebra and let $a \in A_{+} \setminus \{ 0 \}$.
Suppose that an element $x \in A_{+}$ has the following property.
For any finite set $F \subseteq A$, any $\varepsilon > 0$,
and any $n \in \mathbb{N}$,
there is a c.p.c.~order zero map $\varphi \colon M_{n} \to A$
such that:
\begin{enumerate}[label=$\mathrm{(\arabic*)}$]
\item\label{Item_lemind_1}
$\bigl( x^{2} - x \varphi (1) x - \varepsilon \bigr)_{+} \precsim a$.
\item\label{Item_lemind_2}
$\| [\varphi (z), b] \| < \varepsilon$
for any $z \in M_{n}$ with $\| z \| \leq 1$ and any $b \in F$.
\setcounter{TmpEnumi}{\value{enumi}}
\end{enumerate}
Then every positive element $y \in \overline{A x}$
also has the same property.
\end{lemma}

\begin{proof}
Let $F \subseteq A$ be a finite subset,
let $\varepsilon > 0$, and let $n \in \mathbb{N}$.
Let $\delta > 0$ satisfy
\[
\delta
 < \min \left( 1, \, \frac{\varepsilon}{4 (2 \| y \| + 1)} \right).
\]
Since $y \in \overline{A x}$, there is $c \in A \setminus \{ 0 \}$
such that $\| y - c x \| < \delta$.
Set
\[
\eta = \min
   \left( \varepsilon, \, \frac{\varepsilon}{2 \|c\|^{2}} \right).
\]
By assumption there is a c.p.c.~order zero map
$\varphi \colon M_{n} \to A$ such that
\begin{enumerate}
\setcounter{enumi}{\value{TmpEnumi}}
\item\label{Item_lemind_pf1}
$\bigl( x^{2} - x \varphi (1) x - \eta \bigr)_{+} \precsim a$.
\item\label{Item_lemind_pf2}
$\| [\varphi (z), b] \| < \eta$
for any $z \in M_{n}$ with $\| z \| \leq 1$ and any $b \in F$.
\end{enumerate}
By Lemma~\ref{lemkey} we have
\begin{align*}
\left( c x^{2} c^{*} - c x \varphi (1) x c^{*}
        - \frac{\varepsilon}{2} \right)_{+}
& \precsim
   \left( x^{2} - x \varphi (1) x
          - \frac{\varepsilon}{2 \|c\|^{2}} \right)_{+}
\\
& \precsim \left( x^{2} - x \varphi (1) x - \eta \right)_{+}
  \precsim a.
\end{align*}
Also we have
\begin{align*}
& \bigl\| y^{2} - y \varphi (1) y
  - \left( c x^{2} c^{*} - c x \varphi (1) x c^{*}
    - \tfrac{\varepsilon}{2} \right)_{+} \bigr\|
\\
& \hspace*{3em} {\mbox{}}
 \leq \bigl\| y^{2} - y \varphi (1) y
     - c x^{2} c^{*} + c x \varphi (1) x c^{*} \bigr\|
   + \frac{\varepsilon}{2}
\\
& \hspace*{3em} {\mbox{}}
 \leq \| y^{2} - c x y \| + \| c x y - c x^{2} c^{*} \|
    + \bigl\| y \varphi (1) y - y \varphi (1) x c^{*} \bigr\|
\\
& \hspace*{3em} {\mbox{}} \hspace*{3em} {\mbox{}}
      + \bigl\| y \varphi (1) x c^{*} - c x \varphi (1) x c^{*} \bigr\|
      + \frac{\varepsilon}{2}
\\
& \hspace*{3em} {\mbox{}}
 \leq \delta \| y \| + \delta \|c x\| + \delta \| y \|
     + \delta \|c x\| + \frac{\varepsilon}{2}
%
& \hspace*{3em} {\mbox{}}
\\
& \hspace*{3em} {\mbox{}}
 < 2 \delta (2 \| y \| + 1 ) + \frac{\varepsilon}{2}
 < \varepsilon.
\end{align*}
Thus,
\[
\bigl( y^{2} - y \varphi (1) y - \varepsilon \bigr)_{+}
 \precsim \bigl(c x^{2} c^{*} - c x \varphi (1) x c^{*} -
      \tfrac{\varepsilon}{2} \bigr)_{+}
 \precsim a.
\]
Since $\eta < \varepsilon$,
for any $b \in F$ we have $\| [\varphi (z), b] \| < \varepsilon$.
\end{proof}

\begin{definition}\label{deftza}
We say that a simple C*-algebra $A$
is \emph{tracially $\mathcal{Z}$-absorbing} if
$A \ncong \mathbb{C}$
and for every $x, a \in A_{+}$ with $a \neq 0$,
every finite set $F \subseteq A$, every $\varepsilon > 0$,
and every $n \in \mathbb{N}$,
there is a c.p.c.~order zero map $\varphi \colon M_{n} \to A$
such that:
\begin{enumerate}
\item\label{deftza-it1}
$\bigl( x^{2} - x \varphi (1) x - \varepsilon \bigr)_{+} \precsim a$.
\item\label{deftza-it2}
$\| [\varphi (z), b] \| < \varepsilon$
for any $z \in M_{n}$ with $\| z \| \leq 1$ and any $b \in F$.
\end{enumerate}
\end{definition}

Some definitions of this sort,
such as for the weak tracial Rokhlin property
(see, \cite{FG17}, for a definition in nonunitl case),
have an additional condition,
which here would be to assume $\| a \| = 1$
and require $\| \ph (1) a \ph (1) \| > 1 - \ep$.
We don't include such a condition here
because it doesn't appear in Definition~\ref{deftzau}
(\cite[Definition~2.1]{HO13}).
When $A$ is finite in a suitable sense,
it is automatic that such a condition can be satisfied,
as we show in Section~\ref{Sec_Finite}.


\begin{remark}\label{rmkdeftza}
Let $A$ be a simple C*-algebra.
\begin{enumerate}
\item\label{it1_rmkdeftza}
If $A$ is unital, Definition~\ref{deftza}
is equivalent to Definition~\ref{deftzau},
even without assuming simplicity.
In fact, by Lemma~\ref{lemtzau},
Definition~\ref{deftza} implies Definition~\ref{deftzau}.
The converse holds by Lemma~\ref{lemtzau}
and taking $x = 1$ in Lemma~\ref{lemind}.
\item\label{it2_rmkdeftza}
If $A$ is $\sigma$-unital,
it is enough that the properties in
Definition~\ref{deftza}  hold
for \emph{some} strictly positive element $x \in A$.
This follows from
Lemma~\ref{lemind}.
\item\label{it3_rmkdeftza}
In Definition~\ref{deftza},
the ``c.p.c.~order zero" condition on  $\varphi$  may be replaced by
``c.p.c.~$\varepsilon$-order zero" (cf.~Proposition~\ref{epoz}).
\item\label{it4_rmkdeftza}
Definition~\ref{deftza} holds for $A = \mathbb{C}$
by taking $\varphi = 0$.
To rule this out,
the technical condition $A \ncong \mathbb{C}$ is added.
In fact, every  nonzero simple
tracially $\mathcal{Z}$-absorbing C*-algebra
is infinite dimensional (Lemma~\ref{rmkinfdim} below)
and not type~I (Corollary~\ref{cornone} below).
\item\label{it5_rmkdeftza}
In Definition~\ref{deftza},
it is enough to have the map  $\varphi \colon M_{n} \to A$
for some $l \in \mathbb N$ and each $n > l$.
In fact, if Definition~\ref{deftza} holds for $n = k m$
with $k, m \in \mathbb{N}$,
then it also holds for $n = m$.
\end{enumerate}
\end{remark}


\begin{remark}\label{rmkappu}
It is easy to see that in Definition~\ref{deftza}
it is enough to take $x$ in a norm
dense subset of $A_{+}$.
Moreover, if $(e_{i})_{i \in I}$ is an approximate identity
for $A$,
it is enough to take $x$ from the set $\{ e_{i} \colon i \in I \}$.
This follows
from Lemma~\ref{lemind} because the set
\[
\bigl\{ x \in A_{+} \colon
 {\mbox{$x \in \overline{A e_{i}}$ for some $i \in I$}} \bigr\}
\]
is dense in $A_{+}$.
\end{remark}

\begin{lemma}\label{rmkinfdim}
Let $A$ be a nonzero simple
tracially $\mathcal{Z}$-absorbing C*-algebra.
Then $A$ is infinite dimensional.
\end{lemma}

\begin{proof}
Suppose $\dim (A) < \infty$.
Choose $n \in \mathbb{N}$ such that $n^{2} > \dim (A)$.
We claim that
there is no nonzero c.p.c.~order zero map $\varphi \colon M_{n} \to A$.
Indeed,
if we had such a map~$\varphi$,
then \cite[Theorem~3.3]{WZ09}
would provide a homomorphism $\pi \colon M_{n} \to A$
such that $\varphi (z) = \pi (z) \varphi (1)$
for all $z \in M_{n}$.
But $\dim (A) < n^{2}$, so $\pi = 0$.
Thus $\varphi = 0$.
This proves the claim.

Now it follows from Definition~\ref{deftza} that
for every $x, a \in A_{+}$ with $a \neq 0$
and every $\ep > 0$,
we have $(x^2 - \ep)_{+} \precsim a$.
So $x \precsim a$ for every $x, a \in A_{+}$ with $a \neq 0$.
Since ${\mathrm{dim}} (A) < \infty$,
this can only happen if $A \cong \C$,
which is excluded in Definition~\ref{deftza},
or $A = 0$.
\end{proof}

It may happen that $\varphi$ in Definition~\ref{deftza} is zero.
See Example~\ref{exapi} below.
First we recall (see before Proposition~1.6 of~\cite{Cu2})
that a simple C*-algebra~$A$ (not necessarily unital)
is purely infinite if for every $a \in A_{+} \setminus \{ 0 \}$
there is an infinite projection in ${\overline{a A a}}$.
For our purposes,
the following characterization is often more useful.
It is the simple case of \cite[Definition~4.1]{KR00},
according to which a
not necessarily simple C*-algebra~$A$
is purely infinite if and only if
there are no
characters on $A$ and for all $a, b \in A_{+}$
such that $a \in \overline{A b A}$,
we have $a \precsim b$.

\begin{definition}[\cite{KR00}, Proposition~5.4]\label{defpi}
A simple C*-algebra $A$ is purely infinite if and only if
$A \ncong \mathbb{C}$ and
$a \sim b$ for any $a, b \in A_{+} \setminus \{ 0 \}$.
\end{definition}

\begin{example}\label{exapi}
Let $A$ be a  simple purely infinite C*-algebra.
Then $A$ is tracially $\mathcal{Z}$-absorbing.
In fact, we can take $\varphi = 0$ in Definition~\ref{deftza}.
\end{example}

\begin{lemma}\label{P_8809_TAFtoTrZAbs}
Let $A$ be an infinite dimensional simple unital C*-algebra
with tracial rank zero
\cite[Definition 3.6.2]{LnBook}.
Then $A$ is tracially $\mathcal{Z}$-absorbing.
\end{lemma}

In Proposition~\ref{P_8Y03_NonUTAFtoTrZAbs} below,
we remove the requirement that $A$ be unital.

A variant of this lemma is stated in \cite[Lemma~5.1]{MS}
with the assumption that $A$ is separable, 
and is used to show that every
  simple separable unital nuclear infinite-dimensional 
  C*-algebra with tracial rank zero  is approximately divisible, and hence $\mathcal{Z}$-absorbing
  \cite[Theorem~5.4]{MS}.

\begin{proof}[Proof of Lemma~\ref{P_8809_TAFtoTrZAbs}]
We prove the following statement,
which clearly implies the conclusion.
Let $a \in A_{+} \setminus \{ 0 \}$,
let $F \subseteq A$ be finite,
let $\varepsilon > 0$,
and let $n \in \mathbb{N}$.
Then there is a homomorphism
$\varphi \colon M_{n} \to A$
such that:
\begin{enumerate}
\item\label{P_8809_TAFtoTrZAbs-it1}
$1 - \ph (1) \precsim a$.
\item\label{P_8809_TAFtoTrZAbs-it2}
$\| [\varphi (z), x] \| < \varepsilon$
for any $z \in M_{n}$ with $\| z \| \leq 1$ and any $x \in F$.
\setcounter{TmpEnumi}{\value{enumi}}
\end{enumerate}
%

Since $A$ has real rank zero
(by \cite[Theorem 3.6.11]{LnBook})
and is unital, infinite dimensional, and simple,
there are nonzero orthogonal projections
$q_1, q_2 \in {\overline{a A a}}$.
Apply the definition of tracial rank zero
to find a projection $p \in A$,
$m, r (1), r (2), \ldots, r (m) \in \N$,
and a unital injective homomorphism $\ps$
from $B = \bigoplus_{l = 1}^m M_{r (l)}$ to $p A p$
such that:
\begin{enumerate}
\setcounter{enumi}{\value{TmpEnumi}}
\item\label{Item_P_8809_TAFtoTrZAbs_Comm}
$\| [p, x] \| < \frac{\ep}{6}$ for all $x \in F$.
\item\label{Item_P_8809_TAFtoTrZAbs_App}
$\dist (p x p, \, \ps (B)) < \frac{\ep}{6}$ for all $x \in F$.
\item\label{Item_P_8809_TAFtoTrZAbs_Small}
$1 - p \precsim q_1$.
\end{enumerate}
Choose $n_0 \in \N$ such that
\[
\frac{1}{n_0} < \inf_{\ta \in T (A)} \ta (q_2).
\]
Choose $\nu \in \N$ so large that $\nu n > n_0$.

For $l = 1, 2, \ldots, m$,
let $\big( e_{j, k}^{(l)} \big)_{j, k = 1, 2, \ldots, r(l)}$
be the standard system of matrix units for~$M_{r (l)}$.
Use \cite[Lemma 2.3]{OP06} to choose projections
$f_0^{(l)}, f_1^{(l)}, f_2^{(l)}, \ldots, f_{\nu n}^{(l)}
  \in \ps \big( e_{1, 1}^{(l)} \big) A \ps \big( e_{1, 1}^{(l)} \big)$
such that
\[
f_1^{(l)} \sim f_2^{(l)} \sim \cdots \sim f_{\nu n}^{(l)},
\qquad
f_0^{(l)} \precsim f_1^{(l)},
\andeqn
\sum_{j = 0}^{\nu n} f_j^{(l)} = \ps \big( e_{1, 1}^{(l)} \big).
\]
There is then a homomorphism
$\bt_l \colon M_{\nu n} \to A$
such that the images of the diagonal matrix units in $M_{\nu n}$
are $f_1^{(l)}, f_2^{(l)}, \ldots, f_{\nu n}^{(l)}$.
Define a homomorphism
$\bt \colon M_{\nu n} \to A$
by
\[
\bt (z)
 = \sum_{l = 1}^m \sum_{k = 1}^{r (l)}
   \ps \bigl( e_{k, 1}^{(l)} \bigr)
        \bt_l (z) \ps \bigl( e_{1, k}^{(l)} \bigr)
\]
for $z \in M_{\nu n}$.
Clearly $\bt (z)$ commutes with $\ps (b)$
for all $z \in M_{\nu n}$ and $b \in B$,
and $\bt (1) \leq p$.

Let $x \in F$ and let $z \in M_{n}$ satisfy $\| z \| \leq 1$.
Apply \eqref{Item_P_8809_TAFtoTrZAbs_App} to choose $b \in B$ such that $\| \ps (b) - p x p \| < \frac{\ep}{6}$.
We have $[ \bt (z), \, (1 - p) x (1 - p) ] = 0$, and
so
by \eqref{Item_P_8809_TAFtoTrZAbs_Comm} at the third step,
we get
\begin{align*}
\| [\bt (z), x] \|
& \leq \| [\bt (z), \, p x p] \|
  + \| [\bt (z), \, p x (1 - p)] \|
  + \| [\bt (z), \, (1 - p) x p] \|
\\
& \leq 2 \| \bt (z) \| \cdot \| p x p - \ps (b) \|
\\
& \hspace*{3em} {\mbox{}}
      +  \| \bt (z) \| \cdot \| p x (1 - p) \|
      + \| \bt (z) \| \cdot \| (1 - p) x p \|
\\
& < 2 \left( \frac{\ep}{6} \right) +  \left( \frac{\ep}{6} \right)
    +  \left( \frac{\ep}{6} \right)
  < \ep.
\end{align*}
This is~\eqref{P_8809_TAFtoTrZAbs-it2}
with $\bt$ in place of~$\ph$.

For every $\ta \in T (A)$,
for $l = 1, 2, \ldots, m$
we have
\[
\ta \big( f_1^{(l)} \big) = \ta \big( f_2^{(l)} \big) = \cdots
   = \ta \big(  f_{\nu n}^{(l)} \big),
\qquad
\ta \big( f_0^{(l)} \big) \leq \ta \big( f_1^{(l)} \big),
\]
and
\[
\sum_{j = 0}^{\nu n} \ta \big( f_j^{(l)} \big)
  = \ta \big( \ps \big( e_{1, 1}^{(l)} \big) \big).
\]
Therefore,
\[
\ta \big( f_0^{(l)} \big)
\leq \frac{1}{\nu n + 1}
   \ta (\big( \ps \big( e_{1, 1}^{(l)} \big) \big).
\]
Since
\[
p - \bt (1)
  = \sum_{l = 1}^m \sum_{k = 1}^{r (l)}
       \ps \big( e_{k, 1}^{(l)}\big) f_0^{(l)} \ps \big( e_{1, k}^{(l)}\big),
\]
it follows that
\begin{align*}
\ta (p - \bt (1))
& = \sum_{l = 1}^m r (l) \ta \big( f_0^{(l)} \big)
  \leq \frac{1}{\nu n + 1} \sum_{l = 1}^m r (l)
     \ta \big( \ps \big( e_{1, 1}^{(l)} \big) \big)
\\
& = \frac{\ta (p)}{\nu n + 1}
  \leq \frac{1}{\nu n + 1}
  < \frac{1}{n_0}
  < \ta (q_2).
\end{align*}
Since $A$ has strict comparison of projections
using traces
\cite[Theorem 3.7.2]{LnBook},
it follows that $(p - \bt (1) \precsim q_2$.
We already have $1 - p \precsim q_1$,
$q_1 q_2 = 0$, and $q_1 + q_2 \in {\overline{a A a}}$,
so $1 - \bt (1) \precsim a$.
This is~\eqref{P_8809_TAFtoTrZAbs-it1}
with $\bt$ in place of~$\ph$.

To complete the proof,
we define $\ph \colon M_n \to A$
by $\ph (z) = \bt (1_{M_{\nu}} \otimes z)$
for $z \in M_n$.
\end{proof}

Now we give an equivalent definition
for the tracial $\mathcal{Z}$-absorption in terms of
the central sequence algebra.

\begin{notation}\label{N_8422_CentSq}
If $(A_{i})_{i \in I}$
is a family of C*-algebra,
we let $\prod_{i \in I} A_i$ denote that
set of all families $(a_{i})_{i \in I}$
in the set theoretic product of the $A_{i}$
such that $\sup_{i \in I} \| a_{i} \| < \infty$.
(This is the product in the category of C*-algebras.)

For a C*-algebra $A$, we write
\[
A_{\infty}
 = \prod_{n \in \mathbb{N}} A \Big/ \bigoplus_{n \in \mathbb{N}} A.
\]
We identify $a \in A$ with the equivalence class
in $A_{\infty}$ of the corresponding constant sequence
$(a)_{n \in \mathbb{N}}$.
We write $A_{\infty} \cap A^{\prime}$
for the relative commutant of $A$ in $A_{\infty}$.
\end{notation} 

The algebra $A_{\infty} \cap A^{\prime}$
is often written $A^{\infty}$.
We warn that some authors write $A^{\infty}$
for our $A_{\infty}$.

\begin{proposition}\label{prop_cstza}
Let $A$ be a simple separable C*-algebra.
The following
statements are equivalent:
\begin{enumerate}
\item\label{prop_cental_sequence_1}
$A$ is tracially $\mathcal{Z}$-absorbing.

\item\label{prop_cental_sequence_3}
For every $a, x \in A_{+}$
with $a \neq 0$ and every $n \in \N$,
there exists a c.p.c.~order zero map
$\psi \colon M_n \to   A_{\infty} \cap A^{\prime}$
such that $x^{2} - x \psi (1) x \precsim a$ in $A_{\infty}$.
\setcounter{TmpEnumi}{\value{enumi}}
\end{enumerate}
If moreover $A$ is unital, then these statements 
are equivalent to the following:

\begin{enumerate}
\setcounter{enumi}{\value{TmpEnumi}}
\item\label{prop_cental_sequence_2}
For every nonzero $a  \in A_{+} $
 and every $n \in \N$,
there exists a c.p.c.~order zero map
$\psi \colon M_n \to   A_{\infty} \cap A^{\prime}$
such that $1_{A_{\infty}} -  \psi (1) \precsim a$ in $A_{\infty}$.
\end{enumerate}
\end{proposition}

\begin{proof}
\eqref{prop_cental_sequence_1}\,$\Rightarrow$\eqref{prop_cental_sequence_3}:
Let $\{a_{1}, a_{2}, \ldots \}$  be a dense subset of
the closed unit ball of $A$.
Put $F_{m} = \{a_{1}, a_{2}, \ldots, a_{m} \}$ for
$m \in \mathbb{N}$.
Let $a, x \in A_{+}$ with $a \neq 0$ and let $n \in \mathbb{N}$.
We may assume that $\| a \| = 1$.
Tracial $\mathcal{Z}$-absorption of $A$ implies that
for every $m \in \mathbb{N}$
there exists a c.p.c.~order zero map $\varphi_{m} \colon M_{n} \to A$
such that
\[
\left( x^{2} - x \varphi_{m} (1) x - \tfrac{1}{m} \right)_{+}
 \precsim \left(a - \tfrac{1}{2} \right)_{+}
\]
and $\| [\varphi_{m} (z), b] \| < \frac{1}{m}$
for every $z \in M_{n}$ with $\| z \| \leq 1$ and every $b \in F_{m}$.
Now let $\pi \colon \prod_{m \in \mathbb{N}} A \to A_{\infty}$
be the quotient map,
and define $\psi \colon M_n \to A_{\infty}$ by
\[
\psi (z) = \pi \big( ( \varphi_m (z) )_{m \in \N} \big)
\]
for $z \in M_n$.
Then
$\psi (M_n) \subseteq A_{\infty} \cap  A^{\prime}$.
By Lemma~\ref{lembdd},
for every $m \in \mathbb{N}$ there exists  $v_{m} \in A$ such that
$\|v_{m} \| \leq 2 \| x \|^{2}$ and
\[
\big\| \left( x^{2} - x \varphi_{m} (1) x
        - \tfrac{1}{m} \right)_{+} - v_{m} a v_{m}^{*} \big\|
 < \tfrac{1}{m}.
\]
Hence
$\big\| x^{2} - x \varphi_{m} (1) x - v_{m} a v_{m}^{*} \big\|
 < \frac{2}{m}$.
Put $v = \pi \big( ( v_m  )_{m \in \N} \big) \in A_{\infty}$.
Then we have $x^{2} - x \psi (1) x = vav^{*} \precsim a$
in $A_{\infty}$.

\eqref{prop_cental_sequence_3}\,$\Rightarrow$\eqref{prop_cental_sequence_1}:
Suppose that \eqref{prop_cental_sequence_3}
holds. To show that $A$ is tracially
$\mathcal{Z}$-absorbing,
we verify
the condition in Remark \ref{rmkdeftza}\eqref{it3_rmkdeftza}
(Definition~\ref{deftza} with order zero replaced by
$\varepsilon$-order zero).
So
let $a$, $x$, $\varepsilon$, $F$, and~$n$
be as in Definition~\ref{deftza}.
Choose a c.p.c.~order zero map
$\psi \colon M_{n} \to  A_{\infty} \cap A^{\prime}$
such that
$x^{2} - x \psi (1) x \precsim a$ in $A_{\infty}$.
Since $M_{n}$ is nuclear,
there is a sequence
$(\psi_{m})_{m \in \mathbb{N}}$
of c.p.c.\  maps from
$M_{n}$ to $A$
such that,
if we define
\[
\widetilde{\psi}
 = (\psi_m)_{m \in \mathbb{N}}
  \colon M_{n} \to \prod_{m \in \mathbb{N}} A,
\]
then
$\psi = \pi \circ \widetilde{\psi}$.
Since $F$ is finite and $M_n$ is finite dimensional,
it follows that there is $m \in \mathbb{N}$ such that
$\| [\psi_{m} (z), y] \| < \varepsilon$
for every $z \in M_{n}$ with $\| z \| \leq 1$
and every $y \in F$,
such that $\| \psi_{m} (y) \psi_{m} (z) \| < \varepsilon$
for all $y, z \in M_{n}$ with $0 \leq y, z \leq 1$
and $y z = 0$,
and such that there is $w \in A$ with
$\big\| x^{2} - x \psi_{m} (1) x - w a w^{*} \big\| < \varepsilon$.
Thus,
$\big( x^{2} - x \psi_{m} (1) x - \varepsilon \big)_{+} \precsim a$.
Hence, $\psi_{m}$ satisfies
the conditions of Remark \ref{rmkdeftza}\eqref{it3_rmkdeftza},
and so $A$ is tracially $\mathcal{Z}$-absorbing.
Thus \eqref{prop_cental_sequence_1} holds.

Suppose that $A$ is unital.
The implication
\eqref{prop_cental_sequence_2}\,$\Rightarrow$\eqref{prop_cental_sequence_3}  is obvious since
$x^{2} - x \psi (1) x \precsim
1_{A_{\infty}} -  \psi (1)$.
Also, 
\eqref{prop_cental_sequence_3}\,$\Rightarrow$\eqref{prop_cental_sequence_2}
follows by taking $x=1_A$ in \eqref{prop_cental_sequence_3}.
\end{proof}

To conclude this section,
we give another version of tracial $\mathcal{Z}$-absorption
(called strong tracial $\mathcal{Z}$-absorption,
Definition~\ref{defstza} below),
and we compare it with Definition~\ref{deftza}.
By Remark \ref{rmkdeftza}(\ref{it1_rmkdeftza})
and Remark~\ref{rmk_defstza},
both definitions are equivalent to \cite[Definition~2.1]{HO13}
in the simple unital case.
Definition~\ref{defstza} was our first proposal
for the definition of tracial $\mathcal{Z}$-absorption.
At first glance, it seems to be a more natural
definition extending the unital case.
However, it turned out that Definition~\ref{deftza}
is better.
We do not have examples to show that these two
definitions are actually different.

\begin{definition}\label{defstza}
We say that a simple C*-algebra $A$ is
\emph{strongly tracially $\mathcal{Z}$-absorbing}
if $A \not\cong \mathbb{C}$
and for every  $x, a \in A_{+}$ with $a \neq 0$,
every finite set $F \subseteq A$, every $\varepsilon > 0$,
and every $n \in \mathbb{N}$,
there is a c.p.c.~order zero map $\varphi \colon M_{n} \to A$
such that:
\begin{enumerate}
\item\label{defstza-it1}
$x^{2} - x \varphi (1) x \precsim a$.
\item\label{defstza-it2}
$\| [\varphi (z), b] \| < \varepsilon$
for any $z \in M_{n}$ with $\| z \| \leq 1$ and any $b \in F$.
\end{enumerate}
\end{definition}

As an example,
by taking $\varphi = 0$,
we see that
every simple (not necessarily unital) purely infinite C*-algebra
is strongly tracially $\mathcal{Z}$-absorbing.
(Compare with Example~\ref{exapi}.)
In particular,
Examples \ref{E_6825_PINotZSt} and~\ref{E_8702_PINotZStable}
show that strong tracial $\mathcal{Z}$-absorption
does not imply $\mathcal{Z}$-absorption.

\begin{lemma}\label{lemstr}
Let $A$ be a C*-algebra and let $x, y \in A_{+}$.
If $y \in \overline{x A}$ then for any
$c \in (A^{\sim})_{+}$ we have
$y c y \precsim x cx$.
\end{lemma}

\begin{proof}
Choose $(b_{n})_{n \in \mathbb{N}}$ in $A$ such that $y b_{n} \to y$.
Then $b_{n}^{*} x c x b_{n} \to y c y$ and for any
$n \in \mathbb{N}$,
we have
$b_{n}^{*} x c x b_{n} \precsim x cx$.
Taking the limit in the latter gives the result.
\end{proof}

\begin{remark}\label{rmk_defstza}
By Lemma~\ref{lemstr},
if the condition stated in Definition~\ref{defstza}
holds for some $x \in A_{+}$
then it also holds for any positive element $y \in \overline{xA}$.
In particular, Definition~\ref{defstza}
is equivalent to Definition~\ref{deftza}
in the  unital case.
Moreover, it follows from Lemma~\ref{lemstr} that if $A$ is
$\sigma$-unital,
then $A$ is strongly tracially $\mathcal{Z}$-absorbing
if the properties in
Definition~\ref{defstza}
hold for \emph{some} strictly positive element $x \in A$.
\end{remark}

The following result clarifies the relation
between Definitions~\ref{deftza} and
\ref{defstza}.

\begin{proposition}\label{prop_stz}
Let $A$ be a simple $\sigma$-unital C*-algebra
and let $x \in A$ be a strictly positive element.
Then $A$ is strongly tracially $\mathcal{Z}$-absorbing
if and only if
for every $a \in A_{+}$ with $a \neq 0$,
every finite set $F \subseteq A$, every $\varepsilon > 0$,
and every $n \in \mathbb{N}$,
there is a c.p.c.~order zero map $\varphi \colon M_{n} \to A$
such that:
\begin{enumerate}[label=$\mathrm{(\arabic*)}$]
\item\label{it1_prop_stz}
$\left( x^{2 / m} - x^{1 / m} \varphi (1) x^{1 / m}
    - \varepsilon \right)_+
  \precsim a$
for every $m \in \mathbb{N}$.
\item\label{it2_prop_stz}
$\| [\varphi (z), b] \| < \varepsilon$
for any $z \in M_{n}$ with $\| z \| \leq 1$ and any $b \in F$.
\end{enumerate}
\end{proposition}

\begin{proof}
Assume first that $A$ is strongly tracially $\mathcal{Z}$-absorbing.
Let $a$, $F$, $\varepsilon$, and~$n$
be as in the hypotheses,
and apply Definition~\ref{defstza}
with these elements and $x$ as given,
obtaining a c.p.c.~order zero map $\varphi \colon M_{n} \to A$.
This map satisfies~\ref{it2_prop_stz} by construction.
For~\ref{it1_prop_stz},
let $m \in \mathbb{N}$.
Then $x^{1 / m} \in {\overline{x A}}$,
so,
using Lemma \ref{lemstr} at the second step,
\[
\big( x^{2 / m} - x^{1 / m} \varphi (1) x^{1 / m}
        - \varepsilon \big)_{+}
 \precsim x^{2 / m} - x^{1 / m} \varphi (1) x^{1 / m}
 \precsim x^{2} - x \varphi (1) x.
\]

For the converse,
let $a$, $F$, $\varepsilon$, and~$n$ be as in Definition~\ref{defstza}.
Choose $\delta > 0$ with
\[
\delta
 < \frac{\varepsilon}{1
               + 2 \max \big( \{ \| y \| \colon y \in F \} \big)}.
\]
By assumption there is a c.p.c.~order zero map
$\varphi \colon M_{n} \to A$ such that \ref{it1_prop_stz}
and \ref{it2_prop_stz}  hold for
$\delta$ instead of $\varepsilon$.
Define a map $f \colon [0, 1] \to [0, 1]$ by
\[
f (\lambda)
 = \begin{cases}
\tfrac{\lambda}{1 - \delta}
& \hspace*{1em} 0 \leq \lambda \leq 1 - \delta
\\
1
& \hspace*{1em} 1 - \delta < \lambda \leq 1.
\end{cases}
\]
Using functional calculus for c.p.c.~order zero maps
(\cite[Corollary~4.2]{WZ09}),
set $\psi = f (\varphi)$.
As in the proof of Lemma~\ref{lemtzau},
in ${A}^{\sim}$ we have
\begin{equation}\label{equ1_propstz}
1 - \psi (1)
 = \frac{1}{1 - \delta} \big( 1 - \varphi (1) - \delta \big)_{+}.
\end{equation}
We show that
\eqref{defstza-it1} in Definition~\ref{defstza} holds, i.e.,
$x^{2} - x \psi (1) x \precsim a$.
Let $\eta > 0$.
Then
by \eqref{equ1_propstz} we have
\[
\big((1 - \delta) (x^{2} - x \psi (1) x ) - \eta \big)_{+}
  = \big( x (1 - \varphi (1) - \delta)_{+} x - \eta \big)_{+}.
\]
By Lemma~\ref{lemst}, there is $m \geq 1$ such that
\begin{equation}\label{Eq_8531_Subeq}
\big( x (1 - \varphi (1) - \delta)_{+} x - \eta \big)_{+} \precsim
\big( x^{1 / m} (1 - \varphi (1)) x^{1 / m} - \delta \big)_{+}.
\end{equation}
The right hand side of~(\ref{Eq_8531_Subeq})
is Cuntz subequivalent to~$a$ by~\ref{it1_prop_stz}.
Since $\eta > 0$ is arbitrary,
and using~(\ref{equ1_propstz}) at the first step,
we get
\[
x^{2} - x \psi (1) x \sim x (1 - \varphi (1)- \delta)_{+} x \precsim a.
\]
The proof of \eqref{defstza-it2} in Definition~\ref{defstza}
is similar to the
last part of the proof of Lemma~\ref{lemtzau}.
So
$A$ is strongly tracially $\mathcal{Z}$-absorbing.
\end{proof}

\section{Permanence properties}\label{sec_per}

\indent
In this section we show that tracial $\mathcal{Z}$-absorption
passes to hereditary C*-subalgebras (Theorem~\ref{thmher}).
In particular,
no simple tracially $\mathcal{Z}$-absorbing C*-algebra is type~I.
Then we show that tracial $\mathcal{Z}$-absorption passes to
matrix algebras and direct limits.
We also prove that
it is preserved under Morita equivalence
in the class of $\sigma$-unital
simple C*-algebras (Corollary~\ref{cor_morita}).
Moreover, we compare our nonunital definition of
tracial $\mathcal{Z}$-absorption with the unital definition when the
algebra has an approximate identity (not necessarily increasing)
consisting of projections (Proposition~\ref{prop_apup}).

\begin{theorem}\label{thmher}
Let $A$ be a simple tracially $\mathcal{Z}$-absorbing C*-algebra
and let $B$ be a hereditary C*-subalgebra of $A$.
Then $B$
is also tracially $\mathcal{Z}$-absorbing.
\end{theorem}

\begin{proof}
Let $x \in B_{+}$, let $a \in B_{+} \setminus \{ 0 \}$,
let $\varepsilon > 0$,
let $F \subseteq B$ be finite, and let $n \in \mathbb{N}$.
We must find a c.p.c.~order zero map
$\varphi \colon M_{n} \to B$ such that:
\begin{enumerate}
\item\label{it1_thmher}
$\bigl( x^{2} - x \varphi (1) x - \varepsilon \bigr)_{+} \precsim a$.
\item\label{it2_thmher}
$\| [\varphi (z), b] \| < \varepsilon$
for any $z \in M_{n}$ with $\| z \| \leq 1$ and any $b \in F$.
\setcounter{TmpEnumi}{\value{enumi}}
\end{enumerate}
The case $x = 0$ is trivial,
so, by Lemma~\ref{lemind}, we may assume that $\| x \| = 1$.
Also, we may assume that $\| b \| \leq 1$ for all $b \in F$.
Let $y_0$ be the sum of $x$ with the positive and negative parts of
the real and imaginary parts of all the elements of $F$,
and set $y = \| y_0 \|^{-1} y_0$.
Then $y \in B_{+}$, $\| y \| = 1$,
and $F \cup \{ x \} \subseteq {\overline{y B y}}$.
Since $\bigl( y^{1 / m} \bigr)_{m \in \N}$
is an approximate identity for
${\overline{y B y}}$, there is $m \in \mathbb{N}$ such that
\begin{equation}\label{equ1}
\bigl\| [y^{1 / k}, b] \bigr\|  < \tfrac{\varepsilon}{6}
\andeqn
\bigl\| y^{1 / k} b - b \bigr\| < \tfrac{\varepsilon}{6}
\end{equation}
for $k \in \{ m, 2 m, m / 2 \}$ and $b \in F \cup \{ x \}$.
Choose $\delta > 0$ for $\frac{\varepsilon}{6}$
according to \cite[Lemma 2.7]{ABP16}.
We may assume that
$\delta < \frac{\varepsilon}{4}$.
Since $A$ is tracially $\mathcal{Z}$-absorbing,
there exists a c.p.c.~order zero map $\psi \colon M_{n} \to A$
such that:
\begin{enumerate}
\setcounter{enumi}{\value{TmpEnumi}}
\item\label{it3_thmher}
$\bigl( y^{2 / m} - y^{1 / m} \psi (1) y^{1 / m} - \delta \bigr)_{+}
 \precsim a$.
\item\label{it4_thmher}
$\| [\psi (z) , b] \| < \delta$ for every
$z \in M_{n}$ with $\| z \| \leq 1$ and every
$b \in F \cup \left\{ x, y^{2 / m}, y^{1 / m} \right\}$.
\end{enumerate}
Let us show that the conditions of \cite[Lemma 2.7]{ABP16}
are satisfied
for $y^{1 / m}$, $\psi$, and~$B$
in place of $x$, $\ph_0$, and~$B$.
Since $\| y \| \leq 1$,
we have $0 \leq y^{1 / m} \leq 1$.
Also by \eqref{it4_thmher},
$\bigl\| [\psi (z), \ , y^{1 / m}] \bigr\| < \delta$
for any $z \in M_{n}$ with $\| z \| \leq 1$.
Finally,
we claim that $\dist \bigl( \psi (z) y^{1 / m}, \, B \bigr) < \delta$
for any $z \in M_{n}$ with $\| z \| \leq 1$.
In fact, $y^{1 / (2 m)} \psi (z) y^{1 / (2 m)} \in B$
(since $B$ is hereditary) and  by \eqref{it4_thmher} we have
\[
\bigl\| \psi (z) y^{1 / m}
        - y^{1 / (2 m)} \psi (z) y^{1 / (2 m)} \bigr\|
 \leq \bigl\| \psi (z) y^{1 / (2 m)} - y^{1 / (2 m)} \psi (z) \bigr\|
 < \delta,
\]
proving the claim.
Hence by the choice of $\dt$ using \cite[Lemma~2.7]{ABP16},
there is a c.p.c.~order zero map
$\varphi \colon M_{n} \to B$ such that for
every $z \in M_{n}$ with $\| z \| \leq 1$ we have
\begin{equation}\label{equ2}
\bigl\| \psi (z) y^{1 / m} - \varphi (z) \bigr\|
 < \tfrac{\varepsilon}{6}.
\end{equation}

Now we show that $\varphi$ has  properties
\eqref{it1_thmher} and \eqref{it2_thmher} above.
For \eqref{it2_thmher},
let $b \in F$ and let $z \in M_{n}$ satisfy $\| z \| \leq 1$.
Using
(\ref{equ1}), (\ref{equ2}), and~(\ref{it4_thmher})
at the second step, we have
\begin{align*}
\| \varphi (z) b - b \varphi (z) \|
& \leq \bigl\| \varphi (z) - \psi (z) y^{1 / m} \bigr\|\cdot \| b \|
     + \| \psi (z) \|\cdot \bigl\| y^{1 / m} b - b y^{1 / m} \bigr\|
\\
& \hspace*{3em} {\mbox{}}
  + \| \psi (z) b - b \psi (z) \|\cdot \bigl\| y^{1 / m} \bigr\|
  + \| b \|\cdot \bigl\| \psi (z) y^{1 / m} - \varphi (z) \bigr\|
\\
& < \tfrac{\varepsilon}{6} + \tfrac{\varepsilon}{6}
    + \delta + \tfrac{\varepsilon}{6}
  < \tfrac{\varepsilon}{2} + \tfrac{\varepsilon}{4}
  < \varepsilon.
\end{align*}

To prove \eqref{it1_thmher}, first we
use
(\ref{equ2}) and~(\ref{it4_thmher})
at the third step to get
\begin{align*}
& \bigl\| \bigl( y^{2 / m}
            - y^{1 / m} \psi (1) y^{1 / m} - \delta \bigr)_{+}
  - \bigl( y^{2 / m}
             - y^{1 / (2 m)} \varphi (1) y^{1 / (2 m)} \bigr) \bigr\|
\\
& \hspace*{3em} {\mbox{}}
\leq \delta
   + \bigl\| y^{1 / m} \psi (1) y^{1 / m}
          - y^{1 / (2 m)} \varphi (1) y^{1 / (2 m)} \bigr\|
%
& \hspace*{3em} {\mbox{}}
%
& \hspace*{3em} {\mbox{}} \hspace*{3em} {\mbox{}}
%
& \hspace*{3em} {\mbox{}}
\\
& \hspace*{3em} {\mbox{}}
\leq \delta
   + \bigl\| y^{1 / (2 m)} \bigr\|
    \cdot   \bigl\| y^{1 / (2 m)} \psi (1) - \psi (1) y^{1 / (2 m)} \bigr\| \cdot
       \bigl\| y^{1 / m} \bigr\|
\\
& \hspace*{3em} {\mbox{}} \hspace*{3em} {\mbox{}}
   + \bigl\| y^{1 / (2 m)} \bigr\|
     \cdot  \bigl\| \psi (1) y^{1 / m} - \varphi (1) \bigr\|\cdot
       \bigl\| y^{1 / (2 m)} \bigr\|
\\
& \hspace*{3em} {\mbox{}}
< \delta + \tfrac{\varepsilon}{6}
< \tfrac{\varepsilon}{4} + \tfrac{\varepsilon}{6}
< \tfrac{\varepsilon}{2}.
\end{align*}
Thus by \eqref{it3_thmher} we have
\begin{equation}\label{Eq_8702_NewStar}
\bigl( y^{2 / m} - y^{1 / (2 m)} \varphi (1) y^{1 / (2 m)}
  - \tfrac{\varepsilon}{2} \bigr)_{+}
 \precsim \bigl( y^{2 / m} - y^{1 / m} \psi (1) y^{1 / m}
      - \delta \bigr)_{+}
 \precsim a.
\end{equation}
At the second step
using $\bigl\| x - x y^{1 / (2 m)} \bigr\|
 = \bigl\| x - y^{1 / (2 m)} x \bigr\|$
and (\ref{equ1}) several times,
we get
%
%
\begin{align*}
& \bigl\| \bigl( x^{2} - x \varphi (1) x \bigl)
   - \bigl( x y^{2 / m} x - x y^{1 / (2 m)} \varphi (1) y^{1 / (2 m)} x
             - \tfrac{\varepsilon}{2} \bigl)_{+} \bigr\|
\\
& \hspace*{3em} {\mbox{}}
 \leq \tfrac{\varepsilon}{2} + \| x \| 
 \cdot
 \bigl\| x - y^{2 / m} x \bigr\| 
      + \bigl\| x - x y^{1 / (2 m)} \bigr\|
      \cdot \| \varphi (1) x \|
\\
& \hspace*{3em} {\mbox{}} \hspace*{3em} {\mbox{}}
          + \bigl\| x y^{1 / (2 m)} \varphi (1) \bigr\| \cdot
              \bigl\| x - y^{1 / (2 m)} x \bigr\|
\\
& \hspace*{3em} {\mbox{}}
 < \tfrac{\varepsilon}{2} + \tfrac{\varepsilon}{2}
 = \varepsilon.
\end{align*}
Using Lemma~\ref{lemkey} at the second step
and~\eqref{Eq_8702_NewStar} at the third step, we now have
\begin{align*}
\bigl( x^{2} - x \varphi (1) x - \varepsilon \bigr)_{+}
& \precsim \bigl( x y^{2 / m} x - x y^{1 / (2 m)} \varphi (1) y^{1 / (2 m)} x
   - \tfrac{\varepsilon}{2} \bigr)_{+}
\\
& \precsim \bigl( y^{2 / m} - y^{1 / (2 m)} \varphi (1) y^{1 / (2 m)}
      - \tfrac{\varepsilon}{2} \bigr)_{+}
  \precsim a.
\end{align*}

It remains to prove that $B \not\cong \C$.
So suppose $B \cong \C$.
Since $B$ is hereditary and $A$ is simple,
it follows that there is a Hilbert space~$H$
such that $A \cong K (H)$.
We have $A \not\cong \C$
and $A \neq 0$,
so $A$ has a hereditary subalgebra $D$ isomorphic to~$M_2$.
Since $D \not\cong \C$,
by what we have already done,
$D$ is tracially $\mathcal{Z}$-absorbing.
This contradicts Lemma~\ref{rmkinfdim}.
\end{proof}

\begin{corollary}\label{cornone}
Let $A$ be a
nonzero simple tracially $\mathcal{Z}$-absorbing C*-algebra.
Then $A$ is not type~I.
In particular, $\mathcal{K}$ is not
tracially $\mathcal{Z}$-absorbing.
\end{corollary}

\begin{proof}
Suppose that $H$ is a Hilbert space
and ${K} (H)$ is  tracially $\mathcal{Z}$-absorbing.
By Lemma~\ref{rmkinfdim}, $H$
is infinite dimensional.
Let $p$ a projection
in ${K} (H)$ with $\mathrm{rank} (p) = 2$.
Then $p{K} (H) p \cong M_{2}$ would be
tracially $\mathcal{Z}$-absorbing by Theorem~\ref{thmher}.
But this is not the case by  Lemma~\ref{rmkinfdim}.
\end{proof}

\begin{theorem}\label{thmeqtz}
Let $A$ be a simple C*-algebra.
The following are equivalent:
\begin{enumerate}[label=$\mathrm{(\arabic*)}$]
\item\label{thmeqtz_it1}
$A$ is tracially $\mathcal{Z}$-absorbing.
\item\label{thmeqtz_it2}
Every $\sigma$-unital hereditary C*-subalgebra  of $A$
is tracially $\mathcal{Z}$-absorbing.
\item\label{thmeqtz_it3}
For every positive element $x$ in $A$,
$\overline{x A x}$ is tracially $\mathcal{Z}$-absorbing.
\end{enumerate}
\end{theorem}

\begin{proof}
By Theorem~\ref{thmher}, \ref{thmeqtz_it1} implies \ref{thmeqtz_it2}.
Clearly, \ref{thmeqtz_it2} is equivalent to \ref{thmeqtz_it3}
since a hereditary C*-subalgebra $B \subseteq A$ is $\sigma$-unital
if and only if there is $x \in A_{+}$
such that $B = \overline{x A x}$.
It remains to show that
\ref{thmeqtz_it3} implies \ref{thmeqtz_it1}.
Suppose that \ref{thmeqtz_it3} holds.
Thus, in particular, $A \ncong \mathbb{C}$.
Suppose that we are given $a, x \in A_{+} \setminus \{ 0 \}$,
$\varepsilon > 0$, a finite set
$F \subseteq A$,
and $n \in \mathbb{N}$.
Choose $y \in A_{+}$ such that
$F \cup \{ x \} \subseteq \overline{y A y}$
(e.g., let $y$ be  the sum of
$x$ with the positive and negative parts
of the real and imaginary parts of the elements of $F$).
We get $\ph$ as in Definition~\ref{deftza}
by applying this definition to $\overline{y A y}$.
\end{proof}


\begin{lemma}[cf.~\cite{HO13}, Lemma~2.3]\label{lemmatrix}
Let $A$ be a simple C*-algebra which is not of type~I
and let $n \in \mathbb{N}$.
For every nonzero
positive element $a \in {M}_{n} \otimes A = {M}_{n} (A)$,
there exists a nonzero positive element
$b \in A$ such that $1 \otimes b \precsim a$.
\end{lemma}

\begin{proof}
For $j, k = 1, 2, \ldots, n$,
let $e_{j, k} \in M_n$
be the standard matrix unit.
Since $a$ is a positive and nonzero element of ${M}_{n} (A)$,
there exists $j$ such that
$(e_{j, j} \otimes 1_{A^{\sim}}) a (e_{j, j} \otimes 1_{A^{\sim}})
 \neq 0$.
By replacing $a$ with
$(e_{j, j} \otimes 1_{A^{\sim}}) a (e_{j, j} \otimes 1_{A^{\sim}})$
we may assume that $a$ is of the form $e_{j, j} \otimes c$
for some nonzero positive $c \in A$.
Use \cite[Lemma~2.1]{Ph14}
to find $b_1, b_2, \ldots, b_n \in A_{+} \setminus \{ 0 \}$
such that $b_1 \sim b_2 \sim \cdots \sim b_n$,
such that $b_j b_k = 0$ for $j \neq k$,
and such that $b_1 + b_2 + \cdots + b_n \in {\overline{c A c}}$.
Then
\[
1 \otimes b
 \sim e_{j, j} \otimes ( b_1 + b_2 + \cdots + b_n )
 \precsim e_{j, j} \otimes c,
\]
as desired.
\end{proof}

The following is a nonunital analog of \cite[Lemma~2.4]{HO13}.

\begin{proposition}\label{propmatrix}
Let $A$ be a simple tracially $\mathcal{Z}$-absorbing C*-algebra
and let $n \in \mathbb{N}$.
Then ${M}_{n} (A)$ is also
tracially $\mathcal{Z}$-absorbing.
\end{proposition}

\begin{proof}
Let $x, a \in {M}_{n} (A)_{+}$ with $a \neq 0$,
let $\varepsilon > 0$,
let $m \in \mathbb{N}$,
and let $F \subseteq {M}_{n} (A)$ be finite.

By Lemma~\ref{lemmatrix}, it is enough to consider the case where
$a = 1 \otimes b \in {M}_{n} \otimes A = {M}_{n} (A)$ for some nonzero
element $b \in A_{+}$.
Also, there exists $y \in A_{+}$
such that $x \in {\overline{M_n (A) (1 \otimes y) }}$.
Thus, by Lemma~\ref{lemind}, it is enough to
take $x =  1 \otimes y$ with $y \in A_{+}$.
Let $E$ be the set of all matrix entries of elements of $F$.
Put $\delta = \varepsilon / n^{2}$.
Since $A$ is tracially $\mathcal{Z}$-absorbing,
there exists a c.p.c.~order zero map
$\varphi \colon M_{m} \to A$ such that:
\begin{enumerate}
\item\label{8531_PfOfpropmatrix_1}
$\big( y^{2} - y \varphi (1) y - \delta \big)_{+} \precsim b$.
\item\label{8531_PfOfpropmatrix_2}
$\| [\varphi (z), d] \| < \delta$
for any $z \in M_{n}$ with $\| z \| \leq 1$ and any $d \in E$.
\end{enumerate}
Define a c.p.c.~order zero map
$\psi \colon M_{m} \to  {M}_{n} \otimes A$ by
$\psi (z) = 1 \otimes \varphi (z)$ for $z \in {M}_{n}$.
Then
\begin{align*}
\big( x^{2} - x \varphi (1) x - \varepsilon \big)_{+}
& = 1 \otimes \big( y^{2} - y \varphi (1) y - \varepsilon \big)_{+}
\\
& \precsim 1 \otimes \big( y^{2} - y \varphi (1) y - \delta \big)_{+}
  \precsim 1 \otimes b
  = a.
\end{align*}
%
%
Also, for any $z \in M_{n}$ with $\| z \| \leq 1$
and any $d = (d_{j, k})_{1 \leq j, k \leq n} \in F$ we have
\begin{align*}
\| [\psi (z), d] \|
& = \| [1 \otimes \varphi (z), d] \|
  = \big\| ( [\varphi (z), \, d_{j, k}] )_{1 \leq j, k \leq n} \big\|
\\
& \leq \sum_{j, k = 1}^{n} \|[\varphi (z), d_{j, k}] \|
  < n^{2} \delta
  = \varepsilon.
\end{align*}
We have shown that $M_{n} (A)$ is tracially $\mathcal{Z}$-absorbing.
\end{proof}

Proposition~\ref{propmatrix}
also follows from Theorem~\ref{thmztz} below.

\begin{definition}\label{defloc}
We say that a C*-algebra $A$
is \emph{locally tracially $\mathcal{Z}$-absorbing}
if for any $\varepsilon > 0$
and any finite subset $F \subseteq A$
there exists a simple tracially $\mathcal{Z}$-absorbing
C*-subalgebra $B$ of $A$ such that $F \subseteq_{\varepsilon} B$,
that is, for any $a \in F$
there exists $b \in B$ such that $\|a - b\| < \varepsilon$.
\end{definition}

\begin{theorem}\label{thmloc}
Let $A$ be a simple
locally tracially $\mathcal{Z}$-absorbing C*-algebra.
Then $A$ is tracially $\mathcal{Z}$-absorbing.
\end{theorem}

\begin{proof}
It follows from Definition~\ref{defloc} that $A \ncong \mathbb{C}$.
Let  $x$, $a$, $F$, $\varepsilon$, and~$n$
be as in Definition~\ref{deftza}.
Write
$F = \{f_{1}, f_{2}, \ldots, f_{m} \}$.
We may assume that $\| a \| = 1$,
that $\| x \| \leq \frac{1}{4}$, and that $\varepsilon < 1$.
Set $E = F \cup \big\{ x^{1/2}, \, a^{1/2} \big\}$.
Choose $\delta > 0$ such that
\[
\delta < \tfrac{\varepsilon}{4},
\qquad
(2 + \delta) \delta < \tfrac{\varepsilon}{12},
\qquad {\mbox{and}} \qquad
\big( 2 \|a^{1/2} \| + \delta \big) \delta
 < \tfrac{\varepsilon}{12}.
\]
By assumption there is
a simple tracially $\mathcal{Z}$-absorbing C*-subalgebra $B$ of $A$
such that
$E \subseteq_{\delta} B$.

In particular,
there exists $b \in B$ such that $\| a^{1/2} - b \| < \delta$.
Then
\begin{align*}
\| b^{*} b - a \|
& \leq \| b^{*} b - b^{*} a^{1/2} \| + \|b^{*} a^{1/2} - a \|
\\
& \leq \| b \|\delta + \|a^{1/2} \| \delta
  \leq \big( \|a^{1/2} \| + \delta + \|a^{1/2} \| \big) \delta
  < \tfrac{\varepsilon}{12}.
\end{align*}
Set $d = (b^{*} b - \frac{\varepsilon}{12})_{+}$.
Then the estimate just done implies that
\begin{equation}\label{Eq_8531_dPrecsim}
d \precsim a.
\end{equation}
Also, $d \neq 0$
since $\| b^{*} b \| > 1 - \frac{\varepsilon}{12}$
and
$\varepsilon < 1$.

Similarly, there exists $y \in B$
such that $\| x^{1/2} - y \| < \delta$,
and the element $w = y^{*} y$
satisfies $\| w - x \| < \tfrac{\varepsilon}{12}$.
Since $\| x \| \leq \frac{1}{4}$
and $\tfrac{\varepsilon}{12} < \tfrac{1}{12}$,
we have $\|w\| < 1$.
Therefore
\begin{equation}\label{equ1_thmloc}
\| w^2 - x^2 \| < \tfrac{\varepsilon}{6}.
\end{equation}
Choose $e_{1}, e_{2}, \ldots, e_{m} \in B$
such that $\| e_{j} - f_{j} \| < \frac{\varepsilon}{4}$
for $j = 1, 2, \ldots, m $.
Since $B$ is tracially $\mathcal{Z}$-absorbing,
there exists a c.p.c.~order zero
map $\varphi \colon {M}_{n} \to B$ such that:
\begin{enumerate}
\item\label{thmloc_it1}
$\big( w^{2} - w \varphi (1) w - \frac{\varepsilon}{4} \big)_{+}
 \precsim d$.
\item\label{thmloc_it2}
$\| [\varphi (z), e_{j}] \| <  \frac{\varepsilon}{4}$ for
$j \in \{ 1, 2, \ldots, m \}$ and $z \in {M}_{n}$ with $\|z\| \leq 1$.
\end{enumerate}
Using~\eqref{equ1_thmloc}, $\| w \| \leq 1$,
and $\| x \| \leq \frac{1}{4}$ at the third step,
we then get
\begin{align*}
& \big\| \big( w^{2} - w \varphi (1) w
            - \tfrac{\varepsilon}{4} \big)_{+}
 - \big( x^{2} - x \varphi (1) x \big) \big\|
\\
& \hspace*{3em} {\mbox{}}
  \leq \big\| \big( w^{2} - w \varphi (1) w
               - \tfrac{\varepsilon}{4} \big)_{+}
         - w^{2} - w \varphi (1) w \big\|
\\
& \hspace*{6em} {\mbox{}}
    + \big\| \big( w^{2} - w \varphi (1) w \big)
         - \big( x^{2} - x \varphi (1) x \big) \big\|
\\
& \hspace*{3em} {\mbox{}}
  \leq
\tfrac{\varepsilon}{4} + \|w^{2} - x^{2} \|
 + \|w \varphi (1) w - w \varphi (1) x \|
 + \|w \varphi (1) x - x \varphi (1) x \|
\\
& \hspace*{3em} {\mbox{}}
  < \tfrac{\varepsilon}{4} + \tfrac{\varepsilon}{6} +
\tfrac{\varepsilon}{12} + \tfrac{\varepsilon}{12}
  < \varepsilon.
\end{align*}
Hence, using \eqref{thmloc_it1} at the second step
and~(\ref{Eq_8531_dPrecsim}) at the third step, we get
\[
\big( x^{2} - x \varphi (1) x - \varepsilon \big)_{+}
 \precsim \big( w^{2} - w \varphi (1) w
       - \tfrac{\varepsilon}{4} \big)_{+}
  \precsim d
  \precsim a.
\]
Finally, for any $z \in {M}_{n}$ with $\| z \| \leq 1$
and for $j \in \{ 1, 2, \ldots, m \}$,
by~\eqref{thmloc_it2} we have
\begin{align*}
\| [\varphi (z), f_{j}] \|
& \leq \|\varphi (z) f_{j} - \varphi (z) e_{j} \|
      + \|\varphi (z) e_{j} - e_{j} \varphi (z) \|
       + \| e_{j} \varphi (z) - f_{j} \varphi (z) \|
\\
& < \tfrac{\varepsilon}{4}
     + \tfrac{\varepsilon}{4} + \tfrac{\varepsilon}{4}
  < \varepsilon.
\end{align*}
So $A$ is tracially $\mathcal{Z}$-absorbing.
\end{proof}


\begin{corollary}\label{indlim}
Let $A$ be a C*-algebra  which is the direct limit
of a system of simple
tracially $\mathcal{Z}$-absorbing C*-algebras.
Then $A$ is also simple and tracially $\mathcal{Z}$-absorbing.
\end{corollary}

\begin{proof}
This follows immediately from Theorem~\ref{thmloc}.
\end{proof}

For C*-algebras with an approximate identity
(not necessarily increasing)
consisting of projections
we give a ``unital" equivalent statement for our
``nonunital" definition of tracial $\mathcal{Z}$-absorption.

\begin{proposition}\label{prop_apup}
Let $A$ be a simple C*-algebra
with an approximate identity $(p_{i})_{i \in I}$
(not necessarily increasing)
consisting of projections.
Then
$A$ is tracially $\mathcal{Z}$-absorbing if and only if
$p_{i} A p_{i}$ is tracially $\mathcal{Z}$-absorbing
in the unital sense
for each $i \in I$.
\end{proposition}

\begin{proof}
The statement follows from
Remark \ref{rmkdeftza}(\ref{it1_rmkdeftza}),
Theorem~\ref{thmher}, and Theorem~\ref{thmloc}.
\end{proof}

We can now generalize Lemma~\ref{P_8809_TAFtoTrZAbs}
to the nonunital case.

\begin{proposition}\label{P_8Y03_NonUTAFtoTrZAbs}
Let $A$ be an nonelementary simple
not necessarily C*-algebra
with tracial rank zero
(which is tracially AF; see
\cite[Definition 2.1]{Ln1}).
Then $A$ is tracially $\mathcal{Z}$-absorbing.
\end{proposition}

\begin{proof}
By \cite[Corollary~A.22(1)]{FG17},
the algebra $A$ has real rank zero.
Therefore $A$ has an approximate identity $(p_{i})_{i \in I}$
(not necessarily increasing)
consisting of projections.
(This also follows from \cite[Corollary 2.8]{Ln1}.)
The algebras $p_i A p_i$ have tracial rank zero
by \cite[Corollary~A.22(3)]{FG17}.
So they are tracially $\mathcal{Z}$-absorbing
by Lemma~\ref{P_8809_TAFtoTrZAbs}.
Now Proposition~\ref{prop_apup}
implies that $A$ is tracially $\mathcal{Z}$-absorbing.
\end{proof}

\begin{proposition}\label{propk}
Let $A$ be a simple C*-algebra.
Then $A$ is tracially $\mathcal{Z}$-absorbing
if and only if $A \otimes \mathcal{K}$
is tracially $\mathcal{Z}$-absorbing.
\end{proposition}

\begin{proof}
The forward implication follows from Proposition~\ref{propmatrix} and
Corollary~\ref{indlim}.
The converse follows from Theorem~\ref{thmher}
because $A$ is isomorphic
to a hereditary subalgebra of $A \otimes \mathcal{K}$.
%
\end{proof}

\begin{corollary}\label{cor_morita}
Let $A$ and $B$ be simple C*-algebras.
\begin{enumerate}[label=$\mathrm{(\arabic*)}$]
\item\label{cor_morita_it1}
If $A$ and $B$ are stably isomorphic,
and $A$ is tracially $\mathcal{Z}$-absorbing, then so is $B$.
\item\label{cor_morita_it2}
If $A$ and $B$ are $\sigma$-unital and Morita equivalent,
and $A$ is tracially $\mathcal{Z}$-absorbing, then so is $B$.
\end{enumerate}
\end{corollary}

\begin{proof}
Part \ref{cor_morita_it1} follows from Proposition~\ref{propk}.
For \ref{cor_morita_it2},
let $A$ and $B$ be $\sigma$-unital and Morita equivalent.
By \cite[Theorem~5.55]{RW98},
$A \otimes \mathcal{K} \cong B \otimes \mathcal{K}$.
Now the statement follows from \ref{cor_morita_it1}.
\end{proof}

It seems plausible that Corollary \ref{cor_morita}\ref{cor_morita_it2}
holds even if $A$ and $B$ are not $\sigma$-unital,
but more work is needed.

\section{$\mathcal{Z}$-absorption and tracial $\mathcal{Z}$-absorption}\label{sec_tzz}

\indent
In this section we compare tracial $\mathcal{Z}$-absorption with
$\mathcal{Z}$-absorption.
We show that $\mathcal{Z}$-absorption implies
tracial $\mathcal{Z}$-absorption,
that the converse is false,
but that the converse is true
when the algebra is separable and 
nuclear 
 \cite{CETWW}
 (see Remark~\ref{rmk_tzz}).

First we prove the following general result.

\begin{theorem}\label{thmztz}
Let $A$  be a simple tracially $\mathcal{Z}$-absorbing C*-algebra
and let $B$ be a simple C*-algebra.
Then $A \otimes_{\mathrm{min}} B$
is tracially $\mathcal{Z}$-absorbing.
\end{theorem}

\begin{proof}
Suppose we are given $\varepsilon > 0$,
a finite subset $F \subseteq A \otimes_{\mathrm{min}} B$,
$v, c \in ( A \otimes_{\mathrm{min}} B )_{+}$ with $c \neq 0$,
and $n \in {\mathbb{N}}$.
We have to find a c.p.c.~order zero map
$\varphi \colon {M}_{n} \to A \otimes_{\mathrm{min}} B$
such that the following hold:
\begin{enumerate}
\item\label{thmztz_itsmall}
$\left( v^{2} - v \varphi (1) v - \varepsilon \right)_{+} \precsim c$.
\item\label{thmztz_itcom}
$\| [\varphi (z), f] \| < \varepsilon$
for any $z \in {M}_{n}$
with $\| z \| \leq 1$ and any $f \in F$.
\setcounter{TmpEnumi}{\value{enumi}}
\end{enumerate}
We may assume that there are $m \in {\mathbb{N}}$
and
\[
r_1, r_2, \ldots, r_m \in A
\qquad {\mbox{and}} \qquad
s_1, s_2, \ldots, s_m \in B
\]
such that $\| r_j \| \leq 1$ and $\| s_j \| \leq 1$
for $j = 1, 2, \ldots, m$
and such that
$F = \big\{ r_{j} \otimes s_{j} \colon j = 1, 2, \ldots, m \big\}$.
By Kirchberg's Slice Lemma (\cite[Lemma~4.1.9]{Ro02}), there are
nonzero elements $a \in A_{+}$ and $b \in B_{+}$ such that
$a \otimes b \precsim c$ in  $A \otimes_{\mathrm{min}} B$.
Choose $\delta > 0$ such that
$\delta < \varepsilon / 3$.
Since $A \otimes_{\mathrm{min}} B$ has an approximate identity
consisting of elementary tensors
$x \otimes y$ with $x$ and $y$ positive and of norm at most one,
by Remark~\ref{rmkappu} we may assume that
$v = x \otimes y$ with $x \in A_{+}$, $y \in B_{+}$, and
$\| x \|, \, \| y \| \leq 1$.
By \cite[Proposition 2.7(v)]{KR00},
there is $k \in \mathbb{N}$ such that:
\begin{enumerate}
\setcounter{enumi}{\value{TmpEnumi}}
\item\label{thmztz_itkr}
$\left( y^{2} - \delta \right)_{+} \precsim b \otimes 1_{k}$
in ${M}_{\infty} (B)$.
\setcounter{TmpEnumi}{\value{enumi}}
\end{enumerate}
Since $A$ is simple and tracially $\mathcal{Z}$-absorbing,
Corollary~\ref{cornone} implies that
$A$ is not type~I.
Now Lemma~\ref{lemmatrix} implies that there is
$a_{0} \in A_{+} \setminus \{ 0 \}$ such that:
\begin{enumerate}
\setcounter{enumi}{\value{TmpEnumi}}
\item\label{thmztz_itph}
$a_{0} \otimes 1_{k} \precsim a$ in ${M}_{\infty} (A)$.
\setcounter{TmpEnumi}{\value{enumi}}
\end{enumerate}
By assumption $A$ is tracially $\mathcal{Z}$-absorbing.
Applying Definition~\ref{deftza}
with $\{ r_1, r_2, \ldots, r_m \}$ in place of $F$,
with $\delta$ in place of $\varepsilon$, with $n$ as given,
and with $a_{0}$ in place of $a$,
we obtain a c.p.c.~order zero map
$\varphi_{0} \colon {M}_{n} \to A$ such that the following hold:
\begin{enumerate}
\setcounter{enumi}{\value{TmpEnumi}}
\item\label{thmztz_itph01}
$\left( x^{2} - x \varphi_{0} (1) x - \delta \right)_{+}
 \precsim a_{0}$.
\item\label{thmztz_itph02}
$\| [\varphi_{0} (z), r_j] \| < \delta$ for $j = 1, 2, \ldots, m$
and any $z \in {M}_{n}$
with $\| z \| \leq 1$.
\setcounter{TmpEnumi}{\value{enumi}}
\end{enumerate}
Choose $e \in B_{+}$ such that
\begin{enumerate}
\setcounter{enumi}{\value{TmpEnumi}}
\item\label{thmztz_itappu}
$\| e \| \leq 1$,
$\| y e y - y^{2} \| < \delta$,
and $\| e s_{j} - s_{j} \| < \delta$
and $\| s_{j} e - s_{j} \| < \delta$ for $j = 1, 2, \ldots, m$.
\setcounter{TmpEnumi}{\value{enumi}}
\end{enumerate}
Now define $\varphi \colon  {M}_{n} \to A \otimes_{\mathrm{min}} B$ by
$\varphi (z) = \varphi_{0} (z) \otimes e$
for $z \in  {M}_{n}$.
Then $\varphi$ is a  c.p.c.~order zero map.

We show that \eqref{thmztz_itsmall} and
\eqref{thmztz_itcom} hold.
To prove~\eqref{thmztz_itsmall},
first recall that $v = x \otimes y$ at the first step
and use \eqref{thmztz_itappu}
at the last step to get
\begin{align*}
& \big\| \big( v^{2} - v \varphi (1) v \big)
     - \big( x^{2} - x \varphi_{0} (1) x - \delta \big)_{+}
           \otimes \big( y^{2} - \delta \big)_{+} \big\|
\\
& \hspace*{3em} {\mbox{}}
 \leq  \big\| x^{2} \otimes y^{2} - x \varphi_{0} (1) x \otimes y e y
    - \big( x^{2} - x \varphi_{0} (1) x \big) \otimes y^{2} \big\|
    + 2 \delta
\\
& \hspace*{3em} {\mbox{}}
  = \big\| x \varphi_{0} (1) x \otimes ( y^{2} - y e y ) \big\|
      + 2 \delta
\\
& \hspace*{3em} {\mbox{}}
 \leq \| y e y - y^{2} \| + 2 \delta
 < 3 \delta.
\end{align*}
Then use Lemma~\ref{lemkr} at the first step,
use \eqref{thmztz_itkr} and~\eqref{thmztz_itph01}
at the second step,
and use \eqref{thmztz_itph} at the fourth step, to get
\begin{align*}
\big( v^{2} - v \varphi (1) v - 3 \delta \big)_{+}
& \precsim \big( x^{2} - x \varphi_{0} (1) x - \delta \big)_{+}
\otimes \big( y^{2} - \delta \big)_{+}
\\
& \precsim a_{0} \otimes ( b \otimes 1_{k} )
  \sim ( a_{0} \otimes 1_{k} ) \otimes b
  \precsim a \otimes b
  \precsim c.
\end{align*}
Thus
\[
\big( v^{2} - v \varphi (1) v - \varepsilon \big)_{+}
  \precsim \big( v^{2} - v \varphi (1) v - 3 \delta \big)_{+}
  \precsim c.
\]

For \eqref{thmztz_itcom}, let $f \in F$.
Then there is $j \in \{ 1, 2, \ldots, m \}$
such that $f =  r_{j} \otimes s_{j}$.
Let $z \in {M}_{n}$
satisfy $\| z \| \leq 1$.
Using \eqref{thmztz_itph02} and~\eqref{thmztz_itappu}
at the fourth step,
we get
\begin{align*}
& \| [\varphi (z), f] \|
\\
& \hspace*{1em} {\mbox{}}
  = \big\| [\varphi_{0} (z) \otimes e, \, r_{j} \otimes s_{j} ] \big\|
\\
& \hspace*{1em} {\mbox{}}
  = \big\| \varphi_{0} (z) r_{j} \otimes e s_{j} -
               r_{j} \varphi_{0} (z) \otimes s_{j} e \big\|
\\
& \hspace*{1em} {\mbox{}}
  \leq \big\| [ \varphi_{0} (z) , r_{j}] \otimes e s_{j} \big\|
    + \big\| r_{j} \varphi_{0} (z) \otimes (e s_{j} - s_{j}) \big\|
    + \big\| r_{j} \varphi_{0} (z) \otimes (s_{j} - s_{j} e) \big\|
\\
& \hspace*{1em} {\mbox{}}
  < \delta + \delta + \delta
  < \varepsilon.
\end{align*}
This completes the proof.
\end{proof}

\begin{corollary}\label{corztz}
Let $A$ be a simple $\mathcal{Z}$-absorbing C*-algebra.
Then $A$ is tracially $\mathcal{Z}$-absorbing.
\end{corollary}

\begin{proof}
By definition,
$A \cong \mathcal{Z} \otimes A$.
By \cite[Proposition~2.2]{HO13},
$\mathcal{Z}$ is tracially $\mathcal{Z}$-absorbing.
Now Theorem~\ref{thmztz} implies that
$A$ is tracially $\mathcal{Z}$-absorbing.
\end{proof}

For example, the Razak-Jacelon algebra
$\mathcal{W}$ (see \cite{Jcln, Na19}) is 
a stably projectionless tracially $\mathcal{Z}$-absorbing
C*-algebra.

\begin{remark}\label{rmk_tzz}
In \cite{CETWW}, it is proved that every
simple separable nuclear tracially $\mathcal{Z}$-absorbing
C*-algebra $A$ is $\mathcal{Z}$-absorbing.
The special case that 
$A$  is not stably projectionless follows
from the corresponding result for the unital case
(\cite[Theorem~4.1]{HO13}),
the Morita invariance of $\mathcal{Z}$-absorption
in the class of separable C*-algebras 
(\cite[Corollary~3.2]{TW07}), and
 Theorems~\ref{thmher} and~\ref{thmztz}.
\end{remark}

We obtain a different proof of
a special case of \cite[Corollary~3.4]{TW07}.

\begin{corollary}\label{cor_ind_z}
Let $A$ be a simple separable nuclear C*-algebra
which is not stably projectionless
and which is the direct limit
of a system of simple
$\mathcal{Z}$-absorbing C*-algebras.
Then $A$ is $\mathcal{Z}$-absorbing.
\end{corollary}

\begin{proof}
The statement follows from  Corollaries~\ref{corztz} and \ref{indlim}
and Remark~\ref{rmk_tzz}.
\end{proof}

We now give examples of tracially $\mathcal{Z}$-absorbing C*-algebras
which are not $\mathcal{Z}$-absorbing.
Our examples are unital;
no such examples were previously known, even in the unital case.
We start with purely infinite examples.
By Example~\ref{exapi},
it is enough to find some
separable purely infinite simple unital C*-algebra
which is not $\mathcal{Z}$-absorbing.
The work for this is contained in Proposition~\ref{P_8624_NotZStab}.
It is based on the ideas
of \cite[Corollary 1.2]{DkmRdm},
in which it is shown that
an algebra~$A$ as in Proposition~\ref{P_8624_NotZStab}
is not approximately divisible.
The execution here is much messier because we can't use projections.
A stably finite example is then obtained by applying
Lemma~\ref{P_8809_TAFtoTrZAbs}
to an example in  \cite{NiuWng}.

\begin{proposition}\label{P_8624_NotZStab}
For $j = 1, 2$ let $P_j$ be a von Neumann algebra,
let $\om_j \colon P_j \to {\mathbb{C}}$ be a faithful normal state,
and let $G_j \subseteq P_j$ be a discrete
subgroup of the unitary group $U (P_j)$
which is contained in the centralizer of~$\om_j$
and
such that whenever $u, v \in G_j$
are distinct then $\om_j (v^* u) = 0$.
Suppose $G_1$ contains an element $a \neq 1$,
and $G_2$ contains distinct elements $b, c \neq 1$.
Let $(P, \om)$ be the reduced free product von Neumann algebra
$(P, \om) = (P_1, \om_1) \star_{\mathrm{r}} (P_2, \om_2)$,
equipped with its free product state~$\om$.
Let $A \subset P$ be any C*-subalgebra
satisfying $a, b, c \in A$.
Then $A$ is not ${\mathcal{Z}}$-stable.
\end{proposition}

For a von Neumann algebra $N$
and a state $\rh \colon N \to {\mathbb{C}}$,
a unitary $u \in N$ is in the centralizer of~$\rh$
if and only if $\rh \circ {\mathrm{Ad}} (u) = \rh$.

The algebras are called $(M_j, \ph_j)$
in~\cite{Brnt},
but this notation conflicts with notation for matrix algebras.

As discussed at the beginning of
\cite[Section 2]{Brnt},
the state $\om$ is necessarily faithful and normal.
By \cite[Theorem 2]{Brnt},
$P$ is a factor of type ${\mathrm{III}}$.

\begin{proof}[Proof of Proposition~\ref{P_8624_NotZStab}]
Let $D$ be the dimension drop interval algebra
\[
D =
\bigl\{ g \in C \big( [0, 1], \, M_2 \otimes M_3 \big) \colon
  {\mbox{$g (0) \in M_2 \otimes {\mathbb{C}} \cdot 1$
         and $g (1) \in {\mathbb{C}} \cdot 1 \otimes M_3$}} \bigr\}.
\]
It suffices to show that there is no
approximately central sequence of unital homomorphisms from
$D$ to~$A$.
To do this,
we set
\[
N = 24
\andeqn
\ep = \frac{1}{70}.
\]
We let $(e_{j, k})_{j, k = 1, 2}$
be the standard system of matrix units for~$M_2$,
and we let $(f_{j, k})_{j, k = 1, 2, 3}$
be the standard system of matrix units for~$M_3$.
For $n = 1, 2, \ldots, N$,
define $h_n \colon [0, 1] \to [0, 1]$ by
\[
h_n (\ld)
 = \begin{cases}
   0               & \hspace*{1em} 0 \leq \ld \leq \frac{n - 1}{N}
        \\
   N \ld - n + 1   & \hspace*{1em} \frac{n - 1}{N} < \ld < \frac{n}{N}
       \\
   1               & \hspace*{1em} \frac{n}{N} \leq \ld \leq 1.
\end{cases}
\]
Then define functions
$r_{j, n} \in D$ for $j = 1, 2$
and $s_{j, n} \in D$ for $j = 1, 2, 3$
by
\[
r_{j, n} (\ld) = (1 - h_n (\ld)) (e_{j, j} \otimes 1)
\andeqn
s_{j, n} (\ld) = h_n (\ld) (1 \otimes f_{j, j}).
\]
Next, define a compact set $S \subseteq D$
to consist of the functions defined for $\ld \in [0, 1]$
by
\[
\ld \mapsto (u \otimes 1) r_{1, n} (\ld) (u \otimes 1)^*
\]
for $u \in U (M_2)$ and $n = 1, 2, \ldots, N$,
and
\[
\ld \mapsto (1 \otimes v) s_{1, n} (\ld) (1 \otimes v)^*
\]
for $v \in U (M_3)$ and $n = 1, 2, \ldots, N$.

If there were
an approximately central sequence of unital homomorphisms from
$D$ to~$A$,
then,
in particular,
there would be a unital homomorphism
$\ph \colon D \to A$
such that
\begin{equation}\label{Eq_8625_Commutators}
\| [ \ph (g), a ] \| < \ep,
\qquad
\| [ \ph (g), b ] \| < \ep,
\andeqn
\| [ \ph (g), c ] \| < \ep
\end{equation}
for all $g \in S$.
We will show that this can't happen.

Suppose therefore that
$\ph \colon D \to A$
is a unital homomorphism
such that~\eqref{Eq_8625_Commutators} holds for all $g \in S$.
Regard $C ([0, 1])$ as a unital subalgebra of~$D$
in the obvious way.
Let $\mu$ be the Borel probability measure on $[0, 1]$
such that $(\om \circ \ph) (g) = \int_{[0, 1]} g \, d \mu$
for all $g \in C ([0, 1])$.
We have
\[
\sum_{n = 1}^N
 \mu \left( \left[ \frac{n - 1}{N}, \, \frac{n}{N} \right] \right)
\leq 2 \mu ([0, 1])
= 2,
\]
so there exists $n \in \{ 1, 2, \ldots, N \}$ such that,
if we set $I = \left[ \frac{n - 1}{N}, \, \frac{n}{N} \right]$,
then $\mu (I) \leq \frac{2}{N}$.
Set $\af = (\om \circ \ph) (h_n)$.

We have
$r_{1, n} + r_{2, n} = 1 - h_n$,
so
\[
(\om \circ \ph) (r_{1, n}) + (\om \circ \ph) (r_{2, n}) = 1 - \af.
\]
Therefore there are $j, k \in {1, 2}$
such that
\[
(\om \circ \ph) (r_{j, n}) \geq \frac{1 - \af}{2}
\andeqn
(\om \circ \ph) (r_{k, n}) \leq \frac{1 - \af}{2}.
\]
The set of functions
\[
\ld \mapsto (u \otimes 1) r_{1, n} (\ld) (u \otimes 1)^*,
\]
for $u \in U (M_2)$,
is connected
and contains both $r_{j, n}$ and $r_{k, n}$.
Therefore there is $u \in U (M_2)$
such that the function
\[
r (\ld) = (u \otimes 1) r_{1, n} (\ld) (u \otimes 1)^*
\]
satisfies $(\om \circ \ph) (r) = \frac{1}{2} (1 - \af)$.
A similar argument,
starting with
\[
(\om \circ \ph) (s_{1, n})
   + (\om \circ \ph) (s_{2, n})
   + (\om \circ \ph) (s_{3, n})
  = \af,
\]
produces $v \in U (M_3)$
such that the function
\[
s (\ld) = (1 \otimes v) s_{1, n} (\ld) (1 \otimes v)^*
\]
satisfies $(\om \circ \ph) (s) = \frac{\af}{3}$.
Then $(\om \circ \ph) (r + s) \in \big[ \frac{1}{3}, \frac{1}{2} \big]$.
It follows that
\begin{equation}\label{Eq_8625_ThirdHalfConseq}
(\om \circ \ph) (r + s) - \big[ (\om \circ \ph) (r + s) \big]^2
 \geq \frac{1}{6}.
\end{equation}

Define $l \in C ([0, 1])$ by
\[
l (\ld) = \big\| r (\ld) + s (\ld) - (r (\ld) + s (\ld))^2 \big\|
\]
for $\ld \in [0, 1]$.
We have $0 \leq r + s \leq 1$,
so $0 \leq (r + s)^2 \leq r + s$,
and
\[
0 \leq r + s - (r + s)^2 \leq l \leq 1.
\]
Also $l (\ld) = 0$ for $\ld \not\in I$.
Therefore
\begin{align*}
0
& \leq (\om \circ \ph) (r + s) - (\om \circ \ph) \big( (r + s)^2 \big)
\\
& \leq (\om \circ \ph) (l)
  = \int_{[0, 1]} l \, d \mu
  \leq \mu (I)
  \leq \frac{2}{N}
  = \frac{1}{12}.
\end{align*}
Combining this inequality with~\eqref{Eq_8625_ThirdHalfConseq}
gives
\begin{equation}\label{Eq_8625_Est12}
(\om \circ \ph) \big( (r + s)^2 \big)
 - \big[ (\om \circ \ph) (r + s) \big]^2
 \geq \frac{1}{12}.
\end{equation}

Recall that for $x \in P$,
we have $\| x \|_{\om} = \om (x^* x)^{1/2}$.
In particular, $\| x \|_{\om} \leq \| x \|$.
The proof of \cite[Theorem 11]{Brnt}
shows that for every $x \in P$ we have
\[
\| x - \om (x) \cdot 1 \|_{\om}
 \leq 14 \max \big( \| [x, a] \|_{\om}, \,
     \| [x, b] \|_{\om}, \, \| [x, c] \|_{\om} \big).
\]
If $0 \leq x \leq 1$,
a calculation shows that
\[
\| x - \om (x) \cdot 1 \|_{\om}^2
  = \om (x^2) - \om (x)^2.
\]
Put $x = r + s$
and use \eqref{Eq_8625_Est12}~at the first step
and \eqref{Eq_8625_Commutators}~at the third step
to get
\begin{align*}
\frac{1}{12}
& \leq 14^2 \max \big( \| [x, a] \|_{\om}^2, \,
     \| [x, b] \|_{\om}^2, \, \| [x, c] \|_{\om}^2 \big)
\\
& \leq 14^2 \max \big( \| [x, a] \|^2, \,
     \| [x, b] \|^2, \, \| [x, c] \|^2 \big)
  < 14^2 \ep^2
  = \frac{1}{25},
\end{align*}
a contradiction.
\end{proof}

\begin{example}\label{E_6825_PINotZSt}
There is a simple separable unital purely infinite C*-algebra
which is not ${\mathcal{Z}}$-stable.

We start with the example at the end of~\cite{Brnt},
which gives
$(P_1, \om_1)$ and $(P_2, \om_2)$
satisfying the hypotheses of Proposition~\ref{P_8624_NotZStab}.
Let $(P, \om)$ and $a, b, c \in P$ be as there.

Since, by \cite[Theorem 2]{Brnt},
$P$ is a factor of type ${\mathrm{III}}$,
we can apply \cite[Proposition 1.3(i)]{DkmRdm}
(whose method goes at least back to \cite[Proposition 2.2]{Blkd1})
to find
a purely infinite simple separable C*-algebra $A \subseteq P$
such that $a, b, c \in A$.
Then $A$ is not ${\mathcal{Z}}$-stable
by Proposition~\ref{P_8624_NotZStab}.
\end{example}

In fact,
one can explicitly write down a purely infinite simple separable
reduced free product algebra which is not ${\mathcal{Z}}$-stable.

\begin{example}\label{E_8702_PINotZStable}
Let $A$ be
the reduced free product
$A = (M_2 \otimes M_2) \star_{\mathrm{r}} C ([0, 1])$,
taken with respect to the Lebesgue measure state on $C ([0, 1])$
and the state on $M_2 \otimes M_2$
given by tensor product of the usual tracial state ${\mathrm{tr}}$
with the state
$\rho (x) = {\operatorname{tr}} \big(
 {\mathrm{diag}} \big( \frac{1}{3}, \frac{2}{3} \big) x \big)$
on $M_2$.
We claim that $A$ is purely infinite and simple
but not ${\mathcal{Z}}$-stable.

To prove pure infiniteness,
in \cite[Example 3.9(iii)]{Dkm2}
take $A_1 = M_2$
with the state ${\mathrm{tr}}$,
take $F = M_2$
with the state~$\rh$,
and take $B = C ([0, 1])$
with the state given by Lebesgue measure.
These choices satisfy the hypotheses there.
So $A$ is is purely infinite and simple.

To prove failure of ${\mathcal{Z}}$-stability,
in Proposition~\ref{P_8624_NotZStab}
take $P_1 = M_2 \otimes M_2$
with the state ${\operatorname{tr}} \otimes \rh$,
and take $P_2 = L^{\infty} ([0, 1])$
with the state given by Lebesgue measure.
Take
$G_1 = \left\{ 1, \,
 \left( \begin{smallmatrix} 1 & 0 \\ 0 & -1 \end{smallmatrix} \right)
   \otimes 1 \right\}$,
and take $G_2$ to be the set of functions
$\ld \mapsto e^{2 \pi i n \ld}$ for $n \in \Z$.
These choices
satisfy the hypotheses there.
Moreover, $G_2 \subseteq C ([0, 1])$.
Therefore the reduced free product
$A = (M_2 \otimes M_2) \star_{\mathrm{r}} C ([0, 1])$
is a subalgebra of the algebra~$P$ of Proposition~\ref{P_8624_NotZStab}
which contains $a$, $b$, and~$c$,
so is not ${\mathcal{Z}}$-stable by Proposition~\ref{P_8624_NotZStab}.
\end{example}

One might hope that if $A$ is
a purely infinite simple separable C*-algebra,
then the conditions of ${\mathcal{O}}_{\infty}$-stability,
${\mathcal{Z}}$-stability,
and approximate divisibility are all equivalent.
This seems too good to be true.

\begin{example}\label{Ex_8817_WZ}
There is a simple separable stably finite unital C*-algebra
which is tracially $\mathcal{Z}$-absorbing
but not $\mathcal{Z}$-absorbing.
Start with the example in~\cite{NiuWng}
of a simple separable stably finite unital C*-algebra
which has tracial rank zero \cite[Definition 3.6.2]{LnBook}
but is not $\mathcal{Z}$-absorbing.
This algebra is tracially $\mathcal{Z}$-absorbing
by lemma~\ref{P_8809_TAFtoTrZAbs}.
\end{example}

\section{The Cuntz semigroup}\label{sec_cu}

\indent
In this section we generalize Theorem~3.3 of \cite{HO13}
to the nonunital case.
More precisely, we show that
if $A$ is a simple tracially $\mathcal{Z}$-absorbing C*-algebra
then $W (A)$ is
almost unperforated.
This implies that $A$ has strict comparison
in a sense suitable for nonunital C*-algebras.
(See Definition \ref{D_8604_StrComp} below.)
The proof is a modification of the argument given in \cite{HO13}.
Moreover, we show that $A$ is weakly almost divisible,
as in Definition~\ref{DEF-WAD}.

\begin{definition}\label{wpp}
Let $A$ be a C*-algebra and let $a \in {M_{\infty} (A)}_{+}$.
We say that $a$ is \emph{weakly purely positive}
if 0 is  an accumulation point of $\mathrm{sp} (a)$.
\end{definition}

\begin{remark}\label{R_8601_Wpp}
Recall (see before \cite[Corollary 2.24]{APT11})
that $a \in {M_{\infty} (A)}_{+}$ is called purely positive
if there is no projection $p \in {M_{\infty} (A)}_{+}$ such that
$\langle p \rangle = \langle a \rangle$ in $W (A)$.
By \cite[Proposition 2.23]{APT11},
if $A$ is a unital C*-algebra with $\mathrm{tsr} (A) = 1$
then $a \in {M_{\infty} (A)}_{+}$
is weakly purely positive if and only if $a$ is purely positive.
In general, pure positivity implies weak pure positivity.
We don't know whether the converse holds.
\end{remark}

The following lemma is a nonunital analog of \cite[Lemma~3.2]{HO13}.

\begin{lemma}\label{lemwpp}
Let $A$ be a simple tracially $\mathcal{Z}$-absorbing C*-algebra.
Let $a, b \in A_{+}$.
Suppose $b$ is weakly purely positive
and there is a positive integer~$n$
such that $n \langle a \rangle \leq n \langle b \rangle$
in $W (A)$.
Then $a \precsim b$.
\end{lemma}

Some parts of the proof are similar
to the proof of \cite[Lemma 3.2]{HO13}.
(See \cite[Notation 2.13]{HO13} for some of the
notation used in that proof.)
We provide a complete proof because of some differences
due to the absence of
the identity.
(Also, in the proof of
\cite[Lemma~3.2]{HO13},
we could not follow one step in the proof of
the estimate
$\| \widehat{c} (b - \delta)_{+} \widehat{c}^{*} - a_{1} \|
  < \frac{\mu}{3}$.)

\begin{proof}[Proof of Lemma~\ref{lemwpp}]
We may assume that $\| a \| = \|  b \| = 1$.
Let $\varepsilon > 0$;
we prove that $(a - \varepsilon)_{+} \precsim b$.

The assumption says $a \otimes 1_{n} \precsim b \otimes 1_{n}$.
By \cite[Proposition 2.6]{KR00},
there is $\dt > 0$ such that
\[
\big( a \otimes 1_{n} - \tfrac{\varepsilon}{2} \big)_{+}
 \precsim ( b \otimes 1_{n} - \dt )_{+}.
\]
Choose $c_0 \in M_n (A)$
such that
\[
\big\| c_0 \big[ (b \otimes 1_{n} - \dt)_{+} \big] c_0^*
  - \big( a \otimes 1_{n} - \tfrac{\varepsilon}{2} \big)_{+} \big\|
 < \tfrac{\varepsilon}{2}.
\]
It follows from Lemma~\ref{lemkr} that
there is $c_1 \in M_n (A)$
such that
\[
c_1 c_0 \big[ (b \otimes 1_{n} - \dt)_{+} \big] c_0^* c_1^*
  = \big( \big( a \otimes 1_{n} - \tfrac{\varepsilon}{2} \big)_{+}
        - \tfrac{\varepsilon}{2} \big)_{+}.
\]
Thus,
\[
c_1 c_0 \big[ (b - \delta)_{+} \otimes 1_{n} \big] (c_1 c_0)^{*}
  = (a - \varepsilon)_{+} \otimes 1_{n}.
\]

Let $f, h \in C_{0} ((0 , 1])$
be nonnegative functions such that
$f = 0$ on $\big[ \frac{\delta}{2}, 1 \big]$,
$f > 0$ on $\big( 0, \frac{\delta}{2} \big)$,
and $\| f \| = 1$,
and such that
$h = 0$ on $\big[ 0, \frac{\delta}{2} \big]$
and $h = 1$ on $[\dt, 1]$.
Put $d = f (b)$.
Since $b$ is weakly purely positive, $d \neq 0$.
Put $c = c_1 c_0 [h (b) \otimes 1_{n}]$.
Since $h (\ld) (\ld - \delta)_{+} = (\ld - \delta)_{+}$
for $\ld \in [0, 1]$
and $h f = 0$,
we get
\[
c \big[ (b - \delta)_{+} \otimes 1_{n} \big] c^{*}
  = (a - \varepsilon)_{+} \otimes 1_{n}
\andeqn
c (d \otimes 1_{n}) = 0.
\]
Write $c = (c_{j, k})_{1 \leq j, k \leq n}$ with
$c_{j, k} \in A$ for $j, k = 1, 2, \ldots, n$.
Then
\begin{equation}\label{Eq_8701_4Star}
c_{j, k} d = 0.
\end{equation}

For any $\mu > 0$ we will find
$z \in A$ such that
\begin{equation}\label{equwpp1}
\big\| z \big[ (b - \delta)_{+} + d \big] z^{*}
        - (a - \varepsilon)_{+} \big\|
  < \mu.
\end{equation}
This will prove the first step in the calculation
\[
(a - \varepsilon)_{+} \precsim (b - \delta)_{+} + d \precsim b;
\]
since the second step is clear,
and $\varepsilon > 0$ is arbitrary,
we will have finished the proof.

So fix $\mu > 0$.
We have
\begin{equation}\label{Eq_8701_cSum}
\sum_{l = 1}^{n} c_{j, l} (b - \delta)_{+} c_{k, l}^{*}
 =
\begin{cases}
(a - \varepsilon)_{+}  & \hspace*{1em} j = k
\\
0                      & \hspace*{1em} j \neq k.
\end{cases}
\end{equation}

Similarly to the proof of \cite[Lemma~3.1]{HO13}
define $g, h \in C_{0} ((0, 1])$ by
\[
g (\lambda) = \begin{cases}
\sqrt{ 13 \lambda / \mu} & \hspace*{1em} \lambda < \mu / 13
\\
1                        & \hspace*{1em} \lambda \geq \mu / 13
\end{cases}
\andeqn
h (\lambda) = 1 - \sqrt{1 - \lambda}
\]
for $\lambda \in [0, 1]$.
Then
\begin{equation}\label{Eq_8701_g2}
| g (\lambda)^{2} \lambda - \lambda | < \tfrac{\mu}{12}
\andeqn
1 - h (\lambda) = \sqrt{1 - \lambda}
\end{equation}
for all $\lambda \in [0, 1]$.

Put
\[
F = \big\{ (a - \varepsilon)_{+} \big\}
   \cup \big\{ c_{j, k} (b - \delta)_{+} c^{*}_{l, m}
          \colon j, k, l, m = 1, 2, \ldots, n \big\}.
\]

At this point,
we need functional calculus for c.p.c.~order zero maps
for a function $k \in C_0 ((0, 1])$
(\cite[Corollary~4.2]{WZ09}).
We follow the notation there,
writing $k (\ph)$
(and then $\ph^{1/2}$),
rather than $k [\varphi]$
(and $\sqrt{[\varphi]}$)
as in \cite[Notation 2.13]{HO13}.
We briefly recall its construction,
slightly simplified since the domain is unital.
By \cite[Theorem~3.3]{WZ09}),
there is a unital homomorphism $\pi$ from $M_n$ to
the multiplier algebra of the C*-algebra
generated by the range of~$\ph$
such that $\ph (w) = \ph (1) \pi (w) = \pi (w) \ph (1)$
for all $w \in M_n$.
Then $k (\ph) (w) = k (\ph (1)) \pi (w)$
for all $w \in M_n$.
We need two consequences:
\begin{equation}\label{Eq_8701_Comm}
k (\ph) (w) \cdot \ph (1) = \ph (1) \cdot k (\ph) (w)
\end{equation}
for any $k \in C_0 ((0, 1])$ and $w \in M_n$,
and
\begin{equation}\label{Eq_8701_Prod}
k (\ph) (w_1) \cdot k (\ph) (w_2) = k (\ph (1)) \cdot k (\ph) (w_1 w_2)
\end{equation}
for any $f \in C_0 ((0, 1])$ and $w_1, w_2 \in M_n$.

For any C*-algebras $B$ and~$A$,
any map $\ph \colon B \to A$,
and any subset $F \subseteq A$,
we define
\[
C (\ph, F)
 = \sup \big( \bigl\{ \| [\varphi (w), y] \| \colon
   {\mbox{$w \in B$ satisfies $\| w \| \leq 1$ and $y \in F$}}
  \bigr\} \big).
\]
It follows from \cite[Lemma~2.8]{HO13}
that there is $\eta_{0} > 0$ such that
if $\varphi \colon  M_{n} \to A$ is a c.p.c.~order zero map with
$C (\ph, F) < \eta_{0}$, then
\begin{equation}\label{equwpp2}
C ( g (\varphi), F )  < \frac{\mu}{24n^{4}},
\qquad
C ( \varphi^{1/2}, F ) < \frac{\mu}{24n^{4}},
\qquad {\mbox{and}} \qquad
C ( h (\varphi), F ) < \frac{\mu}{6}.
\end{equation}

Use \cite[Lemma 2.5]{AP16} to choose $\eta > 0$
such that $\eta < \min \big( \eta_{0}, \tfrac{\mu}{30} \big)$
and such that whenever
$e, s \in A_{+}$
satisfy
\[
\| e \| \leq 1,
\qquad
\| s \| \leq 1,
\qquad {\mbox{and}} \qquad
\| e s - s e \| < 9 \eta,
\]
then
$\big\| e^{1/2} s - s e^{1/2} \big\| < \mu / 6$.
Choose $x \in A_{+}$ such that
\begin{equation}\label{Eq_8701_2Star}
\| x \| \leq 1
\qquad {\mbox{and}} \qquad
\big\| x (a - \varepsilon)_{+} - (a - \varepsilon)_{+} \big\| < \eta.
\end{equation}

Applying Definition~\ref{deftza},
with $\eta$ in place of $\varepsilon$,
and with $n$, $F$, and $x$ as given, we obtain
a c.p.c.~order zero map $\varphi \colon M_{n} \to A$ such that
\begin{equation}\label{Eq_8817_NewNew}
C ( \varphi, F) < \eta
\qquad {\mbox{and}} \qquad
\big( x^{2} - x \varphi (1) x - \eta \big)_{+} \precsim d.
\end{equation}
Put
\begin{equation}\label{Eq_8701_1Star}
r_0 = 1 - \varphi (1)
\qquad {\mbox{and}} \qquad
r = \big( x^{2} - x \varphi (1) x - \eta \big)_{+}
  = (x r_0 x - \eta)_{+}.
\end{equation}
%
Then $r_0 \in A^{\sim}$ and $r \in A$.
Set
\[
a_{1} = \varphi (1) (a - \varepsilon)_{+},
\ \ \
a_{2} = r_0^{1/2} (a - \varepsilon)_{+} r_0^{1/2},
\ \ \
{\mbox{and}}
\ \ \
a_{3} = r^{1/2} (a - \varepsilon)_{+} r^{1/2}.
\]
Since
\begin{equation}\label{Eq_8701_3Star}
\big\| \big[ r_0, \, (a - \varepsilon)_{+} \big] \big\|
 = \big\| \big[ 1 - \ph (1), \, (a - \varepsilon)_{+} \big] \big\|
 = \big\| \big[ \ph (1), \, (a - \varepsilon)_{+} \big] \big\|
 < \eta,
\end{equation}
the choice of $\eta$ implies
\begin{equation}\label{equwpp3}
\big\| (a - \varepsilon)_{+} - (a_{1} + a_{2}) \big\| < \frac{\mu}{6}.
\end{equation}
For $j, k = 1, 2, \ldots, n$,
let $e_{j, k} \in M_n$
be the standard matrix unit,
and set
\begin{equation}\label{Eq_8701_5Star}
g_{j, k} = g (\varphi) (e_{j, k})
\qquad {\mbox{and}} \qquad
\widehat{c}_{j, k} = \varphi^{1/2} (1) g_{j, k} c_{j, k}.
\end{equation}
Then set
$\widehat{c} = \sum_{j, k = 1}^{n} \widehat{c}_{j, k}$.

We claim that
\begin{equation}\label{Eq_8701_6Star}
\big\| \widehat{c} (b - \delta)_{+} \widehat{c}^{*} - a_{1} \big\|
   < \frac{\mu}{6}.
\end{equation}
We follow the proof of \cite[Lemma~3.1]{HO13}.
One checks,
following the steps there and using
\eqref{Eq_8701_Prod} and~\eqref{Eq_8701_cSum},
that
%
\begin{equation}\label{Eq_8701_jklm}
\sum_{j, k, l, m = 1}^{n} g_{j, k} g_{l, m} c_{j, k}
          (b - \delta)_{+} c_{m, l}^{*}
 = g (\ph) (1)^2 \cdot (a - \varepsilon)_{+}.
\end{equation}
Now,
using \eqref{Eq_8701_5Star} at the second step,
using \eqref{equwpp2} at the third step,
using \eqref{equwpp2} and~\eqref{Eq_8701_Comm} at the fourth step,
using \eqref{Eq_8701_jklm} at the fifth step,
and using \eqref{Eq_8701_g2} at the sixth step,
we have
\begin{align*}
\widehat{c} (b - \delta)_{+} \widehat{c}^{*}
& = \sum_{j, k, l, m = 1}^{n}
  \widehat{c}_{j, k} (b - \delta)_{+} \widehat{c}^{*}_{m, l}
\\
& = \varphi^{1/2} (1) \left( \sum_{j, k, l, m = 1}^{n}
    g_{j, k} c_{j, k} (b - \delta)_{+} c_{m, l}^{*} g_{l, m} \right)
\varphi^{1/2} (1)
\\
& \approx_{\mu / 24} \varphi^{1/2} (1)
   \left( \sum_{j, k, l, m = 1}^{n} g_{j, k} g_{l, m} c_{j, k}
          (b - \delta)_{+} c_{m, l}^{*} \right)
     \varphi^{1/2} (1)
\\
& \approx_{\mu / 24} \varphi (1)
     \left(\sum_{j, k, l, m = 1}^{n} g_{j, k} g_{l, m} c_{j, k}
          (b - \delta)_{+} c_{m, l}^{*} \right)
\\
& = \varphi (1) \cdot g (\ph) (1)^2 \cdot (a - \varepsilon)_{+}
\\
& \approx_{\mu / 12} \varphi (1) (a - \varepsilon)_{+}
\\
& = a_{1}.
\end{align*}
This proves~\eqref{Eq_8701_6Star}.

We claim that
\begin{equation}\label{equwpp4}
\| a_{2} - a_{3} \| < \frac{\mu}{2}.
\end{equation}
First,
using \eqref{Eq_8701_1Star} at the first step,
using \eqref{Eq_8701_2Star} at the second and fourth steps,
and using \eqref{Eq_8701_3Star}
and (\ref{Eq_8701_1Star}) at the third and fifth steps,
we have
\begin{align*}
r (a -\varepsilon)_{+}
& \approx_{\eta} x r_0 x (a - \varepsilon)_{+}
  \approx_{\eta} x r_0 (a - \varepsilon)_{+}
\\
& \approx_{\eta} x (a - \varepsilon)_{+} r_0
  \approx_{\eta} (a - \varepsilon)_{+} r_0
  \approx_{\eta} r_0 (a - \varepsilon)_{+}.
\end{align*}
Thus
\[
\big\| r (a - \varepsilon)_{+} - r_0 (a - \varepsilon)_{+} \big\|
  < 5 \eta
\andeqn
\big\| r (a - \varepsilon)_{+} - (a - \varepsilon)_{+} r_0 \big\|
  < 4 \eta.
\]
Taking adjoints in the second inequality and combining gives
\[
\big\| r (a - \varepsilon)_{+} - (a - \varepsilon)_{+} r \big\|
  < 9 \eta.
\]
By the choice of $\eta$,
and using \eqref{Eq_8701_3Star} for the first inequality,
\[
\big\| r_0^{1/2} (a - \varepsilon)_{+}
        - (a - \varepsilon)_{+} r_0^{1/2} \big\|
  < \frac{\mu}{6}
\quad {\mbox{and}} \quad
\big\| r^{1/2} (a - \varepsilon)_{+}
        - (a - \varepsilon)_{+} r^{1/2} \big\|
  < \frac{\mu}{6}.
\]
Now
\begin{align*}
\| a_{2} - a_{3} \|
& = \big\| r_0^{1/2} (a - \varepsilon)_{+}r_0^{1/2}
     - r^{1/2} (a - \varepsilon)_{+} r^{1/2} \big\|
\\
& \leq \frac{\mu}{6} + \frac{\mu}{6}
 + \big\| r_0 (a - \varepsilon)_{+} - r (a - \varepsilon)_{+} \big\|
\\
& < \frac{\mu}{3} + 5 \eta
  < \frac{\mu}{3} + \frac{\mu}{6}
  = \frac{\mu}{2}.
\end{align*}
This proves \eqref{equwpp4}.

Using (\ref{Eq_8817_NewNew})
and~(\ref{Eq_8701_1Star}) for the second step,
we get $a_{3} \precsim r \precsim d$.
Choose $s \in A$ such that
$\| s d s^{*} - a_{3} \| < \frac{\mu}{6}$.
Since $d = f (b)$,
as in the proof of \cite[Lemma~3.1]{HO13},
replacing $s$ with $s k (b)$
for a function $k \in C ([0, 1])$
which is $1$ on $\big[ 0, \frac{\delta}{2} \big]$
and vanishes on $[\delta, 1]$,
we may assume that $s (b - \delta)_{+} = 0$.
Set $z =  \widehat{c} + s$.
Using \eqref{Eq_8701_4Star} and \eqref{Eq_8701_5Star}
at the first step,
and
using \eqref{Eq_8701_6Star}, \eqref{equwpp4},
and \eqref{equwpp3} at the third step,
we get
\begin{align*}
& \big\| z[(b - \delta)_{+} + d]z^{*} - (a - \varepsilon)_{+} \big\|
\\
& \hspace*{1em} {\mbox{}}
 = \big\| \widehat{c} (b - \delta)_{+} \widehat{c}^{*}
      + s d s^{*} - (a - \varepsilon)_{+} \big\|
\\
& \hspace*{1em} {\mbox{}}
 \leq \big\| \widehat{c} (b - \delta)_{+} \widehat{c}^{*} - a_{1} \big\|
      + \| s d s^{*} - a_{3} \|
      + \| a_{3} - a_{2} \|
      + \| a_{1} + a_{2} - (a - \varepsilon)_{+} \|
\\
& \hspace*{1em} {\mbox{}}
 < \frac{\mu}{6} + \frac{\mu}{6} + \frac{\mu}{2} + \frac{\mu}{6}
 = \mu.
\end{align*}
This proves \eqref{equwpp1} and completes the proof.
\end{proof}

The following theorem is the nonunital generalization
of \cite[Theorem 3.1]{HO13}.
The proof is essentially the same.

\begin{theorem}\label{thm_almup}
Let A be a simple tracially $\mathcal{Z}$-absorbing C*-algebra.
Then $W (A)$ is
almost unperforated.
\end{theorem}

\begin{proof}
Let $a, b \in M_{\infty} (A)_{+}$.
Suppose there is $k \geq 2$
such that $k \langle a \rangle \leq (k - 1) \langle b \rangle$.
We have to show that $a \precsim b$.
Choose $n \in \mathbb{N}$ such that
$a, b \in M_{n} (A)_{+}$.
Since $M_{n} (A)$ is tracially $\mathcal{Z}$-absorbing
(by Proposition~\ref{propmatrix}), we may assume that $a, b \in A$.
If $b$ is weakly purely positive then
Lemma~\ref{lemwpp} implies that $a \precsim b$.
If $b$ is not weakly purely positive
then $b$ is not purely positive.
So there exists a projection $p \in A$
such that $\langle p \rangle = \langle b \rangle$.
We may assume that $p \neq 0$.
Since $A$ is not type~I (by Corollary~\ref{cornone}),
we can apply \cite[Lemma~3.1]{HO13}
to obtain a
weakly purely positive element $c \in A$ such
that
$(k - 1) \langle p \rangle \leq k \langle c \rangle$ and $c \leq p$.
Thus
$k \langle a \rangle
 \leq (k - 1) \langle p \rangle
 \leq k \langle c \rangle$.
Therefore, by Lemma~\ref{lemwpp}, $a \precsim c \leq p \sim b$.
\end{proof}

For strict comparison,
we need dimension functions.
We give careful statements since there is conflicting
terminology in the literature.

\begin{definition}\label{D_8605_DimFcn}
Let $A$ be a C*-algebra.
A {\emph{dimension function}} $d$ on $A$ is
an additive order preserving function
$d \colon W (A) \to [0, \infty]$.
It is equivalent to give a function
$d \colon M_{\infty} (A)_{+} \to [0, \infty]$
such that $d (a \oplus b) = d (a) + d (b)$
and $a \precsim b$ implies $d (a) \leq d (b)$,
and we use the same letter for both versions
of a dimension function.
The set of all dimension functions on $A$
is denoted by $\mathrm{DF} (A)$.
\end{definition}

This definition is given after \cite[Theorem 4.5]{Ro04}.
We warn that at the beginning of \cite[Section 4]{Ro91},
the algebra $A$ is assumed unital
and it is required that $d ( \langle 1_A \rangle ) = 1$;
a state there is
an additive order preserving function
from $W (A)$ to $[0, \infty)$.
In \cite[Definition I.1.2]{BlkHdm},
the algebra $A$ is not assumed unital,
but it is required that
$\sup_{a \in A_{+}} d (a) = 1$.
(In \cite[Definition I.1.2]{BlkHdm}, a dimension function $d$
is actually taken to be defined on $M_{\infty} (A)$,
but to satisfy $d (a) = d (a^* a)$ for all $a \in M_{\infty} (A)$.)

The following is also from after \cite[Theorem 4.5]{Ro04},
except that the definition of lower semicontunity
is from the beginning of \cite[Section 4]{Ro91}.

\begin{definition}\label{D_8605_lsc}
Let $A$ be a C*-algebra,
and let $d \colon M_{\infty} (A)_{+} \to [0, \infty]$
be a dimension function.
We say that $d$ is {\emph{lower semicontinuous}}
if whenever $(a_n)_{n \in {\mathbb{N}}}$
is a sequence in $M_{\infty} (A)_{+}$ and $\| a_n - a \| \to 0$,
then $d (a) \leq \liminf_{n \to \infty} d (a_n)$.
\end{definition}

\begin{lemma}[Proposition 4.1 of~\cite{Ro91}]\label{L_8605_ToLDF}
Let $A$ be a C*-algebra,
and let $d \colon M_{\infty} (A)_{+} \to [0, \infty]$
be a dimension function.
Define ${\overline{d}} \colon M_{\infty} (A)_{+} \to [0, \infty]$
by
\[
{\overline{d}} (a)
 = \lim_{\varepsilon \to 0^{+}} d \big( (a - \varepsilon)_{+} \big)
\]
for $a \in M_{\infty} (A)_{+}$.
Then ${\overline{d}}$ is lower semicontinuous,
and ${\overline{d}} = d$ if and only if
$d$ is lower semicontinuous.
\end{lemma}

\begin{proof}
The proof of \cite[Proposition 4.1]{Ro91}
works just as well in the present situation,
without normalization and with values allowed to be
in $[0, \infty]$.
\end{proof}

We have not found the following definition in the literature
in the nonunital case.
(For example, it isn't in~\cite{BRTTW}.)

\begin{definition}\label{D_8604_StrComp}
Let $A$ be a C*-algebra.
We say that $A$ has
{\emph{weak strict comparison}}
if whenever $a, b \in M_{\infty} (A)_{+}$
satisfy $a \in {\overline{A b A}}$
and $d (a) < d (b)$ for every
dimension function $d$ on~$A$ with $d (b) = 1$,
then $a \precsim b$.
\end{definition}

Purely infinite simple C*-algebras
have weak strict comparison,
since then $a \in {\overline{A b A}}$ implies $a \precsim b$.

For simple unital C*-algebras,
we show that our definition reduces to strict comparison.

\begin{proposition}\label{P_8604_UnitalCase}
Let $A$ be a simple unital C*-algebra
which has a quasitrace.
Then $A$ has strict comparison
(defined using quasitraces)
if and only if $A$ has weak strict comparison.
\end{proposition}

As usual,
we take quasitraces to be normalized $2$-quasitraces,
and we denote the set of quasitraces on~$A$
by ${\mathrm{QT}} (A)$.

Some authors define strict comparison using only tracial states.
Without knowing that every quasitrace is a trace,
this might be a different concept.

\begin{proof}[Proof of Proposition~\ref{P_8604_UnitalCase}]
Observe that if $b \in M_{\infty} (A)_{+} \setminus \{ 0 \}$,
then $d (b) > 0$
for every nonzero
dimension function $d$ on~$A$.
Indeed,
suppose $d (b) = 0$ and $b \in M_n (A)_{+}$.
Then $\bigl\{ a \in M_n (A) \colon d (a^* a) = 0 \bigr\}$
is a proper ideal in $M_n (A)$
(a priori not necessarily closed).
But $M_n (A)$ has no nontrivial ideals,
closed or not,
so $b = 0$.
It follows that $d_{\ta} (b) > 0$
for all $\ta \in {\mathrm{QT}} (A)$.

Suppose $A$ has strict comparison.
Let $a, b \in M_{\infty} (A)_{+}$
satisfy $a \in {\overline{A b A}}$
and $d (a) < d (b)$ for every
dimension function $d$ on~$A$ with $d (b) = 1$.
If $b = 0$ then $a = 0$,
so $a \precsim b$.
Otherwise,
let $\ta \in {\mathrm{QT}} (A)$.
Then $d = d_{\ta} (b)^{-1} d_{\ta}$
is a dimension function on~$A$ with $d (b) = 1$.
So $d (a) < d (b)$ by hypothesis,
and it follows that $d_{\ta} (a) < d_{\ta} (b)$.
Since $\ta$ is arbitrary,
strict comparison implies $a \precsim b$.

Conversely,
assume that $A$ has weak strict comparison.
We first claim that
$W (A)$ is almost unperforated.
Let $a, b \in M_{\infty} (A)_{+}$,
and suppose there are are $m, n \in {\mathbb{N}}$
such that $n \langle a \rangle \leq m \langle b \rangle$
and $n > m$.
We want to show that $a \precsim b$.
(This version of almost unperforation is clearly equivalent
to the usual,
and is \cite[Definition 3.1]{Ro04}.)
If $b = 0$ then $a = 0$,
so $a \precsim b$.
Otherwise, for any dimension function $d$ on~$A$ with $d (b) = 1$,
we have
$n d (a) \leq m d (b) < n d (b)$,
so $d (a) < d (b)$.
Then $a \precsim b$ by weak strict comparison.

The rest is essentially contained in the proof
of \cite[Theorem 5.2(a)]{Ro91}.
Let $a, b \in M_{\infty} (A)_{+}$
satisfy $d_{\tau} (a) < d_{\tau} (b)$
for all $\ta \in {\mathrm{QT}} (A)$.
Clearly $b \neq 0$.
Fix $\varepsilon > 0$;
we claim that
for all dimension functions $d$ on~$A$ such that
$d (b) = 1$,
we have $d \big( (a - \varepsilon)_{+} \big) < d (b)$.

We prove the claim.
Since $A$ is simple and unital,
for every $c \in M_{\infty} (A)_{+}$
there is $n$ such that $\langle c \rangle \leq n \langle b \rangle$.
This shows that the values of $d$ are in $[0, \infty)$,
that is, $d$ is a state on $W (A)$
in the sense described at the beginning of \cite[Section 3]{Ro91}.
So the dimension function ${\overline{d}}$
of Lemma~\ref{D_8605_lsc}
also has values in $[0, \infty)$.
It is nonzero because ${\overline{d}} (1) = d (1)$.
Theorem II.2.2 of~\cite{BlkHdm}
provides $\ta \in {\mathrm{QT}} (A)$
such that $d (1)^{-1} {\overline{d}} = d_{\tau}$.
Now, using the assumption $d_{\tau} (a) < d_{\tau} (b)$
at the third step,
\[
d \big( (a - \varepsilon)_{+} \big)
 \leq {\overline{d}} (a)
 = d (1) d_{\tau} (a)
 < d (1) d_{\tau} (b)
 = {\overline{d}} (b)
 \leq d (b).
\]
The claim is proved.

Combining the claim and \cite[Proposition 3.1]{Ro91},
we get $(a - \varepsilon)_{+} \precsim b$.
Since this is true for all $\varepsilon > 0$,
we have $a \precsim b$,
as desired.
\end{proof}

Definition~\ref{D_8604_StrComp} was chosen because
it is the conclusion of \cite[Corollary 4.7]{Ro04},
and
gives the same answer as the usual definition
in the simple unital case.

More direct analogs of strict comparison
seem to be unsuitable.
As a heuristic example,
consider $A = C (S^2 \times S^2) \otimes {\mathcal{K}}$.
The algebra $A_0 = C (S^2 \times S^2)$ doesn't have strict comparison,
because there is a rank~$2$ projection $q_1 \in M_{\infty} (A_0)$
which is not Murray-von Neumann subequivalent to
the trivial rank~$3$ projection $q_2 = 1_{M_{3}} \otimes 1$.
So $A$ shouldn't have strict comparison either.
It presumably would if one restricted to dimension functions with
finite values.
Indeed,
define
\[
x = {\mathrm{diag}} \big( 1, \tfrac{1}{2}, \tfrac{1}{3}, \ldots \big)
  \in \mathcal{K}
\andeqn
e = {\mathrm{diag}} \big( 1, 0, 0, \ldots \big)
  \in \mathcal{K}.
\]
Any dimension function $d$ on~$A$
has $d (1 \otimes x) \geq n d (1 \otimes e)$
for every $n \in {\mathbb{N}}$.
So if $d (1 \otimes x)$ is finite then $d (1 \otimes e) = 0$.
Similar reasoning shows that in fact if
$d (1 \otimes x)$ is finite
then $d (p) = 0$ for every projection $p \in A$.
So $d (q_1) = d (q_2) = 0$
for every dimension function with
finite values,
and the standard example to show lack of strict comparison fails.

It doesn't help to restrict to
lower semicontinuous dimension functions
with finite values.

It seems, therefore,
that infinite values must be allowed.
But now the dimension function
given by $d (0) = 0$
and $d (a) = \infty$
for all nonzero $a \in M_{\infty} (A)_{+}$
gives $d (q_1) = d (q_2) = \infty$,
so again the standard example to show lack of strict comparison fails.
This dimension function is lower semicontinuous,
so again it doesn't help to
restrict to lower semicontinuous
dimension functions.

This discussion leaves several other possibilities,
which we do not address.
\begin{enumerate}
\item\label{Item_8618_AddLSC}
In Definition~\ref{D_8604_StrComp} use only
lower semicontinuous dimension functions.
\item\label{Item_8618_FinOnPed}
Use all dimension functions which are finite
on the Pedersen ideal of $A \otimes {\mathcal{K}}$.
\item\label{Item_8618_FinOnPedLSC}
In~(\ref{Item_8618_FinOnPed}) use only
lower semicontinuous dimension functions.
\end{enumerate}
%

Returning to the main development, from Theorem~\ref{thm_almup}
we get the following consequence.

\begin{proposition}\label{rmk_sc}
Let $A$ be a simple (not necessarily unital)
tracially $\mathcal{Z}$-absorbing C*-algebra.
Then $A$ has weak strict comparison.
\end{proposition}

\begin{proof}
Let $a, b \in M_{\infty} (A)_{+}$
satisfy $a \in {\overline{A b A}}$
and $d (a) < d (b)$ for every
dimension function $d$ on~$A$ with $d (b) = 1$.
Choose $n$ such that $a, b \in M_n (A)$.
Then $M_n (A)$ is tracially $\mathcal{Z}$-absorbing
by Proposition~\ref{propmatrix}.
By Theorem~\ref{thm_almup}, $W (M_n (A))$ is
almost unperforated.
Now apply \cite[Corollary~4.7]{Ro04},
noting that this result
(unlike some other results in~\cite{Ro04})
does not require the algebra to be unital.
\end{proof}

\begin{definition}[\cite{OPW17}, Definition~3.1]\label{DEF-WAD}
Let $A$ be a C*-algebra.
We say that
$A$ is {\emph{weakly almost divisible}}
if for any $a \in M_{\infty} (A)_{+}$,
any $n \in \mathbb{N}$, and any $\varepsilon > 0$,
there is $\eta \in W (A)$ such that $n \eta \le \langle a \rangle$
and $\langle (a - \varepsilon)_{+} \rangle \le (n + 1) \eta$.
\end{definition}

Weak almost divisibility
is a weak version of divisibility.
It seems to be needed when one wants to prove that
the crossed product by an action with the tracial Rokhlin property
preserves strict comparison
in the absence of tracial $\mathcal{Z}$-stability.
See~\cite{OPW17}.

\begin{proposition}\label{propwad}
Let $A$ be a simple C*-algebra.
If $A$ is tracially $\mathcal{Z}$-absorbing,
then $A$ is weakly almost divisible.
\end{proposition}

\begin{proof}
Let $a \in M_{\infty} (A)_{+}$,
$n \in \N$, and $\varepsilon > 0$ be as in Definition~\ref{DEF-WAD}.
We have to find $b \in M_{\infty} (A)_{+}$
such that
\[
n \langle b \rangle \leq \langle a \rangle
\qquad {\mbox{and}} \qquad
\big\langle (a - \varepsilon)_{+} \big\rangle
   \leq (n + 1) \langle b \rangle.
\]
We may assume that $(a - \varepsilon)_{+} \neq 0$.
Since $M_m (A)$ is also tracially $\mathcal{Z}$-absorbing
for any $m \in \N$
(Proposition~\ref{propmatrix}),
we may assume that $a \in A$.
Use \cite[Lemma~2.1]{Ph14}
to find $z \in A_{+} \setminus \{ 0 \}$ such that
\begin{equation}\label{ineq_tza_weak_almost_divisibility}
(n + 2) \langle z \rangle
 \leq \big\langle (a - \varepsilon)_{+} \big\rangle.
\end{equation}
Choose $z_0 \in B_{+} \setminus \{ 0 \}$
such that $z_0 \precsim z$ in~$A$.
By Theorem~\ref{thmeqtz},
the subalgebra $B = \overline{aAa}$ of $A$
is tracially $\mathcal{Z}$-absorbing.
Thus, there exists a c.p.c.~order zero map
$\psi \colon M_{n} \to B$ such that
\begin{equation}\label{propwad_equ6}
\big( a - a^{1/2} \psi (1) a^{1/2} - \tfrac{\varepsilon}{2} \big)_{+}
 \precsim z_0.
\end{equation}

Let $t \in C_0 ((0, 1])$ denote the identity function,
and use \cite[Corollary 4.1]{WZ09}
to find a homomorphism
$\rho \colon C_0 ((0, 1]) \otimes M_{n} \to B$
satisfying $\rho (t \otimes x) = \psi (x)$
for all $x \in M_n$.
For $j, k = 1, 2, \ldots, n$,
let $e_{j, k} \in M_n$
be the standard matrix unit.
Set $b = \psi (e_{1, 1})$.
We have
\begin{equation}\label{Eq_6911_CSub}
\psi (1)
 = \sum_{j = 1}^n \psi (e_{j, j})
 \sim 1_{n} \otimes  \psi (e_{1, 1})
 = 1_{n} \otimes b.
\end{equation}
So $n \langle b \rangle = \langle \psi (1) \rangle$.
Since $\psi (1) \in B = {\overline{a A a}}$,
it follows that $n \langle b \rangle \leq \langle a \rangle$.

Now we show that
$\langle (a - \varepsilon)_{+} \rangle \le (n + 1) \langle b \rangle$.
Using \cite[Lemma~1.5]{Ph14} in the first step,
\eqref{propwad_equ6} in the second step,
and \eqref{Eq_6911_CSub} in the third step,
we get
\begin{align*}
(a - \varepsilon)_{+}
& \precsim \big( a^{1/2} \psi (1) a^{1/2}
    - \tfrac{\varepsilon}{2} \big)_{+}
   \oplus \big(a - a^{1/2} \psi (1) a^{1/2}
    - \tfrac{\varepsilon}{2} \big)_{+}
\\
& \precsim \psi (1) \oplus z_0
  \precsim (1_{n} \otimes b) \oplus z.
\end{align*}
Hence
\begin{align}\label{eq_tza_weak_almost_divisibility}
\langle (a - \varepsilon)_{+} \rangle
\leq n \langle b \rangle + \langle z \rangle.
\end{align}

It remains to show that $\langle z \rangle \leq \langle b \rangle$.
Since $A$ is simple and $b \neq 0$,
we certainly have $z \in {\overline{A b A}}$.
Since $A$ has weak strict comparison (by Proposition~\ref{rmk_sc}),
it now suffices to show that
$d (z) < d (b)$ for all $d \in \mathrm{DF} (A)$
with $d (b) = 1$.
If $d (z) = 0$, we are done.
Otherwise, by (4.1) at
the beginning of \cite[Section 4]{Ro91},
there is $m \in {\mathbb{N}}$
such that
$\big\langle (a - \varepsilon)_{+} \big\rangle
 \leq m \langle b \rangle$.
Then
\[
d (z)
 \leq d \big( (a - \varepsilon)_{+} \big)
 \leq m d (b)
 = m.
\]
That is, $d (z)$ is finite.
By \eqref{ineq_tza_weak_almost_divisibility}
and \eqref{eq_tza_weak_almost_divisibility} we have
\[
n d (z)
 < (n + 2) d (z) - d (z)
 \leq d \big( (a - \varepsilon)_{+} \big) - d (z)
 \leq n d (b).
\]
We have shown that $d (z) < d (b)$
for all $d \in \mathrm{DF} (A)$ with $d (b) = 1$,
as desired.
\end{proof}

\section{Tracial $\mathcal{Z}$-absorption for finite C*-algebras}\label{Sec_Finite}

\indent
The objective of this section
is to prove Proposition~\ref{P_8813_TZA_Fin},
which states that
if $A$ is
a finite simple tracially $\mathcal{Z}$-absorbing C*-algebra,
then in Definition~\ref{deftza}
we can assume $\| a \| = 1$
and require $\| \ph (1) a \ph (1) \| > 1 - \ep$.
This is parallel to the definition
of the weak tracial Rokhlin property for finite group actions.

We don't actually use Proposition~\ref{P_8813_TZA_Fin}.
However, almost all the work which goes into it is
needed for the proof of a result in
\cite{AGJP17}
that the permutation action
on the minimal tensor product of finitely many copies
of a tracially $\mathcal{Z}$-absorbing C*-algebra
has the weak tracial Rokhlin property,
provided that the tensor product is finite.

It isn't entirely clear what the ``right'' definition of
finiteness is for nonunital C*-algebras,
even simple ones.
We will use the following definition,
which seems to be the most common.

\begin{definition}[V.2.2.1 in \cite{BlEC}]\label{D_8813_Finite}
A nonunital C*-algebra $A$ is {\emph{finite}}
if its unitization $A^{+}$ is finite in the usual sense,
that is, $1$ is a finite projection in~$A^{+}$.
Otherwise, $A$ is {\emph{infinite}}.
\end{definition}

Equivalently,
all one sided invertible elements in~$A^{+}$
are two sided invertible.

There is further discussion of finiteness
in \cite[Sections V.2.2 and V.2.3]{BlEC}.
For our further discussion,
we need the following definition and notation.

\begin{definition}[Definition 3.2 in \cite{KR00}]\label{D_8813_FinElt}
Let $A$ be a C*-algebra and let $ a\in A_{+}$.
Then $a$ is {\emph{infinite}}
if there exists $b \in A_{+} \setminus \{ 0 \}$
such that $a \oplus b \precsim a$.
Otherwise, $a$ is {\emph{finite}}.
\end{definition}

\begin{notation}\label{N_8813_PedIdeal}
For any C*-algebra~$A$,
we let $\Ped (A)$ denote the Pedersen ideal
(smallest dense ideal) in~$A$;
see \cite[Section 5.6]{Pdsn1}.
\end{notation}

The following elementary fact will be used often enough
that we record it here.

\begin{lemma}\label{L_8816_CutToPed}
Let $A$ be a C*-algebra, let $a \in A_{+}$, and let $\ep > 0$.
Then $(a - \ep)_{+} \in \Ped (A)$.
\end{lemma}

\begin{proof}
This is immediate from the proof of \cite[Theorem 5.6.1]{Pdsn1}.
\end{proof}

For a simple nonunital C*-algebra~$A$,
we can consider at least four further conditions,
none of which appears in~\cite{BlEC}:
\begin{enumerate}
\item\label{Item_8810_FiniteCond_Ped}
Every positive element in $\Ped (A)$
(see Notation~\ref{N_8813_PedIdeal})
is finite.
\item\label{Item_8810_FiniteCond_Dense}
The finite elements of $A_{+}$ are dense in $A_{+}$.
\item\label{Item_8810_FiniteCond_Cut}
For every $a \in A_{+}$ and $\ep > 0$,
the element $(a - \ep)_{+}$ is finite.
\item\label{Item_8810_FiniteCond_DimFcn}
There are a dimension function
(see Definition~\ref{D_8605_DimFcn}) $d$ on~$A$
and a positive element $a$ in $\Ped (A)$
such that $0 < d (a) < \infty$.
\end{enumerate}
It isn't reasonable to require that all positive elements
of~$A$ be finite,
since it follows from \cite[Proposition 3.7]{KR00}
that every nonzero separable stable C*-algebra
contains infinite elements.

Condition~\eqref{Item_8810_FiniteCond_Cut}
is the condition we actually use.

Conditions \eqref{Item_8810_FiniteCond_Ped},
\eqref{Item_8810_FiniteCond_Dense},
and \eqref{Item_8810_FiniteCond_Cut}
are very similar.
Apart from the obvious implications,
from \eqref{Item_8810_FiniteCond_Ped}
to \eqref{Item_8810_FiniteCond_Cut}
(because $(a - \ep)_{+} \in \Ped (A)$ by Lemma~\ref{L_8816_CutToPed})
and from \eqref{Item_8810_FiniteCond_Cut}
to \eqref{Item_8810_FiniteCond_Dense},
we have not investigated the relations between these
three conditions.
Proposition~\ref{L_8810_FinEltToUnit} below
shows that finiteness as in Definition~\ref{D_8813_Finite}
implies~\eqref{Item_8810_FiniteCond_Cut}.
Corollary~\ref{C_8813_DfToFinite} below
shows that Condition~\eqref{Item_8810_FiniteCond_DimFcn}
implies Condition~\eqref{Item_8810_FiniteCond_Cut}.
Definition~\ref{D_8813_Finite} has the advantage that
important permanence properties are immediate.

\begin{remark}\label{R_8813_PermFin}
Finiteness as in Definition~\ref{D_8813_Finite} has the
following properties.
\begin{enumerate}
\item\label{Item_R_8813_PermFin_SubH}
All subhomogeneous C*-algebras are finite.
\item\label{Item_R_8813_PermFin_LimInj}
Direct limits of direct systems of finite C*-algebras with injective
maps
are finite.
\item\label{Item_R_8813_PermFin_LimSubh}
Arbitrary direct limits of subhomogeneous C*-algebras are finite.
\item\label{Item_R_8813_PermFin_Subalg}
Subalgebras of finite C*-algebras are finite.
\end{enumerate}
\end{remark}

\begin{lemma}\label{L_8808_DimFcnOnSimp}
Let $A$ be a simple C*-algebra,
and let $d$ be a dimension function on~$A$
(Definition~\ref{D_8605_DimFcn}).
\begin{enumerate}
\item\label{Item_8808_Vanish}
If there is $a \in \Ped (A)_{+} \setminus \{ 0 \}$
such that $d (a) = 0$,
then $d (b) = 0$ for all $b \in \Ped (A)_{+}$.
\item\label{Item_8808_Finite}
If there is $a \in \Ped (A)_{+} \setminus \{ 0 \}$
such that $d (a)$ is finite,
then $d (b)$ is finite for all $b \in \Ped (A)_{+}$.
\end{enumerate}
\end{lemma}

\begin{proof}
Both the sets
\[
\bigl\{ a \in A \colon d (a^* a) = 0 \bigr\}
\andeqn
\bigl\{ a \in A \colon d (a^* a) < \infty \bigr\}
\]
are easily seen to be
(not necessarily closed)
ideals in~$A$.
Therefore each is either equal to $\{ 0 \}$
or contains $\Ped (A)$.
\end{proof}

\begin{corollary}\label{C_8813_DfToFinite}
Let $A$ be a simple C*-algebra.
Assume that there are a dimension function
(see Definition~\ref{D_8605_DimFcn})
$d$ on~$A$
and $a \in \Ped (A)_{+}$
such that $0 < d (a) < \infty$.
Then all positive elements of $\Ped (A)$
are finite.
\end{corollary}

\begin{proof}
Suppose there is an infinite element
$a \in \Ped (A)_{+}$.
By definition, there is $b \in A_{+} \setminus \{ 0 \}$
such that $a \oplus b \precsim a$.

Let $d$ be a dimension function on~$A$
as in the statement.
Then $d (a) + d (b) \leq d (a)$.
So $d (b) = 0$ or $d (a) = \infty$.
If $d (a) = \infty$
then $d (x) = \infty$ for all $x \in \Ped (A)_{+} \setminus \{ 0 \}$
by Lemma \ref{L_8808_DimFcnOnSimp}(\ref{Item_8808_Finite}),
contradicting the hypothesis of the statement. So $d(a)<\infty$ and 
 $d (b) = 0$.
Choose $\ep > 0$ such that $\ep < \| b \|$.
Then $(b - \ep)_{+} \neq 0$,
and is in $\Ped (A)$ by Lemma~\ref{L_8816_CutToPed}.
Therefore,
$d\left( (b - \ep)_{+} \right)\leq d(b)=0$, which implies that
 $d (x) = 0$ for all $x \in \Ped (A)_{+}$
by Lemma \ref{L_8808_DimFcnOnSimp}(\ref{Item_8808_Vanish}),
contradicting the hypothesis of the statement.
\end{proof}

\begin{proposition}\label{L_8810_FinEltToUnit}
Let $A$ be a simple  C*-algebra.
Suppose that there are $a \in A_{+}$ and $\ep > 0$
such that $(a - \ep)_{+}$ is infinite.
Then $A$ is infinite
in the sense of Definition~\ref{D_8813_Finite}. 
\end{proposition}

Although it wasn't convenient to write the proof this way,
having chosen $b \in A_{+} \setminus \{ 0 \}$
which is orthogonal to $(a - \ep)_{+}$,
heuristically the key point is that
$(a - \frac{\ep}{3})_{+}$ is in the algebraic
ideal generated by~$b$. 

This proposition should also be true in the unital case,
but we don't need it there.

\begin{proof}[Proof of Proposition~\ref{L_8810_FinEltToUnit}]
Replacing $a$ by $\| a \|^{-1} a$ and $\ep$ by $\| a \|^{-1} \ep$,
we can assume that $\| a \| = 1$.
Then clearly $\ep < 1$.

Define continuous functions $f_1, f_2 \colon [0, 1] \to [0, 1]$
by
\[
f_1 (\ld)
 = \begin{cases}
   0                   & \hspace*{0.5em} 0 \leq \ld \leq \frac{2 \ep}{3}
        \\
   3 \ep^{-1} \ld - 2  & \hspace*{0.5em} \frac{2 \ep}{3} < \ld < \ep
       \\
   1                   & \hspace*{0.5em} \ep \leq \ld \leq 1
\end{cases}
\quad {\mbox{and}} \quad
f_2 (\ld)
 = \begin{cases}
   0          & \hspace*{0.5em} 0 \leq \ld \leq \frac{\ep}{3}
        \\
   3 \ep^{-1} \ld - 1
              & \hspace*{0.5em} \frac{\ep}{3} < \ld < \frac{2 \ep}{3}
       \\
   1          & \hspace*{0.5em} \frac{2 \ep}{3} \leq \ld \leq 1
\end{cases}
\]
Define
\[
a_0 = (a - \ep)_{+},
\qquad
e_1 = f_1 (a),
\andeqn
e_2 = f_2 (a).
\]
Then $e_1 a_0 = a_0$.
By assumption,
there is $c \in A_{+} \setminus \{ 0 \}$
such that $a_0 \oplus c \precsim a_0$.
Since $A$ is not unital,
${\overline{(1 - e_2) A (1 - e_2)}}$
is a nonzero hereditary subalgebra of~$A$.
By \cite[Lemma 2.4]{Ph14},
there exists a positive element
$c_0 \in {\overline{(1 - e_2) A (1 - e_2)}}$
such that $\| c_0 \| = 1$ and $c_0 \precsim c$.
Then set $b = f_2 (c_0)$ and $b_0 = f_1 (c_0)$.
So $b_0 \neq 0$
and $b b_0 = b_0$.
Also $b \precsim c_0$,
so $a_0 + b \precsim a_0$.
Note that $a_0 c_0=0$, and 
so $a_0 b=0$.

Since $A$ is simple,
\cite[Lemma 1.13]{Ph14} provides
$n \in \N$ and $y_1, y_2, \ldots, y_n \in A$
such that
$\left\| e_2 - \sum_{j = 1}^n y_j b y_j^* \right\| < \frac{1}{2}$.
In $W (A)$ we then have
\[
\langle e_1 \rangle
  \leq \bigl\langle \bigl( e_2 - \tfrac{1}{2} \bigr)_{+} \bigr\rangle
  \leq \left\langle \sum_{j = 1}^n y_j b y_j^* \right\rangle
  \leq \sum_{j = 1}^n \langle y_j b y_j^* \rangle
  \leq n \langle b \rangle
  \leq \langle a_0 \rangle + n \langle b \rangle.
\]
It follows from an induction argument that
$\langle a_0 \rangle + (n + 1) \langle b \rangle
 \leq \langle a_0 \rangle$,
so $e_1 + b \precsim a_0$.
Therefore there is $v \in A$ such that
\begin{equation}\label{Eq_8811_CSE}
\| v a_0 v^* - (e_1 + b) \| < \frac{1}{2}.
\end{equation}
Define
\[
s = v a_0^{1/2} + (1 - e_1 - b)^{1/2}.
\]

Note that if $A$ is unital then
$s\in A$, and if $A$ is nonunital then
$s\in A^{+}$.

Since $e_1 a_0 = a_0$ and $b a_0 = 0$, we have
\[
(1 - e_1 - b) a_0 = a_0 (1 - e_1 - b) = 0,
\]
and it follows (for example, by polynomial approximation) that
\begin{equation}\label{Eq_8813_ProdZ}
(1 - e_1 - b)^{1/2} a_0^{1/2} = a_0^{1/2} (1 - e_1 - b)^{1/2} = 0.
\end{equation}

We claim that $s$ is right invertible.
To prove this,
we show that $\| 1 - s s^* \| < \frac{1}{2}$.
Now
\[
s s^*
 = v a_0^{1/2} a_0^{1/2} v^*
   + (1 - e_1 - b)
   + v a_0^{1/2} (1 - e_1 - b)^{1/2}
   + (1 - e_1 - b)^{1/2} a_0^{1/2} v^*.
\]
The last two terms are zero by~\eqref{Eq_8813_ProdZ}.
The first is equal to $v a_0 v^*$.
So the relation $\| 1 - s s^* \| < \frac{1}{2}$
follows from~\eqref{Eq_8811_CSE}.
The claim is proved.

Now we claim that $s b_0 = 0$.
Since $e_1 b_0 = 0$ and $b b_0 = b_0$,
we have $(1 - e_1 - b) b_0 = 0$,
so $(1 - e_1 - b)^{1/2} b_0 = 0$.
Using $a_0 b_0 = 0$, so $a_0^{1/2} b_0 = 0$, we now get
$s b_0 = 0$,
as desired.

It follows from the last claim that $s$ is not two sided invertible,
completing the proof.
\end{proof}

\begin{lemma}\label{L_8810_xhxTohxh}
For every $\ep > 0$ there is $\dt > 0$ such that whenever
$A$ is a C*-algebra,
and $x, h \in A_{+}$ satisfy $\| x \| \leq 1$,
$\| h \| \leq 1$, and $\| x h x \| > 1 - \dt$,
then $\| h x h \| > 1 - \ep$.
\end{lemma}

\begin{proof}
The statement is trivial if $\ep > 1$,
so assume $\ep \leq 1$.
Set $\dt = 1 - (1 - \ep)^{1/2} > 0$.
Now let $A$, $x$, and $h$
be as in the hypotheses.
Using selfadjointness of $x h x$ at the second step
and $x^2 \leq x$ at the last step,
we have
\[
1 - \ep
  = (1 - \dt)^2
  < \| (x h x)^2 \|
  = \| x h x^2 h x \|
  \leq \| h x^2 h \|
  \leq \| h x h \|.
\]
This completes the proof.
\end{proof}

The following lemma generalizes \cite[Lemma 1.8]{Ph14}.

\begin{lemma}\label{L_8810_Gen1Point8}
For every $\ep > 0$ there is $\dt > 0$ such that whenever
$A$ is a C*-algebra,
$\af, \bt \in [0, \infty)$,
and $r, x, h \in A_{+}$ satisfy
\[
\| r \| \leq 1,
\qquad
\| x \| \leq 1,
\qquad
\| h \| \leq 1,
\andeqn
\| r x - x \| < \dt,
\]
then
\[
\bigl( r x r - [ \af + \bt + \ep ] \bigr)_{+}
 \precsim \bigl( h^{1/2} x h^{1/2} - \af \bigr)_{+}
      \oplus (r^2 - r h r - \bt)_{+}.
\]
\end{lemma}

\begin{proof}
Use \cite[Lemma~2.5.11]{LnBook} 
to choose $\dt > 0$ such that whenever $A$ is a C*-algebra
and $r, x \in A_{+}$ satisfy $\| r \| \leq 1$, $\| x \| \leq 1$,
and $\| r x - x \| < \dt$,
then $\bigl\| r x^{1/2} - x^{1/2} \bigr\| < \frac{\ep}{2}$.

Now let $A$, $\af$, $\bt$, $r$, $x$, and $h$
be as in the hypotheses.
We have
\[
\bigl\| x^{1/2} r h r x^{1/2} - x^{1/2} h x^{1/2} \bigr\|
 \leq 2 \bigl\| r x^{1/2} - x^{1/2} \bigr\|
 < \ep,
\]
so, by \cite[Lemma 1.6]{Ph14},
\begin{equation}\label{Eq_8811_Star18}
\bigl( x^{1/2} r h r x^{1/2} - [\af + \ep] \bigr)_{+}
  \precsim \bigl( x^{1/2} h x^{1/2} - \af \bigr)_{+}.
\end{equation} 
Now,
using \cite[Proposition 2.3(ii)]{ERS}
at the first and fourth steps,
\cite[Lemma 1.5]{Ph14} at the second step,
and
\eqref{Eq_8811_Star18} and Lemma~\ref{lemkey} at the third step,
we have
\begin{align*}
\bigl( r x r - [ \af + \bt + \ep ] \bigr)_{+}
& \sim \bigl( x^{1/2} r^2 x^{1/2} - [ \af + \bt + \ep ] \bigr)_{+}
\\
& \precsim \bigl( x^{1/2} r h r x^{1/2} - [\af + \ep] \bigr)_{+}
      \oplus \bigl( x^{1/2} (r^2 - r h r) x^{1/2} - \bt \bigr)_{+}
\\
& \precsim \bigl( x^{1/2} h x^{1/2} - \af \bigr)_{+}
      \oplus ( r^2 - r h r - \bt)_{+}
\\
& \sim \bigl( h^{1/2} x h^{1/2} - \af \bigr)_{+}
      \oplus ( r^2 - r h r - \bt)_{+}.
\end{align*}
This completes the proof.
\end{proof}

The following lemma is a nonunital version of \cite[Lemma 2.9]{Ph14}.

\begin{lemma}\label{L_8810_Nonunital2Point6}
Let $A$ be a simple C*-algebra which is not of type~I,
let $r \in A_{+}$ be a finite element,
and let $x \in \bigl( {\overline{r A r}} \bigr)_{+}$
satisfy $\| x \| = 1$.
Then for every $\ep > 0$
there are $\dt > 0$
and $y \in \bigl( {\overline{x A x}} \bigr)_{+} \setminus \{ 0 \}$
such that whenever $h \in A_{+}$ satisfies
$\| h \| \leq 1$ and $( r^2 - r h r - \dt )_{+} \precsim y$,
then $\| h x h \| > 1 - \ep$.
\end{lemma}

\begin{proof}
By Lemma~\ref{L_8810_xhxTohxh},
and changing the value of~$\ep$,
it suffices to prove this with the conclusion
$\| x h x \| > 1 - \ep$
instead of $\| h x h \| > 1 - \ep$.

If $\ep > 1$,
there is nothing to prove.
So assume $\ep \leq 1$.

Set $\ep_0 = \frac{\ep}{15}$.
Apply Lemma~\ref{L_8810_Gen1Point8}
with $\ep_0$ in place of~$\ep$,
getting $\rh > 0$
(called $\dt$ there).

For $\af, \bt \in [0, \infty)$ with $\af < \bt$,
define a continuous function
$f_{\af, \bt} \colon [0, \infty) \to [0, 1]$
by
\[
f_{\af, \bt} (\ld)
 = \begin{cases}
   0                           & \hspace*{1em} 0 \leq \ld \leq \af
        \\
 (\bt - \af)^{-1} (\ld - \af)  & \hspace*{1em} \af < \ld < \bt
       \\
   1                           & \hspace*{1em} \bt \leq \ld.
\end{cases}
\]
The elements $f_{1/n, 2 / n} (r)$,
for $n \in \N$,
form an approximate identity for ${\overline{r A r}}$.
Choose $n \in \N$ so large that
$\frac{2}{n} < \ep_0$
and the element $r_0 = f_{1/n, 2 / n} (r)$
satisfies $\| r_0 x - x \| < \min \bigl( \frac{\rh}{3}, \ep_0 \bigr)$.
Define a continuous function $g \colon [0, \infty) \to [0, \infty)$
by
\[
g (\ld)
 = \begin{cases}
   \frac{n^2 \ld}{4} & \hspace*{1em} 0 \leq \ld \leq \frac{2}{n}
        \\
   \frac{1}{\ld} & \hspace*{1em} \frac{2}{n} < \ld.
\end{cases}
\]
Set $s = g (r)$.
Then $\| s \| \leq \frac{n}{2}$.
We have
\[
\| s r - r_0 \| \leq \frac{2}{n} < \ep_0
\andeqn
\| s r \| \leq 1.
\]
So
\begin{equation}\label{Eq_8812_StSt}
\bigl\| s (r^2 - r h r ) s - (r_0^2 - r_0 h r_0 ) \bigr\| < 4 \ep_0.
\end{equation}

By \cite[Lemma 2.8]{Ph14},
there are $a, b \in \bigl( {\overline{x A x}} \bigr)_{+}$
such that $\| a \| = \| b \| = 1$, $a b = b$,
and whenever $d \in {\overline{b A b}}$ satisfies $\| d \| \leq 1$,
then $\| x d - d \| < \min \bigl( \frac{\rh}{3}, \ep_0 \bigr)$.
Use \cite[Lemma 2.4]{Ph14}
to choose nonzero orthogonal positive elements
$z_1, z_2 \in {\overline{b A b}}$
such that $\| z_1 \| = \| z_2 \| = 1$.
For $j = 1, 2$,
set $b_j = f_{\ep_0, 2 \ep_0} (z_j)$.
Use \cite[Lemma 2.6]{Ph14}
to choose nonzero positive elements
$c_j \in {\overline{f_{2 \ep_0, 3 \ep_0} (z_j)Af_{2 \ep_0, 3 \ep_0} (z_j)}}$
for $j = 1, 2$ such that $c_1 \sim c_2$.
We may clearly assume $\| c_1 \| = \| c_2 \| = 1$.
Then
\[
b_1 b_2 = 0,
\qquad
0 \leq c_j \leq b_j \leq 1,
\qquad
a b_j = b_j,
\andeqn
b_j c_j = c_j.
\]
Define
$y = c_1$ and $\dt = 4 \ep_0 / n^2$.

Now let $h \in A_{+}$ satisfy
$\| h \| \leq 1$ and $( r^2 - r h r - \dt )_{+} \precsim y$.
Suppose $\| x h x \| \leq 1 - \ep$.
The choices of $a$ and $b$ imply that
\[
\bigl\| x (b_1 + b_2) - (b_1 + b_2) \bigr\| < \min( \ep_0, \frac{\rho}{3})
\andeqn
\bigl\| x (b_1 + b_2)^{1/2} - (b_1 + b_2)^{1/2} \bigr\| < \ep_0.
\]
Therefore
\begin{align}\label{Eq_8812_3Star}
\bigl\| h^{1/2} (b_1 + b_2) h^{1/2} \bigr\|
& = \bigl\| (b_1 + b_2)^{1/2} h (b_1 + b_2)^{1/2} \bigr\|
\\
& < \bigl\| (b_1 + b_2)^{1/2} x h x (b_1 + b_2)^{1/2} \bigr\|
    + 2 \ep_0
\notag
\\
& \leq \| x h x \| + 2 \ep_0
  \leq 1 - \ep + 2 \ep_0
  = 1 - 13 \ep_0.
\notag
\end{align}
We have
\begin{align}\label{Eq_8814_New}
\bigl\| r_0 (b_1 + b_2) - (b_1 + b_2) \bigr\|
& \leq \| r_0 x - x \|\cdot \| b_1 + b_2 \|
    + 2 \bigl\| x (b_1 + b_2) - (b_1 + b_2) \bigr\|
\\
& < \min \left( \frac{\rh}{3}, \ep_0 \right)
    + 2 \min \left( \frac{\rh}{3}, \, \ep_0 \right)
  = \min ( \rh, \, 3 \ep_0 ),
\notag
\end{align}
so
\begin{equation}\label{Eq_8812_Star}
\bigl\| r_0 (b_1 + b_2) r_0 - (b_1 + b_2) \bigr\|
  < 6 \ep_0.
\end{equation}
Now,
using $1 - \ep_0 < 1$ and \cite[Lemma 1.10]{Ph14} at the first step;
using \eqref{Eq_8812_Star}
and \cite[Corollary 1.6]{Ph14} at the second step;
and at the third step using (\ref{Eq_8814_New})
and the choice of~$\rh$ using Lemma~\ref{L_8810_Gen1Point8},
with $\af = 1 - 13 \ep_0$,
with $\bt = 5 \ep_0$,
and with $\ep_0$ in place of~$\ep$,
\begin{align*}
c_1 + c_2
& \precsim \bigl( b_1 + b_2 - (1 - \ep_0) \bigr)_{+}
\\
& \precsim \bigl( r_0 (b_1 + b_2) r_0 - (1 - 7 \ep_0) \bigr)_{+}
\\
& \precsim
   \bigl( h^{1/2} (b_1 + b_2) h^{1/2} - (1 - 13 \ep_0) \bigr)_{+}
   \oplus ( r_0^2 - r_0 h r_0 - 5 \ep_0 )_{+}.
\end{align*}
It follows from~\eqref{Eq_8812_3Star}
that
$\bigl( h^{1/2} (b_1 + b_2) h^{1/2} - (1 - 13 \ep_0) \bigr)_{+} = 0$.
So,
using \eqref{Eq_8812_StSt}
and \cite[Corollary 1.6]{Ph14} at the second step,
using Lemma~\ref{lemkey} at the third step,
and using the choice of $\dt$ and $\| s \| \leq \frac{n}{2}$
at the fourth step,
\begin{align*}
c_1 + c_2
& \precsim ( r_0^2 - r_0 h r_0 - 5 \ep_0 )_{+}
\\
& \precsim \bigl( s (r^2 - r h r ) s - \ep_0 \bigr)_{+}
\\
& \precsim \left( r^2 - r h r - \frac{\ep_0}{\| s \|^2} \right)_{+}
  \precsim ( r^2 - r h r - \dt )_{+}
  \precsim y
  = c_1.
\end{align*}
Since $A$ is simple, $c_1 c_2 = 0$, $c_1 \sim c_2$, and $r$ is finite,
this contradicts \cite[Lemma 3.8]{KR00}.
The proof is complete.
\end{proof}

\begin{proposition}\label{P_8813_TZA_Fin}
Let $A$ be a finite simple tracially $\mathcal{Z}$-absorbing C*-algebra.
Then for every $x, a \in A_{+}$ with $\| a \| = 1$,
every finite set $F \subseteq A$, every $\varepsilon > 0$,
and every $n \in \mathbb{N}$,
there is a c.p.c.~order zero map $\varphi \colon M_{n} \to A$
such that:
\begin{enumerate}
\item\label{Item_P_8813_TZA_Fin_1}
$\bigl( x^{2} - x \varphi (1) x - \varepsilon \bigr)_{+} \precsim a$.
\item\label{Item_P_8813_TZA_Fin_2}
$\| [\varphi (z), b] \| < \varepsilon$
for any $z \in M_{n}$ with $\| z \| \leq 1$ and any $b \in F$.
\item\label{Item_P_8813_TZA_Fin_Add}
$\| \ph (1) a \ph (1) \| > 1 - \ep$.
\setcounter{TmpEnumi}{\value{enumi}}
\end{enumerate}
\end{proposition}

By Lemma~\ref{L_8810_xhxTohxh},
we could equally well use the condition
$\| a \ph (1) a \| > 1 - \ep$.

\begin{proof}[Proof of Proposition~\ref{P_8813_TZA_Fin}]
Let $x, a \in A_{+}$ with $\| a \| = 1$,
let $F \subseteq A$ be finite, let $\varepsilon > 0$,
and let $n \in \mathbb{N}$.

We may assume that  $x \neq 0$.
Define
\[
\et = \min \left( \frac{1}{8}, \, \frac{\ep}{24},
       \, \frac{\ep}{8 \| x \| ( \| x \| + 1 )} \right).
\]
Choose $x_0 \in A_{+}$ such that
\[
\| x_0 \| = 1,
\qquad
\| x_0 x - x \| < \et,
\andeqn
\| x_0 a - a \| < \et.
\]
Define $x_1 = (x_0 - \et)_{+}$.
Then $x_1$ is finite by Proposition~\ref{L_8810_FinEltToUnit}.
Since $\| a \| = 1$,
we have
\begin{equation}\label{Eq_8825_Star}
\| x_1 a x_1 - a \|
 \leq 2 \| x_1 - x_0 \| + \| x_0 a - a \| + \| a x_0 - a \|
 < 4 \et.
\end{equation}
So $\| x_1 a x_1 \| > 1 - 4 \et$.
Thus $x_1 a x_1 \neq 0$,
and we can
set $a_0 = \| x_1 a x_1 \|^{-1} x_1 a x_1$.

We claim that
\begin{equation}\label{Eq_8814_a0}
\| a_0 - a \| < \frac{\ep}{2}.
\end{equation}
To prove the claim,
using $1 \geq \| x_1 a x_1 \| > 1 - 4 \et$
at the first step,
using \eqref{Eq_8825_Star} at the second step,
using $\et \leq \frac{1}{8}$ at the third step,
and using $\et \leq \frac{\ep}{20}$ at the fourth step,
\[
\| a_0 - a \|
 \leq \left( 1 - \frac{1}{1 - 4 \et} \right) + \| x_1 a x_1 - a \|
 < \frac{4 \et}{1 - 4 \et} + 4 \et
 \leq 8 \et + 4 \et
 \leq \frac{\ep}{2}.
\]
The claim is proved.

Apply Lemma~\ref{L_8810_Nonunital2Point6}
with $x_1$ in place of~$r$,
with $\et$ in place of~$\ep$,
and with $a_0$ in place of~$x$,
getting $\dt > 0$
and
$y \in \bigl( {\overline{a_0 A a_0}} \bigr)_{+} \setminus \{ 0 \}$
as there.
By \cite[Lemma 2.4]{Ph14},
there is $c \in A_{+} \setminus \{ 0 \}$
such that $c \precsim y$ and $c \precsim a$.
Set
\[
\rh = \min \left( \dt, \, \frac{\ep}{2},
       \, \frac{\ep}{2 \| x \|^2} \right).
\]
Apply the definition of tracial $\mathcal{Z}$-absorption
with $x_1$ in place of~$x$,
with $\rh$ in place of~$\ep$,
with $c$ in place of~$a$,
and with $F$ and~$n$ as given,
getting a c.p.c.~order zero map $\varphi \colon M_{n} \to A$
such that:
\begin{enumerate}
\setcounter{enumi}{\value{TmpEnumi}}
\item\label{Item_P_8813_TZA_Fin_1_Pf}
$\bigl( x_1^{2} - x_1 \varphi (1) x_1 - \rh \bigr)_{+} \precsim c$.
\item\label{Item_P_8813_TZA_Fin_2_Pf}
$\| [\varphi (z), b] \| < \rh$
for any $z \in M_{n}$ with $\| z \| \leq 1$ and any $b \in F$.
\end{enumerate}

Since $\rh \leq \ep$,
Condition~\eqref{Item_P_8813_TZA_Fin_2} in the conclusion
follows from~\eqref{Item_P_8813_TZA_Fin_2_Pf}.

We prove Condition~\eqref{Item_P_8813_TZA_Fin_1} in the conclusion.
We have $\| x_0 x - x \| < \et$ and $\| x_1 - x_0 \| \leq \et$,
so $\| x_1 x - x \| < (1 + \| x \|) \et$.
Therefore, using the choice of~$\et$ at the third step,
\begin{align*}
\bigl\| \bigl( x x_1^2 x - x x_1 \ph (1) x_1 x \bigr)
  - \bigl( x^2 - x \ph (1) x \bigr) \bigr\|
& \leq 4 \| x \|\cdot \| x_1 x - x \|
\\
& < 4 \| x \| ( \| x \| + 1 ) \et
  \leq \frac{\ep}{2}.
\end{align*}
Using this and \cite[Corollary 1.6]{Ph14} at the first step,
using $\| x \|^2 \rh \leq \tfrac{\ep}{2}$ at the second step,
using Lemma~\ref{lemkey} at the third step,
using \eqref{Item_P_8813_TZA_Fin_1_Pf} at the fourth step,
and using $c \precsim a$ at the fifth step,
we now get
\begin{align*}
\bigl( x^2 - x \ph (1) x - \ep \bigr)_{+}
& \precsim
   \bigl( x x_1^2 x - x x_1 \ph (1) x_1 x - \tfrac{\ep}{2} \bigr)_{+}
\\
& \precsim
  \bigl( x x_1^2 x - x x_1 \ph (1) x_1 x - \| x \|^2 \rh \bigr)_{+}
\\
& \precsim \bigl( x_1^2 - x_1 \ph (1) x_1 - \rh \bigr)_{+}
  \precsim c
  \precsim a.
\end{align*}
This is~\eqref{Item_P_8813_TZA_Fin_1}.

Finally,
since $c \precsim y$, $\rh \leq \dt$, and $\et \leq \frac{\ep}{2}$,
\eqref{Item_P_8813_TZA_Fin_1_Pf}
and the choices of $\dt$ and~$y$
using Lemma~\ref{L_8810_Nonunital2Point6}
imply that $\| \ph (1) a_0 \ph (1) \| > 1 - \frac{\ep}{2}$.
Combining this estimate with~(\ref{Eq_8814_a0})
gives Condition~\eqref{Item_P_8813_TZA_Fin_Add} in the conclusion.
\end{proof}

\section*{Acknowledgment}
MA was partially supported by grants from INSF (no.~91058675) and IPM (no.~98460121). NG was supported by a grant from INSF (no.~98009270). SJ would like to acknowledge the support of Tarbiat Modares University where he was a graduate student while this work was in its preliminary stages. The work of NCP was partially supported by the
  US National Science Foundation under
  Grant DMS-1501144.
  
  While preparing the final draft of this paper, a  preprint 
appeared on arXiv \cite{CLS2021}
shortly before our first arXiv version. 
Though there are minor overlap between the results of two papers
(mainly some permanence properties and a variant of 
Part~\eqref{thmx_Cu_it1} of Theorem~\ref{thmx_Cu}), 
the main results
are  different. The main result of 
\cite{CLS2021} is that every
 separable simple nuclear  
 tracially ${\mathcal{Z}}$-absorbing C*-algebra  is ${\mathcal{Z}}$-absorbing,
 which extends \cite[Theorem~4.1]{HO13} to the nonunital case
 and answers an open question of our unpublished
 version, and in a sense complements our work.
 We had a proof of this result in the special case where
 the C*-algebra is not stably projectionless (Remark~\ref{rmk_tzz}).
Though we posted the first arXiv version of the current paper in
September  2021, its ideas and results had been obtained
long time ago. For instance, Definition~\ref{defx_tz}
was quoted from the unpublished version of the current paper
   in \cite[Definition~6.6]{FG17ar},
   \cite{SJ_thesis}, and \cite[Definition~4.4]{FLL21}.
   Moreover, this definition and some results of this paper
   (including the permanenece properties and  almost unperforated
   Cuntz semigroup of tracially $\mathcal{Z}$-absorbing C*-algebras)
   were announced by the authors in several conferences,
   for instance, by the fourth  author in his
    ICM  talk in 2018 entitled ``Traciall  $\mathcal{Z}$-stability
   in the nonunital case" (ICM Operator Algebras Satellite Conference,
   http://mtm.ufsc.br/icmoa/), by the second
    author in his talk in 2020  entitled ``Group Actions on Tracially 
   $\mathcal{Z}$-absorbing C*-algebras" (The 7th Workshop on
Operator Algebras and their Applications, IPM,
http://math.ipm.ac.ir/conferences/2020/OpeAlg2020/),  
    and by the third  author in his 
   talk in 2021 in COSy entitled ``Simple Tracially $\mathcal{Z}$-absorbing C*-algebras"
   (http://www.fields.utoronto.ca/activities/20-21/COSy).


\begin{thebibliography}{99}

\bibitem{AGJP17}
M.~Amini, N.~Golestani, S.~Jamali, N. C.~Phillips,
\emph{Finite group actions on simple tracially $\mathcal{Z}$-absorbing C*-algebras},
in preparation.

\bibitem{APT11}
P.~Ara, F.~Perera, and A.~S.\  Toms,
{\emph{K-theory for operator algebras. Classification of C*-algebras}},
pages 1--71 in:
{\emph{Aspects of Operator Algebras and Applications}},
P.~Ara, F~Lled\'{o}, and F.~Perera (eds.),
Contemporary Mathematics vol.~534,
Amer.\  Math.\  Soc., Providence RI, 2011.

\bibitem{ABP16}
D.~Archey, J.~Buck, and N.~C.\  Phillips,
{\emph{Centrally large subalgebras and tracial
${\mathcal{Z}}$-absorption}},
International Mathematics Research Notices
{\textbf{292}}(2017), 1--21.

\bibitem{AP16}
D.~Archey and N.~C.\  Phillips,
{\emph{Permanence
of stable rank one for centrally large subalgebras and
crossed products by minimal homeomorphisms}},
J.\ Operator Th. {\bf 83} (2020), 353--389.

\bibitem{Avtz}
D.~Avitzour, {\emph{Free products of C*-algebras}},
Trans.\  Amer.\  Math.\  Soc.\  {\textbf{271}}(1982), 423--435.

\bibitem{Blkd1} B.~Blackadar,
{\emph{Weak expectations and nuclear C*-algebras}},
Indiana Univ.\  Math.~J.\  {\textbf{27}}(1978), 1021--1026.

\bibitem{BlEC}
B.~Blackadar,
{\emph{Operator algebras. Theory of C*-Algebras and von Neumann
  Algebras}},
Encyclopaedia of Mathematical Sciences, vol.~122:
{\emph{Operator Algebras and Non-commutative Geometry, III}},
Springer-Verlag, Berlin, 2006.

\bibitem{BlkHdm} B.~Blackadar and D.~Handelman,
{\emph{Dimension functions and traces on C*-algebras}},
J.~Funct.\  Anal.\  {\textbf{45}}(1982), 297--340.

\bibitem{BRTTW}
B.~Blackadar, L.~Robert, A.~P.\  Tikuisis, A.~S.\  Toms, and W.~Winter,
{\emph{An algebraic approach to the radius of comparison}},
Trans.\  Amer.\  Math.\  Soc.\  {\textbf{364}}(2002), 3657--3674.

\bibitem{Brnt} L.~Barnett,
{\emph{Free product von Neumann algebras of type ${\mathrm{III}}$}},
Proc.\  Amer.\  Math.\  Soc.\  {\textbf{123}}(1995), 543--553.
  
\bibitem{CE}
J.\ Castillejos and S.\ Evington, {\emph{Nuclear dimension of stably projectionless C*-algebras}}, Anal.\ PDE.\ {\bf 13}(7) (2020), 2205--2240. 

\bibitem{CETWW} J.\ Castillejos, S.\ Evington, A.\ Tikuisis, S.\ White, and W.\ Winter,  {\emph{Nuclear dimension of simple C*-algebras}}, Invent. Math. (2020), https://doi.org/10.1007/s00222-020-01013-1.

\bibitem{CETW} J.\ Castillejos, S.\ Evington, A.\ Tikuisis and S.\ White,  {\emph{Uniform Property $\Gamma$}}, to appear in International Mathematics Research Notices, https://doi.org/10.1093/imrn/rnaa282.


\bibitem{CLS2021}
J.\ Castillejos, K.\ Li, and G.\ Szab\'{o},
\emph{On tracial $\mathcal{Z}$-stability of simple 
non-unital C*-algebras},
preprint 2021 (arXiv: 2108.08742v2 [math.OA]).

\bibitem{Cu2}
J.~Cuntz,
{\emph{K-theory for certain C*-algebras}},
Ann.\  Math.\  {\textbf{113}}(1981), 181--197.

\bibitem{Dkm2} K.~J.\  Dykema,
{\emph{Purely infinite, simple C*-algebras arising from free product
constructions,~II}},
Math.\  Scand.\  {\textbf{90}}(2002), 73--86.

\bibitem{DkmRdm} K.~J.\  Dykema and M.~R{\o}rdam,
{\emph{Purely infinite simple C*-algebras arising from free product
constructions}},
Canad.\  J.\  Math.\  {\textbf{50}}(1998),  323--341.

\bibitem{EGLN17}
G.~A.\  Elliott, G.~Gong, H.~Lin, and Z.~Niu,
{\emph{Simple stably projectionless C*-algebras
  with generalized tracial rank one}},
J. Noncommut. Geom. {\bf 14} (2020), 251--347. 

\bibitem{ERS}
G.~A.\  Elliott, L.~Robert, and L.~Santiago,
{\emph{The cone of lower semicontinuous traces on a C*-algebra}},
Amer.\  J.\  Math.\  {\textbf{133}}(2011), 969--1005.

\bibitem{FG17ar}
M.~Forough and N.~Golestani,
{\emph{Tracial Rokhlin property for finite group actions on 
non-unital simple C*-algebras}},
preprint 2017 (arXiv: 1711.10818v1 [math.OA]).


\bibitem{FG17}
M.~Forough and N.~Golestani,
{\emph{The weak tracial Rokhlin property for finite group actions on simple C*-algebras}}
Doc. Math. {\bf 25}(2020), 2507--2552.


\bibitem{FLL21}
X. Fu, K. Li, and H. Lin, \emph{Tracial approximate divisibility and stable rank one},
preprint 2021 (arXiv: 2108.08970v2 [math.OA]).

\bibitem{FL21}
X. Fu  and H. Lin, \emph{Non-amenable simple C*-algebras with tracial approximation},
preprint 2021 (arXiv: 2101.07900v1 [math.OA]).

\bibitem{GH20}
E. Gardella and I. Hirshberg, \emph{Strongly outer actions of amenable groups on 
$\mathcal{Z}$-stable C*-algebras}, 
preprint 2020 (arXiv: 1811.00447v3 [math.OA]).

\bibitem{GngLin}
G.~Gong and H.~Lin,
{\emph{On classification of non-unital simple amenable
  C*-algebras, I}},
preprint 2017 (arXiv: 1611.04440v3 [math.OA]).


\bibitem{GngLinII}
G.~Gong and H.~Lin,
{\emph{On classification of non-unital simple amenable
  C*-algebras, II}},
preprint 2017 (arXiv: 1702.01073v4 [math.OA]).

\bibitem{GngLinIII}
G.~Gong and H.~Lin,
{\emph{On classification of non-unital amenable simple C*-algebras, III, 
Stably projectionless C*-algebras}},
preprint 2021 (arXiv: 2112.14003v1 [math.OA]).

 
\bibitem{HO13}
I.~Hirshberg and J.~Orovitz,
{\emph{Tracially ${\mathcal{Z}}$-absorbing C*-algebras}},
J.~Funct.\  Anal.\  {\textbf{265}}(2013), 765--785.



\bibitem{Jcln}
B.~Jacelon,
{\emph{A simple, monotracial, stably projectionless C*-algebra}},
J.~London Math.\  Soc.\  (2) {\textbf{87}}(2013), 365--383.

\bibitem{SJ_thesis}
S. Jamali, \emph{Simple tracially $\mathcal{Z}$-absorbing
C*-algebras}, Ph.D. thesis (Persian), Tarbiat Modares University, 2018.
 
\bibitem{JS}
X.~Jiang and H.~Su,
{\emph{On a simple unital projectionless C*-algebra}},
Amer.\  J.\  Math.\  {\textbf{121}}(1999), 359--413.

\bibitem{Ke17}
D.~Kerr,
{\emph{Dimension, comparison, and almost finiteness}}, to appear in
 J. Eur. Math. Soc., https://doi.org/10.4171/JEMS/995. 

\bibitem{KR00}
E.~Kirchberg and M.~R{\o}rdam,
{\emph{Non-simple purely infinite C*-algebras}},
Amer.\  J.\  Math.\  {\textbf{122}}(2000), 637--666.

\bibitem{KR02}
E.~Kirchberg and M.~R{\o}rdam,
{\emph{Infinite non-simple C*-algebras: absorbing the Cuntz algebra
${\mathcal{O}}_{\infty}$}},
Adv.\  Math.\  {\textbf{167}}(2002), 195--264.

\bibitem{KW04}
E.~Kirchberg and W.~Winter,
{\emph{Covering dimension and quasidiagonality}},
International J.\  Math.\  {\textbf{15}}(2004), 63--85.
 
\bibitem{Ln1} H.~Lin, {\emph{Tracially AF C*-algebras}},
Trans.\  Amer.\  Math.\  Soc.\  {\textbf{353}}(2001), 693--722.

\bibitem{LnBook}
H.~Lin,
{\emph{An Introduction to the Classification of Amenable
C*-Algebras}},
World Scientific, River Edge~NJ, 2001.

\bibitem{LN}
H.~Lin and Z.~Niu,
{\emph{Lifting KK-elements, asymptotic unitary equivalence and
  classification of simple C*-algebras}},
Advances in Math.\  {\textbf{219}}(2008), 1729--1769.

\bibitem{MS}
H.~Matui and Y.~Sato,
{\emph{Strict comparison and $\mathcal{Z}$-absorption of
  nuclear C*-algebras}},
Acta Math. {\textbf{209}}(2012), 179--196.

\bibitem{NiuWng}
Z.~Niu and Q.~Wang,
{\emph{A tracially AF algebra which is not $Z$-stable}},
M\"{u}nster J. of Math. {\bf 14} (2021), 41--57.

\bibitem{Na19}
N.~Nawata,
{\emph{Trace scaling automorphisms of the stabilized 
Razak-Jacelon algebra}},
Proc. Lond. Math. Soc.  (3) {\textbf{118}}(2019),
no. 3, 545--576.


\bibitem{OPW17}
J.~Orovitz, N.~C.\  Phillips, and Q.~Wang,
{\emph{Strict comparison and crossed products}},
in preparation.

\bibitem{OP06}
H.~Osaka and N.~C.\  Phillips,
{\emph{Stable and real rank for crossed products by
  automorphisms with the tracial Rokhlin property}},
Ergod.\  Th.\  Dynam.\  Sys.\  {\textbf{26}}(2006), 1579--1621.


\bibitem{Pdsn1}
G.~K.\  Pedersen,
{\emph{C*-Algebras and their Automorphism Groups}},
Academic Press, London, New York, San Francisco, 1979.

\bibitem{Ph14}
N.~C.\  Phillips,
{\emph{Large subalgebras}},
preprint (arXiv: 1408.5546v1 [math.OA]).

\bibitem{RW98}
I.~Raeburn and D.~P.\  Williams,
{\emph{Morita Equivalence and Continuous-Trace C*-Algebras}},
Mathematical Surveys and Monographs no.~60,
American Mathematical Society, Providence RI, 1998.


\bibitem{Ro91}
M.~R{\o}rdam,
{\emph{On the structure of simple C*-algebras tensored with
   a UHF-algebra}},
J.~Funct.\  Anal.\  {\textbf{100}}(1991), 1--17.

\bibitem{Ro02}
M.~R{\o}rdam,
{\emph{Classification of nuclear, simple C*-algebras}},
pages 1--145 of:
M.~R{\o}rdam and E.~St{\o}rmer,
{\emph{Classification of nuclear C*-algebras. Entropy in operator
algebras}},
Encyclopaedia of Mathematical Sciences vol.\  126,
Springer-Verlag, Berlin, 2002.

\bibitem{Ro04}
M.~R{\o}rdam,
{\emph{The stable and the real rank of ${\mathcal{Z}}$-absorbing
 C*-algebras}},
Internat.\  J.\  Math.\  {\textbf{15}}(2004), 1065--1084.

\bibitem{RrPC}
M.~R{\o}rdam,
personal communication to N.~Golestani.



\bibitem{Ti14} A.\ Tikuisis, {\emph{Nuclear dimension, $Z$-stability, and algebraic simplicity for
stably projectionless C*-algebras}}, Math.\ Ann.\ {\bf 358}(2014), 729--778. 

\bibitem{TW07}
A.~S.\  Toms and W.~Winter,
{\emph{Strongly self-absorbing C*-algebras}},
Trans.\  Amer.\  Math.\  Soc.\  {\textbf{359}}(2007), 3999--4029.


\bibitem{W1}
W.~Winter,
{\emph{Nuclear dimension and $\mathcal{Z}$-stability
 of pure C*-algebras}},
Invent.\  Math.\  {\textbf{187}}(2012), 259--342.

\bibitem{W2}
W.~Winter,
{\emph{Localizing the Elliott conjecture at strongly
 self-absorbing C*-algebras}},
J.~reine angew.\  Math.\  {\textbf{692}}(2014), 193--231.

\bibitem{WZ09}
W.~Winter and J.~Zacharias,
{\emph{Completely positive maps of order zero}},
M\"{u}nster J.\  Math.\  {\textbf{2}}(2009), 311--324.

\end{thebibliography}
\end{document}